\documentclass[amsmath, amssymb, aps, prf]{revtex4-2}
\pdfoutput=1

\usepackage[english]{babel}
\usepackage{hyperref}
\usepackage{setspace}
\usepackage{parskip}
\usepackage{xcolor}
\usepackage{lineno}
\usepackage{hyphenat}

\usepackage{graphicx}
\usepackage{algorithmicx, algorithm}
\usepackage[noend]{algpseudocode}
\usepackage{subfigure}
\usepackage{bm}
\usepackage{dcolumn}
\usepackage{caption, float}

\floatname{algorithm}{Architecture}

\newenvironment{proof}{\noindent\textbf{Proof}}{\hspace*{\fill}$\Box$\medskip}
\newtheorem{remark}{Remark}
\newtheorem{theorem}{Theorem}

\newcommand{\ip}[2]{#1 \cdot #2}
\newcommand{\pt}[1]{\dot{#1}}
\newcommand{\partialt}[1]{\partial_t {#1}}
\newcommand{\assign}{:=}

\newcommand{\tmop}[1]{\ensuremath{\operatorname{#1}}}

\newcommand{\bbR}{\mathbb{R}}

\newcommand{\bu}{\vec{u}}

\newcommand{\Ra}{\tmop{Ra}}

\renewcommand\arraystretch{1.3}

\begin{document}

\title{
  OnsagerNet: Learning Stable and Interpretable Dynamics using a Generalized Onsager Principle
}

\author{Haijun Yu}%
\email{hyu@lsec.cc.ac.cn}

\author{Xinyuan Tian}%
\email{txy@lsec.cc.ac.cn}
\affiliation{NCMIS \& LSEC, Institute of Computational Mathematics and
    Scientific/Engineering Computing, Academy of Mathematics and Systems
    Science, Chinese Academy of Sciences, Beijing 100190 and \\
School of Mathematical Sciences, University of Chinese Academy of Sciences,
Beijing 100049, China}%

\author{Weinan E}%
\email{weinan@princeton.edu}
\affiliation{Department of Mathematics and the Program in Applied and
  Computational Mathematics, Princeton University, Princeton, NJ 08544,
  USA}%

\author{Qianxiao Li}
\email{qianxiao@nus.edu.sg}
\affiliation{Department of Mathematics, National University of Singapore,
Singapore 119077}%
\affiliation{Institute of High Performance Computing, A*STAR, Singapore
138632}%


\begin{abstract}  
  We propose a systematic method for learning stable and physically
  interpretable dynamical models using sampled trajectory data
  from physical processes based on a generalized
  Onsager principle.
  The learned dynamics are autonomous ordinary differential
  equations parameterized by neural networks that retain clear
  physical structure information, such as free energy, diffusion,
  conservative motion and external forces.
  For high dimensional problems with a low dimensional slow manifold,
  an autoencoder with metric preserving regularization is introduced to find the low dimensional generalized coordinates
  on which we learn the generalized Onsager dynamics.
  Our method exhibits clear advantages over existing methods on
  benchmark problems for learning ordinary differential equations.
  We further apply this method to study Rayleigh-B\'{e}nard convection and
  learn Lorenz-like low dimensional autonomous reduced order models that capture both qualitative and quantitative
  properties of the underlying dynamics.
  This forms a general approach to building reduced order models for forced
  dissipative systems.
\end{abstract}
\maketitle

\section{Introduction}

Discovering mathematical models from observed dynamical data
is a scientific endeavor that dates back to the work of
Ptolemy, Kepler and Newton. With the recent advancements
in machine learning, data-driven approaches have become
viable alternatives to those relying purely on human insight.
The former has become an active area of research with many
proposed methodologies, which can be roughly classified
into two categories. The first is \emph{unstructured}
approaches, where the learned dynamics are not {\it a priori}
constrained to satisfy any physical principles.
Instead, minimizing reconstruction error or model complexity are dominant objectives.
For example, sparse symbolic regression methods
try to find minimal models by searching in a big dictionary of candidate analytic representations \citep{bongardAutomatedReverseEngineering2007,schmidt_distilling_2009,brunton_discovering_2016};
Numerical discretization embedding methods attempt to learn hidden
dynamics by embedding known numerical discretization schemes
{\citep{longPDENetLearningPDEs2018,raissi_multistep_2018,xieNonintrusiveInferenceReduced2018}};
Galerkin-closure methods use neural networks to reduce the projection error of traditional Galerkin methods {\citep{wanDataassistedReducedorderModeling2018,panDataDrivenDiscoveryClosure2018,wangRecurrentNeuralNetwork2020,maModelReductionMemory2019,xie_modeling_2019,wangRecurrentNeuralNetwork2020}}.
Another line of unstructured methods directly apply time-series modelling tools
such as LSTM/GRU and reservoir
computing~\cite{pathakModelFreePredictionLarge2018a,vlachasBackpropagationAlgorithmsReservoir2020,szunyoghMachineLearningBasedGlobalAtmospheric2020}
to model temporal evolution of physical processes.
For example, it is demonstrated in~\cite{pathakModelFreePredictionLarge2018a} that reservoir computing
can effectively capture the dynamics of high-dimensional chaotic systems, such as the
Kuramoto-Sivashinsky equation.
These unstructured methods can often achieve high predictive accuracy,
but the learned dynamics may not have clear physical structure and theoretical guarantee of temporal stability.
This in turn limit their ability to capture qualitative properties of the underlying model on large time intervals.
Such issues are significant in that often, the goal of building reduced models is
not to reproduce the original dynamics in its entirety, but rather to extract
some salient and physically relevant insights. This leads to the second class
of \emph{structured} approaches, where one imparts from the outset some
physically motivated constraints to the dynamical system to be learned, e.g.
gradient systems {\citep{kolterLearningStableDeep2019,gieslLyapunov2020}},
Hamiltonian systems
\citep{zhongSymplecticODENetLearning2020,jin_symplectic_2020,zhongDissipativeSymODENEncoding2020},
 etc.
These methods have the advantage that the learned models have pre-determined structures and stability may be automatically ensured.
However, works in this direction may have the limitation that the structures imposed may be
too restrictive to handle complex problems beyond benchmark examples.
We demonstrate this point in our benchmark examples to be presented later.
It is worth noting that recent works address this issue by a combination of unstructured model
fitting with information from an \textit{a priori} known, but imperfect model, e.g.~\cite{wiknerCombiningMachineLearning2020}.

In this paper, we propose a systematic method that overcomes the aforementioned limitations.
The key idea is a neural network parameterization of a reduced dynamical model, which we  call \emph{OnsagerNet}, based on a highly general extension of the Onsager principle for dissipative dynamics \citep{onsagerReciprocalRelationsIrreversible1931,onsagerReciprocalRelationsIrreversible1931a}. We choose the Onsager principle, which builds on the second law of thermodynamics and microscopic reversibility, as the starting point due to its simplicity, generality and physical motivation. It naturally gives rise to stable structures for dynamical systems in generalized coordinates, and has been used extensively to derive mathematically well-posed models for complex systems in fluid mechanics, material science, biological science, etc. (see e.g. \citep{qian_variational_2006, doi_onsager_2011,yang_hydrodynamic_2016,giga_variational_2017,jiang_application_2019, doi_application_2019a, xu_generalized_2019}).
However, there are two challenges in using it to find dynamical models with high-fidelity: \begin{enumerate}
    \item How to choose the generalized coordinates;
    \item How to determine the large amount of free coefficients in the structured dynamical system. Some in-depth discussions could be found in \citet[Chap 1]{beris_thermodynamics_1994}.
\end{enumerate}
One typical example is the modeling of liquid crystal
systems \citep{DoiEdwards1986,deGennes1993}. If the underlying system is near equilibrium,
it is possible to determine the coefficients of a reduced-order (macroscopic) model
from mathematical analysis or quasi-equilibrium approximations of the underlying microscopic model  (see e.g.
 \cite{kroger_derivation_2007,yu_nonhomogeneous_2010,e_principles_2011,han_microscopic_2015}). Otherwise,
this task becomes significantly harder, yet a lot of interesting phenomena happen in this regime. Therefore, a method to derive high fidelity macroscopic models that operate far from equilibrium without \emph{a priori} scale information is much sought after.

We tackle these two tasks in a data-driven manner.
{For the first task we learn an approximately isometric embedding to find generalized coordinates. In the linear case this reduces to the principal component analysis (PCA), and in general we modify the autoencoder (AE) \citep{hintonReducingDimensionalityData2006,rifaiContractiveAutoEncodersExplicit2011} with metric preserving regularization designed for trajectory data.}
For the second task, we avoid the explicit fitting of the prohibitively large number of response coefficients in linearization approaches. Instead, we parameterize nonlinear relationships between generalized coordinates and the physical quantities governing the dynamics (e.g. free energy, dissipation matrices) as neural networks and train them on observed dynamical data with an embedded Runge-Kutta method.
The overall approach of OnsagerNet gives rise to stable and interpretable dynamics, yet retaining a degree of generality in the structure to potentially capture complex interactions.
By interpretability, we mean that the structure of OnsagerNet parameterization has physical origins,
and its learned components can thus be used for subsequent analysis (e.g. visualizing energy landscapes, as shown in Sec.~\ref{sec:applications}).
Moreover, we show that the network architecture that we used to parameterize the hidden dynamics has a more general hypothesis space than other existing structured approaches, and can provably represent many dynamical systems with physically meaningful origins (e.g. systems described by generalized Poisson brackets). At the same time, the stability of the OnsagerNet is ensured by its internal structure motivated by physical considerations, and this does not reduce its approximation capacity when applied to dynamical data arising from dissipative processes.
We demonstrate the last point by applying the method to the well-known Rayleigh-B\'{e}nard convection (RBC) problem, based on which Lorenz discovered a minimal 3-mode ordinary differential equation (ODE) system that exhibits
chaos \citep{lorenz_deterministic_1963}.
Lorenz used three dominant modes obtained from linear stability analysis as reduced coordinates and derived the governing equations by a Galerkin projection. As a severely truncated low-dimensional model, the Lorenz model has limited quantitative accuracy when the system is far from the linear stability region.
In fact, it is shown by \cite{curryOrderDisorderTwo1984} that
one needs to use more than 100 modes in a Fourier-Galerkin method
to get quantitative accuracy for the 2-dimensional RBC problem that is far from the linear region.
In this paper, we show that OnsagerNet is able to learn low dimensional
Lorenz-like autonomous ODE models with few modes, yet having high quantitative fidelity to the underlying RBC problem.
This validates, and improves upon, the basic approach of Lorenz to capture
complex flow phenomena using low dimensional nonlinear dynamics.
Furthermore, this demonstrates the effectiveness of a principled combination of physics and machine learning when dealing with data from scientific applications.

\section{From Onsager principle to OnsagerNet}

The Onsager principle \citep{onsagerReciprocalRelationsIrreversible1931,onsagerReciprocalRelationsIrreversible1931a} is a well-known model for the dynamics of dissipative systems near equilibrium. Given generalized coordinates $h = (h_1,\dots,h_m)$, their
dynamical evolution is modelled by
\begin{equation}\label{eq:onsager_var}
M \pt{h} = - \nabla V(h)
\end{equation}
where $V:\mathbb{R}^m \rightarrow \mathbb{R}$ is a potential function, typically with a thermodynamic interpretation such as free energy or negated entropy. The matrix $M$ models the energy dissipation of the system (or entropy production) and is positive semi-definite in the sense that $\ip{h}{M h} \ge 0$ for all $h$, owing to the second law of thermodynamics. An important result due to Onsager, known as the \emph{reciprocal relations}, shows that if the system possesses microscopic time reversibility, then $M$ is symmetric, i.e. $M_{ij}=M_{ji}$.

The dynamics \eqref{eq:onsager_var} is simple, yet it can model a large variety of systems close to equilibrium in the linear response regime (see e.g. \citep{degroot_nonequilibrium_1962,doi_onsager_2011}).  However, many dynamical systems in practice are far from equilibrium, possibly sustained by external forces (e.g. the flow induced transition in liquid crystals \citep{kroger_derivation_2007,yu_nonhomogeneous_2010}, the kinetic equation with large Knudsen number \cite{han_uniformly_2019}, etc), and in such cases~\eqref{eq:onsager_var} no longer suffices. Thus, it is an important task to generalize or extend~\eqref{eq:onsager_var} to handle these systems.

The extension of known but limited physical principles to more general domains have been a continual exercise in the development of theoretical physics. In fact,~\eqref{eq:onsager_var} itself can be viewed as a generalization of the Rayleigh's phenomenological principle of least dissipation for hydrodynamics \citep{rayleigh_instability_1878}.
While classical non-equilibrium statistical physics has been further developed in many aspects, including the study of transport phenomena via Green-Kubo relations~\citep{green_markoff_1954,kubo_statistical-mechanical_1957-1} and the relations between fluctuations and entropy production~\citep{evans_fluctuation_2002}, an extension of the dynamical model~\eqref{eq:onsager_var} to model general non-equilibrium systems remains elusive. Furthermore, whether a simple yet general extension with solid physical background exists at all is questionable.

This paper takes a rather different approach. Instead of the possibly arduous tasks of developing a general dynamical theory of non-equilibrium systems from first principles, we seek a \emph{data-driven} extension of~\eqref{eq:onsager_var}. In other words, given a dynamical process to be modelled, we posit that it satisfies some reasonable extension of the usual Onsager principle, and determine its precise form from data. In particular, we define the \emph{generalized Onsager principle}
\begin{equation}
\big(M (h) + W (h)\big) {\pt{h}} = - \nabla V (h) + g (h). \label{eq:GOPmform}
\end{equation}
where the matrix valued functions $M$ (resp. $W$) maps $h$ to symmetric positive semi-definite (resp. anti-symmetric) matrices. The last term $g:\mathbb{R}^{m} \rightarrow \mathbb{R}^{m}$ is a vector field that represents external forces to the otherwise closed system, and may interact with the system in a nonlinear way. The anti-symmetric term $W$ models the conservative part of the system, and together with $g$ they greatly extend the degree of applicability of the classical Onsager's principle.

We will assume $M (h) + W (h)$ is invertible everywhere, and hence $(M (h) + W
(h))^{- 1}$ can be written
as a sum of a symmetric positive semi-definite matrix, denoted by $\tilde{M}
(h)$ and
a skew-symmetric matrix, denoted by $\tilde{W} (h)$ (see Theorem
\ref{thm:psdmatrix} in Appendix \ref{apd:psdmatrix} for a proof), thus we have
an
equivalent form for equation \eqref{eq:GOPmform}:
\begin{equation}
\pt{h} = - (\tilde{M} (h) + \tilde{W} (h)) \nabla V (h) + {f} (h),
\label{eq:GOPdform}
\end{equation}
where $f = (M+W)^{-1} g$. We will now work with~\eqref{eq:GOPdform} as it is
more convenient.

We remark that the form of the generalized Onsager dynamics~\eqref{eq:GOPmform}
is not an arbitrary extension of~\eqref{eq:onsager_var}. In fact, the
dissipative-conservative decomposition ($M+W$) and dependence on $h$ are well
motivated from classical dynamics. To arrive at the former, we make the crucial
observation that a dynamical system defined by generalized Poisson brackets
\citep{beris_thermodynamics_1994} have precisely this decomposition  (See
Appendix \ref{apd:gdyn}).
Moreover, such forms also appeared in partial differential equation (PDE)
models for complex fluids {\citep{yang_hydrodynamic_2016,zhao_general_2018}}.
Note that general Poisson brackets are required to satisfy the Jacobi identity
(see e.g. \citep{beris_thermodynamics_1994,ottinger_equilibrium_2005}), but in
our approach we do not enforce such a condition. We refer to
\cite{hernandez_structurepreserving_2021} for a neural network implementation
based on the generalized Poisson brackets.

\subsection{Generalization of Onsager principle for model
    reduction}
\label{apd:GOPasMR}
Here, we show that the
generalized Onsager principle \eqref{eq:GOPmform} or its equivalent form
\eqref{eq:GOPdform}
is invariant under coordinate transformations and is a suitable structured form
for model reduction.
In the original Onsager
principle \cite{onsagerReciprocalRelationsIrreversible1931,
    onsagerReciprocalRelationsIrreversible1931a}, the system is
assumed to be not far away from the equilibrium, such that the system
dissipation takes a quadratic form: $\| \pt{h} \|_M^2$, i.e. the
matrix $M$ is assumed to be constant. In our generalization, we
assume $M$ is a general function that depends $h$. Here we give some
arguments why this is necessary and how it can be obtained if
the underlying dynamic is given from the viewpoint of
model reduction.

To explain nonlinear extension is necessary, we assume that the
underlying high-dimensional dynamics which produces the observed trajectories
satisfies the form of the generalized Onsager principle with constant $M$ and
$W$.
The dynamics described by an underlying PDE after spatial
semi-discretization takes the form
\begin{equation}
\left(M + W\right) \pt{u} = - \nabla_u V + g, \quad u \in \mathbb{R}^N
\label{eq:PDEsemdis},
\end{equation}
where $M$ is a symmetric positive semi-definite \emph{constant} matrix, $W$
is
an anti-symmetric matrix. \ For the dynamics to make sense, we need $M + W$ to
be invertible. $\nabla_u V$ is a column vector of length $N$. By taking
inner product of {\eqref{eq:PDEsemdis}} with $\pt{u} $, we have an
energy dissipation relation of form
\[ \pt{V} = - \ip{\pt{u}} {M \pt{u}}  + \ip{g}{\pt{u}}. \]
Now, suppose that $u$ has a solution in a low-dimensional manifold that could
be well described by hidden variables $h (t) \in \mathbb{R}^m$ with $m \ll
N$.
Denote the low-dimensional solution as $u = u (h (t)) + \varepsilon$, and
plug
it into {\eqref{eq:PDEsemdis}}, we obtain
\[ (M + W) \nabla_h u \,\pt{h} = - \nabla_u V (u (h)) + g + O(\varepsilon) .
\]
Multiplying $J^T:=(\nabla_h u)^T$ on both sides of the above equation, we
have
\[ J^T (M + W) J\,\pt{h} = J^T [-\nabla_u V
(u (h)) + g] + O (\varepsilon) . \]
So, we obtain the following
ODE
system with model reduction error $O (\varepsilon)$:
\begin{equation}
(\bar{M} + \bar{W}) \pt{h} = - \nabla_h V + \bar{g}, \label{eq:ModelRed}
\end{equation}
where
\begin{equation}
\bar{M} = J^T M J, \quad \bar{W} = J^T W J, \quad \bar{g} = J^T g.
\label{eq:ModelRedCoefMats}
\end{equation}
Note that as long as $\nabla_h u$ has a full column rank,
$(\bar{M} + \bar{W})$ is invertible, so the ODE system makes sense. Now
$\bar{M}$, $\bar{W}$ and $\bar{g}$ depend on $h$ in general if
$\nabla_h u$ is not a constant matrix.
Moreover, if the solutions considered all exist
exactly in a low-dimensional manifold, i.e. $\varepsilon=0$ in the above
derivation,
then \eqref{eq:ModelRed} is exact, which means the generalized Onsager
principle is  invariant to non-singular coordinate transformations.

\begin{remark}
    For underlying dynamics given in the alternative form
    \begin{equation}
    \pt{u}  = - (M + W) \nabla_u V + f, \quad u \in \mathbb{R}^N
    \label{eq:PDEsemdis2} .
    \end{equation}
    We first rewrite it into form {\eqref{eq:PDEsemdis}} as
    \[ (M + W)^{- 1} \pt{u} = - \nabla_u V + (M + W)^{- 1} f. \] Then,
    use same procedure as before to obtain
    \[ J^T (M + W)^{- 1} J \, \pt{h} =
    J^T [-\nabla_u V (u (h)) + (M + W)^{- 1} f]
    + O (\varepsilon), \]
    from which, we obtains a reduced model
    with error $O (\varepsilon)$:
    \begin{equation}
    \pt{h} = - (\tilde{M} + \tilde{W}) \nabla_h V + \tilde{f},
    \label{eq:ModelRed2}
    \end{equation}
    where $\tilde{M} = (G + G^T) / 2$, $\tilde{W} = (G - G^T) / 2$,
    $\tilde{f} = G J^T (M + W)^{- 1} f$,
    and $G = [J^T (M + W)^{- 1} J]^{- 1}$.
\end{remark}

\begin{remark}
    When $M, W, g$ in {\eqref{eq:PDEsemdis}} are constant, if
    linear PCA is used for model reduction, in which $\nabla_h u$ is a
    constant matrix, then we have $\bar{M}, \bar{W}, \bar{g}$ are
    constants. However, if we consider the incompressible Navier-Stokes
    equations written
    in form {\eqref{eq:PDEsemdis}}, then $M$ and $W$ are not constant
    matrices.
    We obtain nonlinear coefficients in both formulations for the model
    reduction problem of incompressible Naiver-Stokes equations.
\end{remark}

Note that the presence of state ($h$) dependence on all the terms implies that we are not linearizing the system about some equilibrium state, as is usually done in linear response theory. Consequently, the functions $W,M,g$ and $V$ may be complicated functions of $h$, and we will learn them from data by parameterizing them as suitably designed neural networks. In this way, we preserve the physical intuition of non-equilibrium physics (dissipation term $M$, conservative term $W$, potential term $V$ and external fields $g$), yet exploit the flexibility of function approximation using data and learning.

In summary, one can view~\eqref{eq:GOPmform} and \eqref{eq:GOPdform} both as an extension of the classical Onsager  principle to handle systems far from equilibrium, or as a reduction of a high dimensional dynamical system defined by generalized Poisson brackets. Both of these dynamics are highly general in their respective domains of application, and serve as solid foundations on which we build our data-driven methodology.

\subsection{Dissipative structure and temporal stability of OnsagerNet}

From the mathematical point of view, modeling dynamics
using~\eqref{eq:GOPmform} or~\eqref{eq:GOPdform} also has clear advantages, in
that the learned system is well-behaved as it evolves through time, unlike
unstructured approaches such as dynamic mode decomposition (DMD) (see e.g.
\cite{schmid_dynamic_2010, takeishi_learning_2017}) and direct parameterization
by neural networks, which may achieve a short time trajectory-wise accuracy,
but cannot assure mid-to-long time stability as the learned system evolves in
time. In our case, we can show in the following result that under mild
conditions, the dynamics described by~\eqref{eq:GOPmform}
or~\eqref{eq:GOPdform} automatically ensures a dissipative structure, and
remains stable as the system evolves in time.

\begin{theorem}\label{thm:noblowup}
    The solutions to the system {\eqref{eq:GOPdform}} satisfy an energy evolution law
	\begin{equation}
	\pt{V}(h)  = - \big\| \nabla V(h) \big\|_{\tilde{M}}^2
	 + \ip{f(h)}{\nabla V(h)}. \label{eq:GOPdissipation2}
	\end{equation}
	If we assume further that there exist positive constants $\alpha, \beta$ and non-negative constants $c_0, c_1$ such that $\ip{h}{\tilde{M} h} \geqslant \alpha \|h\|^2$ (uniformly positive dissipation), $V (h) \geqslant \beta \| h \|^2$ (coercive potential) and
	$\|f(h) \|\leqslant c_0+ c_1 \|h\|$ (external force has linear growth). Then, $\|h(t)\|,V(h(t)) < \infty$ for all $t$. In particular, if there is no external force, then $V(h(t))$ is non-increasing in $t$.
\end{theorem}

  \begin{proof}
  	Equation \eqref{eq:GOPdissipation2} can be obtained by pairing
  	both sides of equation \eqref{eq:GOPdform} with $\nabla V$.
  	We now prove boundedness.
    Using Young's inequality, we have for any $\varepsilon$,
	\begin{equation*}
	\begin{aligned}
	\ip{f(h)}{\nabla V(h)} & \le {\varepsilon}{\| \nabla V(h) \|^2}
	+ \frac{1}{4\varepsilon}\| f(h) \|^2 \\
	&\le {\varepsilon}{\| \nabla V(h) \|^2}
	+ \frac{1}{2\varepsilon}\big(c_0^2 + c_1^2 \|h\|^2\big)
	\end{aligned}
	\end{equation*}
	Putting the above estimate with $\varepsilon=\alpha$ back to \eqref{eq:GOPdissipation2}, to get
	\begin{eqnarray*}
    \pt{V}(h)  & =  & - \| \nabla V(h) \|_{\tilde{M}}^2 + \alpha \| \nabla V(h) \|^2
    + \frac{1}{2 \alpha} (c_0^2 + c_1^2\| h \|^2) \\
		& \leqslant & \frac{c_1^2}{2 \alpha \beta} V (h) + \frac{c_0^2}{2 \alpha}
	\end{eqnarray*}
	By Gr\"{o}nwall inequality, we obtain
	\[
	    V (h) \le
	    \begin{cases}
	    e^{\frac{c_1^2}{2 \alpha \beta} t} V_0
	    + (e^{\frac{c_1^2}{2 \alpha \beta} t} -1 ) \frac{c_0 ^2\beta }{c_1^2}, & c_1 > 0, \\
	    V_0 + \frac{c_0 ^2 }{2 \alpha} t,  & c_1 = 0,
	    \end{cases}
	\]
	where $V_0 = V(h(0))$.
	Finally, by the assumption $\| h \|^2 \leqslant \frac{1}{\beta} V$, $h$ is
	bounded. Note that when $f(h)\equiv 0$, we have $c_0=c_1=0$, so the above inequality is reduced to $V(h) \le V_0$, which can be obtained directly from \eqref{eq:GOPdissipation2}
	 without the requirement of $\alpha>0, \beta>0$.
\end{proof}

Theorem~\ref{thm:noblowup} shows that the dynamics is physical under the assumed conditions and we will design our neural network parameterization so that these conditions are satisfied.

\section{The OnsagerNet architecture and learning algorithm}

In this section, we introduce the detailed network architecture for the parameterization of the generalized Onsager dynamics, and discuss the details of the training algorithm.

\subsection{Network architecture}

We implement the generalized Onsager principle based on equation
{\eqref{eq:GOPdform}}. \ $\tilde{M} (h)$, $\tilde{W} (h)$, $V (h)$ and
$f(h)$ are represented by neural networks with shared hidden layers and are combined
according to {\eqref{eq:GOPdform}}. The resulting composite
network is named \emph{OnsagerNet}. Accounting for (anti-)symmetry, the numbers of independent variables in
$\tilde{M} (h)$, $\tilde{W} (h), V (h), f (h)$ are
$(m + 1) m / 2, (m - 1) m / 2$, $1$, and $m$, respectively.

One important fact is that $V (h)$, as an energy function, should be lower
bounded.  To ensure this automatically, one may define
$V(h) = \frac{1}{2} U(h)^2 + C$, where $C$ is some constant that is smaller than or equal to
$V$'s lower bound. Since the constant $C$ does not affect the
dynamics, we drop it in numerical implementation.  The actual form we
take is
\begin{equation}
  V (h) = \frac{1}{2} \sum_{i = 1}^m \Big(U_i(h)+ \sum_{j=1}^m\gamma_{ij} h_j \Big)^2 + \beta \| h \|^2,
  \label{eq:PotentialNet}
\end{equation}
where $\beta \geqslant 0$ is a positive hyper-parameter as assumed in Theorem
\ref{thm:noblowup} for forced systems. $U_i$ has a similar structure as one component of
$\tilde{W}(h)$.  We use $m$ terms in the form \eqref{eq:PotentialNet}
to ensure that an original quadratic energy after a dimension
reduction can be handled easily.
To see this, suppose the potential function for the high-dimensional problem (with coordinates $u$) is defined as $V(u)=u^T {A u}$ with $A$ symmetric positive definite. Further, let a linear PCA  $u \approx u_0 + J h$ be used for dimensionality reduction.
Then, $V(h)\approx (u_0+J h)^T A(u_0 + J h)=
\|u_0\|^2 + 2 u_0^TAJ h + h^T J^TAJ h =
\| v_0+G  h \|^2+ \text{constant}$, where $G^T G = J^TAJ, v_0^T = u_0^T AJG^{-1} $.
Hence, with $(\gamma_{ij})$ representing the matrix $G$,
a constant $U_i$ suffices to fully represent quadratic potentials.
The autograd mechanism implemented in
PyTorch {\citep{paszke_pytorch_2019}} is used to calculate $\nabla V(h)$.

  To ensure the positive semi-definite property of $\tilde{M} (h)$, we
  let $\tilde{M} (h) = L (h) L (h)^T + \alpha I$, where $L (h)$
  is a lower triangular matrix, $I$ is the identity matrix,
  $\alpha \geqslant 0$. Note that the degree of freedom of $L (h)$ and
  $\tilde{W} (h)$ can be combined into one $m \times m$ matrix,
  whose upper triangular part $R(h)$ determines
  $\tilde{W} (h) = R(h) - R(h)^T$.
  A standard multi-layer perception neural
  network with residual network structure (ResNet) \citep{he2016deep} is used to generate adaptive bases, which takes $(h_1, \ldots, h_m)$ as input, and outputs
  $\{ L (h), R (h), U_i(h) \}$ as linear combinations of those bases.
  The external force ${f}_i (h)$ are parameterized based on \emph{a priori} information, and should be of limited capacity so as not to dilute
  the physical structures imposed by the other terms.
  For forced systems considered in this paper,
  we typically take $f_i$ as affine functions of $h$.
  The final output
  of the OnsagerNet is given by
\begin{equation}
  \pt{h_i} = \sum_{k=1}^m \left(L (h)^{} L (h)^T_{} + \alpha I + \tilde{W}
    (h)\right)_{i, k} \big(- \partial_{h_k} V (h)\big) + {f}_i (h),
    \quad i=1,\ldots, m, \label{eq:OnsagerNet}
\end{equation}
where $V (h)$ is defined by {\eqref{eq:PotentialNet}}. Note that in an unforced system we have
$\alpha = \beta = 0$, {since they are only
  introduced in forced system to ensure a stability of the learned system}
as required by Theorem \ref{thm:noblowup}. The computation procedure of OnsagerNet is described in Architecture
\ref{alg:GOPNet}. The full architecture is summarized in
Fig.~\ref{fig:OnsagerNet}.

\begin{algorithm}[H]
	\caption{OnsagerNet$(\alpha, \beta, l, n_l; h)$}
	\label{alg:GOPNet}
	\hspace*{0.02in} {\bf Input:}
    $h\in \bbR^m$, parameters $\alpha\ge 0, \beta \ge 0$, activation function $\sigma_A$,
    number of hidden layer $l$ and number of nodes in each hidden layer: $n_l$\\
	\hspace*{0.02in} {\bf Output:}
    $\operatorname{OnsagerNet}(\alpha, \beta, l, n_l; h) \in \bbR^m$
	\begin{algorithmic}[1]
      \State Calculate the shared sub-net output using a $l$-layer
      neural network $\phi = \operatorname{MLP}(h)\in \bbR^{n_l}$. Here
      $\operatorname{MLP}$ has $l-1$ hidden layer and one output layer,
      each layer has $n_l$ nodes. Activation function $\sigma_A$ is
      applied for all hidden layers and output layer.  If $l>1$, ResNet
      shortcuts are used.

      \State Evaluate $U_i$ using a linear layer as
      $U_i = \sum_{j} \omega^{(1)}_{ij} \phi_j + b^{(1)}_{i}$, calculate $V$
      according to \eqref{eq:PotentialNet}

      \State Use the autograd mechanism of $\operatorname{PyTorch}$ to
      calculate the gradient $\nabla_{h_k} V(h),\ k=1,\ldots, m.$

      \State Evaluate $A\in \bbR^{m^2}$ as
      $A_i = \sum_{j} \omega^{(2)}_{ij} \phi_j + b^{(2)}_{i}.$

      \State Reshape $A$ as a $m\times m$ matrix, take its
      lower-triangular part including the main diagonal as $L$, the
      upper-triangular part without the main diagonal as $R$ to form
      $\tilde{W}=R-R^T$.

      \If {the system is forced}
      \State calculate the external force/control ${f}_i$ using a priori form.
      \Else
      \State take ${f}_i=0.$
      \EndIf

      \State Calculate the output of OnsagerNet using \eqref{eq:OnsagerNet}
      \State \Return $\operatorname{OnsagerNet}(\alpha, \beta, l, n_l; h)$.
	\end{algorithmic}
\end{algorithm}
\begin{figure}[ht]
    \centering
    \includegraphics[width=0.4\linewidth, ]{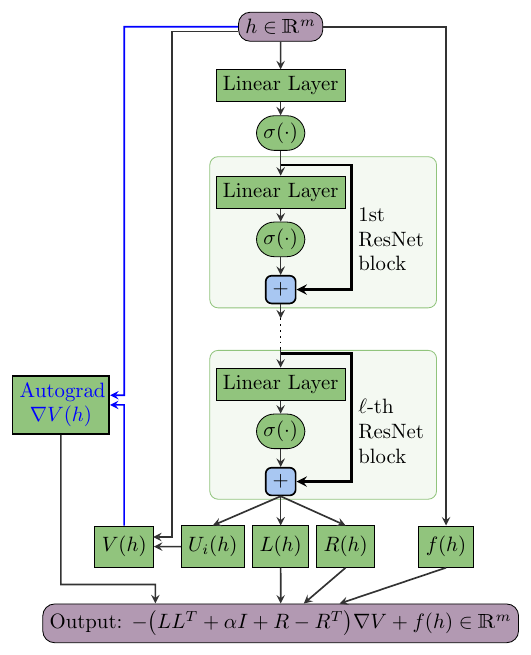}
    \includegraphics[width=0.4\linewidth, ]{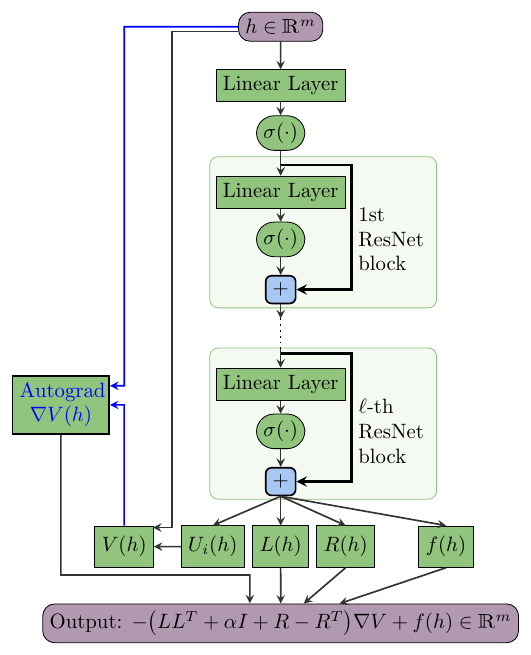}
    \caption{\label{fig:OnsagerNet}  The computational graph of OnsagerNet.
    Left: affine $f(h)$, Right: nonlinear $f(h)$.}
\end{figure}

\subsection{Training objective and algorithm}
To learn ODE model represented by OnsagerNet based on sampled ODE trajectory data, we minimize following loss function
\begin{equation} \label{eq:ode_loss}
\mathcal{L}_{\tmop{ODE}} =
\frac{1}{| S |} \sum_{ (h(t), h(t+\tau)) \in S}
\frac{1}{\tau^2}
\big\| h (t + \tau) - \tmop{RK}2 (\tmop{OnsagerNet}; h(t), \tau/n_s, n_s)
\big\|^2.
\end{equation}
Here
$\tau$ is the time interval of sample data.
${S}$ is the sample set.
$\tmop{RK2}$ stands for a second order Runge-Kutta method (Heun method).
$n_s$ is the number of $\tmop{RK2}$ steps used to forward
the solution of OnsagerNet from snapshots at $t$ to $t+\tau$.
Other Runge-Kutta method can be used. For simplicity, we only present results using
Heun method in this paper.
This Runge-Kutta embedding method has several advantages over the linear multi-step methods
\citep{raissi_multistep_2018,xieNonintrusiveInferenceReduced2018}, e.g. the variable time interval case and long time interval case can be easily handled by Runge-Kutta methods.

With the definition of the loss function and model architecture, we can then use standard stochastic gradient algorithms to train the network. Here we will use the Adam optimizer \citep{kingmaAdamMethodStochastic2015,reddi_convergence_2018} to minimize
the loss function with a learning rate scheduler that halves the learning rate if the
loss is not decreasing in certain numbers of iterations.

\subsection{Reduced order model via embedding}

The previous section describes the situation when no dimensionality reduction is sought or required, and an Onsager dynamics is learned directly on the original trajectory coordinates. On the other hand, if the data are generated from a numerical discretization of some PDE or a large ODE system, and we
want to learn a small reduced order model, then a dimensionality reduction
procedure is needed. One can either use linear principal component analysis (PCA)
or nonlinear embedding, e.g. the autoencoder, to find a set of good latent coordinates from the high dimensional data.
When PCA is used, we perform PCA and OnsagerNet training separately.
When an autoencoder is used, we can train it either separately or together
with the OnsagerNet.
The end-to-end training loss function is taken as
\begin{align}
\mathcal{L}_{\tmop{tot}} &= \mathcal{L}_{\tmop{AE}} + \mathcal{L}_{\tmop{ODE}} +
\mathcal{L}_{\tmop{reg}} \nonumber \\
&=
\frac{1}{| S |} \sum_{ (u(t), u(t+\tau)) \in S}
\Big\{
\beta_{ae}\| u(t) - \psi\circ \varphi \circ u(t) \|^2
+ \beta_{ae}\| u(t+\tau) - \psi\circ \varphi \circ u(t+\tau) \|^2
\nonumber  \\
& \qquad\qquad
+ \frac{1}{\tau^2}
\big\| \varphi \circ u (t + \tau) - \tmop{RK}2 (\tmop{OnsagerNet};
\varphi\circ u(t), \tau/n_s, n_s) \big\|^2 \nonumber  \\
& \qquad\qquad
+ \beta_{isom}
\Big(
\big| \| u(t+\tau) - u(t) \|^2
- \| \varphi \circ u(t+\tau) - \varphi\circ u(t) \|^2
 \big| - \alpha_{isom} \Big)_+
\Big\}, \label{eq:e2e_loss}
\end{align}
where $\varphi$, $\psi$ stands for the
encoder function and decoder function of the autoencoder, respectively. In the last term,
$\left| \| u(t+\tau) - u(t) \|^2
- \| \varphi \circ u(t+\tau) - \varphi\circ u(t) \|^2 \right|$
is an estimate of the isometric loss (i.e. deviation from $\varphi$ being an isometry) of the encoder function based on trajectory data, with
$\alpha_{isom}$ being a constant smaller than the
average isometric loss of the PCA dimension reduction.
Here, $(\cdot)_+$ stands for positive part and $\beta_{isom}$ is a penalty constant.
$\beta_{ae}$ is a parameter to balance the autoencoder accuracy and OnsagerNet fitting accuracy.

{
The choice of autoencoder architecture follows from our observation that PCA
performed respectably on a number of examples we studied.
Thus, we build the autoencoder by extending the basic PCA.
The resulting architecture, which we call PCA-ResNet,
is a stacked architecture with each layer
consisting of a fully connected autoencoder block with
a PCA-type short-cut connection:
\begin{equation}
    h^{k+1} = \text{PCA}_{n_k, n_{k+1}} (h^k) + W_2^k \sigma (W^k_1 h^k +
    b^k),
    \quad k=0,\ldots, L-1,
\end{equation}
where $n_{k+1}<n_{k}$ and $\text{PCA}_{n_k, n_{k+1}}$ is a PCA transform
from $n_k$ dimension to $n_{k+1}$ dimension. $\sigma$ is a smooth activation
function. The parameters $W_2^k$, $W_1^k$
$b^k$ in the encoder are initialized close to zero, such that the encoder becomes
a small perturbation of the PCA. On can regard such an autoencoder as a nonlinear
extension of PCA. The decoder is designed similarly.
We note that there was a similar but different autoencoder proposed to find
slow variables \cite{linot_deep_2020}.

}

\section{Applications}
\label{sec:applications}

In this section, we present various applications of the
OnsagerNet approach. We will start with benchmark problems
and then proceed to investigate more challenging settings
such as the Rayleigh-B\'{e}nard convection (RBC) problem.

\subsection{Benchmark problem 1: deterministic damped Langevin equation}

Here, we use the OnsagerNet to learn a deterministic damped
Langevin equation
\begin{equation}
  \pt{x} = v, \qquad \pt{v} = - \frac{\gamma}{m} \pt{x} - \frac{1}{m} \nabla_x U (x).
  \label{eq:Langevin}
\end{equation}
The friction coefficient $\gamma$ may be a constant or a
parameter that depends on $v$. For the potential $U (x)$, we
consider two cases:
\begin{itemize}
\item Hookean spring
  \begin{equation}
    U (x) = \frac{\kappa}{2} x^2 . \label{eq:Hookean}
  \end{equation}
  \item Pendulum model
  \begin{equation}
    U (x) = \frac{4 \kappa}{\pi^2} \left( 1 - \cos \left( \frac{\pi x}{2} \right)
    \right) . \label{eq:pendulum}
  \end{equation}
\end{itemize}
Note that no dimensionality reduction is required here and the
coordinate $h=(x, v)$ entails the full phase space. The goal
of this toy example is to quantify the accuracy and
stability of the OnsagerNet, as well as to highlight the
interpretability of the learned dynamics.

We normalize the parameter $\gamma, \kappa$ by $m$, i.e. we
take $m = 1$. To generate data, we simulate the systems using
a third order strong stability preserving Runge-Kutta method
{\citep{shu_efficient_1988}} for a fixed period of time with
initial conditions $\{ x_0, v_0 \}$ sampled from
$\Omega_S = [- 1, 1]^2$. Then, we use OnsagerNet
{\eqref{eq:OnsagerNet}} to learn an ODE system by fitting
the simulated data.

In particular, we will demonstrate that the energy $V$
learned has physical meaning.  Note that the energy function
in the generalized Onsager principle need not be unique. For
example, for the heat equation $\pt{u} = \Delta u$, both
$\frac{1}{2} \| u \|^2$ and $\frac{1}{2} \| \nabla u \|^2$
can serve as an energy functional governing the dynamics,
with diffusion operators ($M$ matrix) being $- \Delta$ and
the identity respectively. The linear Hookean model is
similar. Let
$V (x, v) = \frac{1}{2} \kappa x^2 + \frac{1}{2} v^2$, then
\begin{equation}
  \left(\begin{array}{c}
          \pt{x}\\
          \pt{v}
        \end{array}\right) = \left(\begin{array}{cc}
                               0 & 1\\
                               - \kappa  & - \gamma
  \end{array}\right) \left(\begin{array}{c}
                             x\\
                             v
                           \end{array}\right) = - \left(\begin{array}{cc}
                                                          0 & - 1\\
                                                          1 & \gamma
                                                        \end{array}\right) \nabla V (x, v) . \label{eq:linLangevin}
\end{equation}
The eigenvalue of the matrix $A = \left(\begin{array}{cc}
                                          0 & 1\\
                                          - 1 & - \gamma
\end{array}\right)$ is $\lambda_{1, 2} = - \frac{\gamma}{2} \pm \frac{1}{2}
\sqrt{\gamma^2 - 4}$. \ When $\gamma \geqslant 2$, we always
obtain real negative eigenvalues, and the system is
over-damped. From equation {\eqref{eq:linLangevin}}, we may
define another energy
$\tilde{V} (x, v) = \frac{1}{2} x^2 + \frac{1}{2} v^2$ with
dissipative matrix and conservative matrix
$- \frac{1}{2} (A + A^T)$, $\frac{1}{2} (A - A^T)$
respectively.  For this system,
$\hat{V} (x, v) : = \tilde{V} (F (x, v))$ with any
non-singular linear transform $F$ could serve as energy, and the
corresponding dynamics is
\[ \left(\begin{array}{c}
           \pt{x}\\
           \pt{v}
   \end{array}\right) = A (F^T F)^{- 1} \nabla \hat{V} . \]
Hence, we use this transformation to align the learned
energy before making comparison with the exact energy
function.


Let us now present the numerical results. We test two cases:
1) Hookean model with $k = 4$, $\gamma = 3$; 2) Pendulum
model with $k = 4$ and $\gamma (v) = 3| v |^2$.  To generate
sample data, we simulate the ODE systems to obtain 100
trajectories with uniform random initial conditions
$(x, v) \in [- 1, 1]^2$. For each trajectory, 100 pairs of
snapshots at $(i\, T/100, i\, T/100 + 0.001)$,
$i = 0, \ldots, 99$ are used as sample data. Here $T=5$ is the
chosen time period of sampled trajectories.
Snapshots from the first 80 trajectories are taken as the
training set, while the remaining snapshots are taken as the
test set.

We test three different methods for learning dynamics:
OnsagerNet, a symplectic dissipative ODE net (SymODEN
\citep{zhongSymplecticODENetLearning2020}), and a simple
multi-layer perception ODE network (MLP-ODEN) with ResNet
structure.  To make the numbers of trainable parameters in
three ODE nets comparable, we choose $l=1$ and $n_l=12$ for
OnsagerNet, $n_l =17$ for SymODEN, and MLP-ODEN with 2
hidden layers and each layer has $9$ hidden units, such that
the total numbers of tunable parameters in OnsagerNet,
SymODEN and MLP-ODEN are 120, 137, and 124, correspondingly.

To test the robustness of those networks paired with
different activation functions, five $C^1$ activation
functions are tested, including ReQU, ReQUr, softplus,
sigmoid and $\tanh$. Here, ReQU, defined as
$(x) \assign x^2$ if $x \geqslant 0$, otherwise 0, is the
rectified quadratic unit studied in
{\citep{liBetterApproximationsHigh2020}}.  Since
$\tmop{ReQU}$ is not uniformly Lipschitz, we introduce ReQUr
as a regularization of ReQU, defined as
$\tmop{ReQUr} (x) : = \tmop{ReQU} (x) - \tmop{ReQU} (x -
0.5)$.

The networks are trained using a new version of Adam
\citep{reddi_convergence_2018} with mini-batches of size 200
and initial learning rate $0.0256$.
The learning rate is halved if the loss is not decreasing in 25 epochs.
The
default number of iterations is set to 600 epochs.

For the Hookean case, the mean squared error (MSE) loss on
testing set can be reduced to about $10^{-5}\sim 10^{-8}$
for all the three ODE nets, depending on different random
seeds and activation functions used. For the nonlinear
pendulum case, the MSE loss on testing set can be reduced
to about $10^{-4} \sim 10^{-5}$ for OnsagerNet and
$10^{-3} \sim 10^{-5}$ for MLP-ODEN, but only $10^{-2}$ for
SymODEN, see Fig \ref{fig:cmp_ODENets}.  The reason for the
low accuracy of SymODEN is that in the SymODEN, the
diffusion matrix is assumed to be a function of general
coordinate $x$ only
\citep{zhongSymplecticODENetLearning2020}, but here in the
damped pendulum problem, the diffusion term depends on $v$.
From the test results presented in Figure
\ref{fig:cmp_ODENets}(a), we see that the results of OnsagerNet
are not sensitive to the nonlinear activation functions
used.
Moreover, \ref{fig:cmp_ODENets}(b) shows that OnsagerNet
has much better long time prediction accuracy.
Since the nonlinearities in many practical dynamical systems
are of polynomial type, we will mainly use ReQU
and ReQUr as activation functions for other numerical
experiments in this paper.

\begin{figure}[!ht]
  \centering
  \subfigure[Testing MSE accuracy: the height of the bars stands for
      $-\log_{10}$ of the testing MSE.
      The results are averages from training with three different random seeds.
      The heights of the red crosses on top of the bars indicate standard deviations.]
    {
      \includegraphics[width=0.45\linewidth]{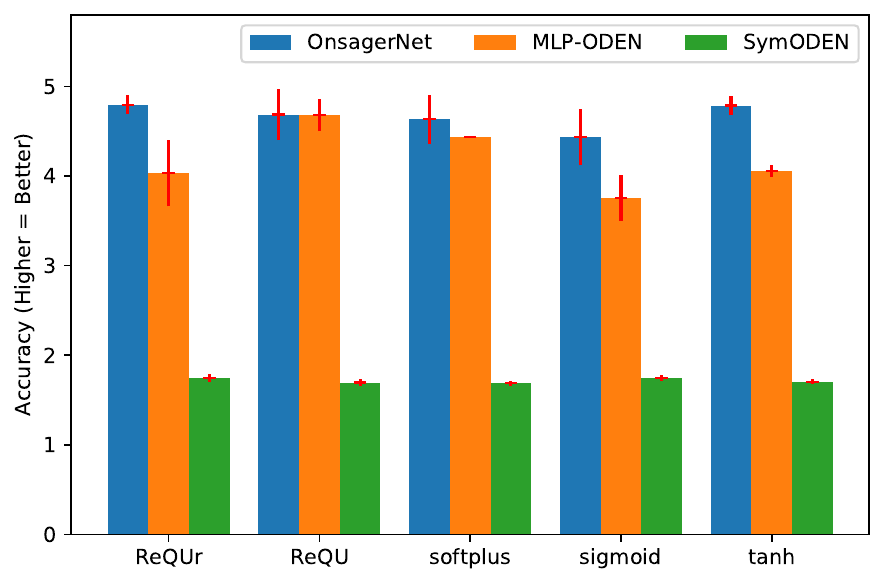}
    }
  \subfigure[
    Relative error of predictions versus time for three ODE nerual networks with the
    ReQUr activation function.
  ]{
     \includegraphics[width=0.45\linewidth]{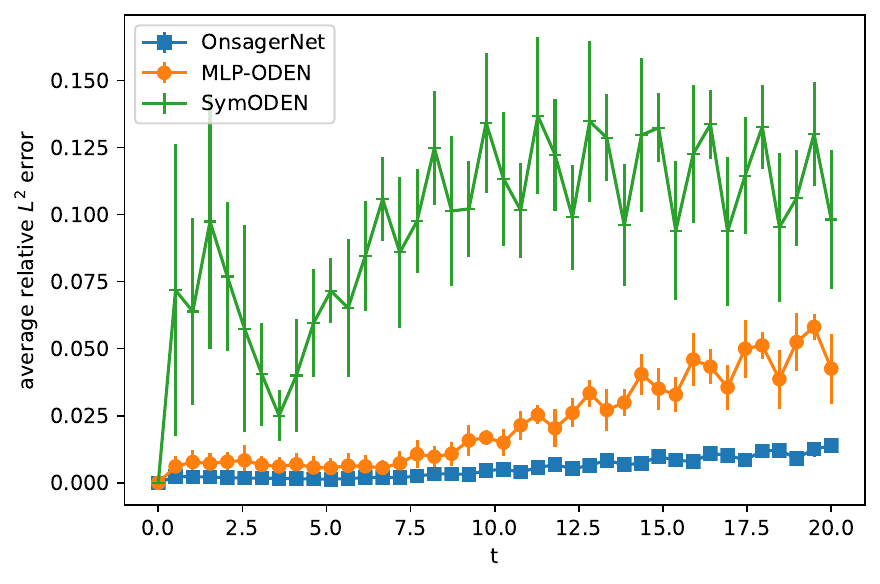}
  }
  \caption{\label{fig:cmp_ODENets} The accuracy of learned
    dynamics by using three different ODE neural networks
    and five different activation functions for the
    nonlinear damped pendulum problem with $\kappa=4$,
    $\gamma = 3 v^2 $.   }
\end{figure}

In Fig \ref{fig:LangevinPot} we plot the contours of learned
energy functions using OnsagerNet and compare with the exact
ground truth total energy function $U(x)+v^2/2$.
We observe a clear correspondence up to rotation and scaling for the linear
Hookean model case and a scaling of the nonlinear pendulum
case. In all cases, the minimum $(x = 0, v = 0)$ is
accurately captured.
Note that we used a linear transform to align the learned free energy.
After the alignment, the relative $L^2$-norm error between the
learned and physical energy for the two tested cases are
$6.3 \times 10^{- 3}$ and $8.6 \times 10^{-2}$
respectively. To verify that the OnsagerNet approach is able to
learn non-quadratic potentials, we also test an example with exact
double-well type potential $U(x)=(x^2-0.5)^2$ and $\gamma=3$. The relative
$L^2$-norm error between the learned and exact energy is $7.5\times 10^{-2}$,
where a simple $\min$-$\max$ rescaling is used before calculating the numerical
error for this example.

\begin{figure}[ht]
  \centering
  \subfigure[ReQUr OnsagerNet results for Hookean model with  $k = 4, \gamma =
  3$.]{
    \includegraphics[width=0.7\linewidth]{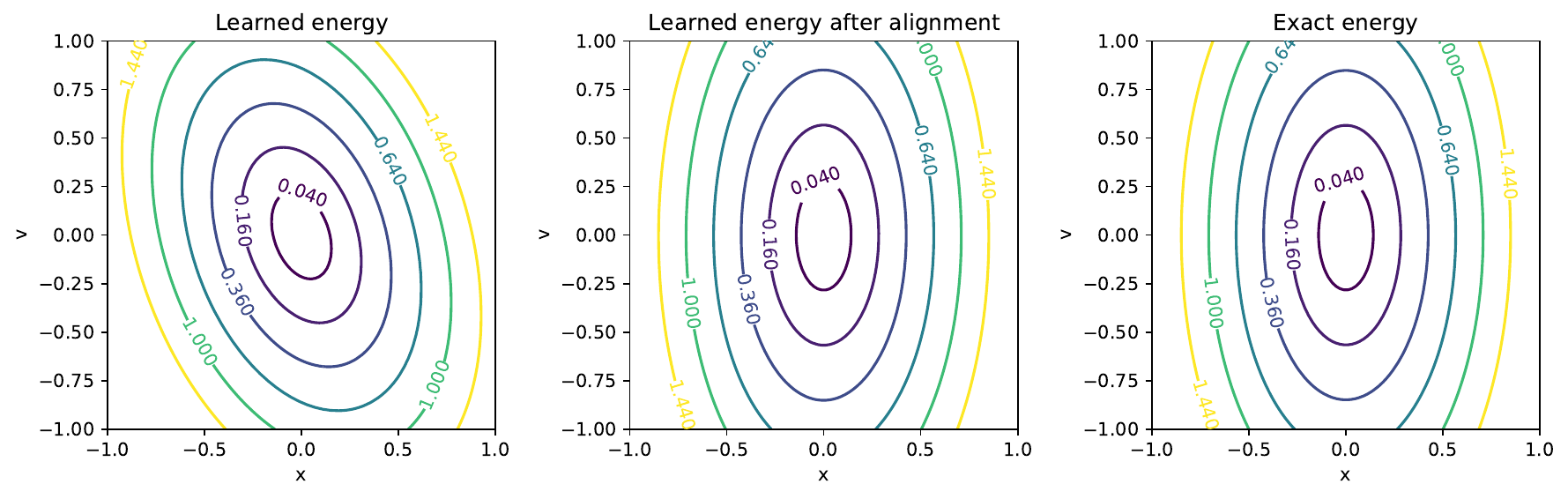}
  }
  \subfigure[ReQU OnsagerNet results for pendulum model with $k = 4$, $\gamma =
  3|v|^2$.]{
    \includegraphics[width=0.7\linewidth]{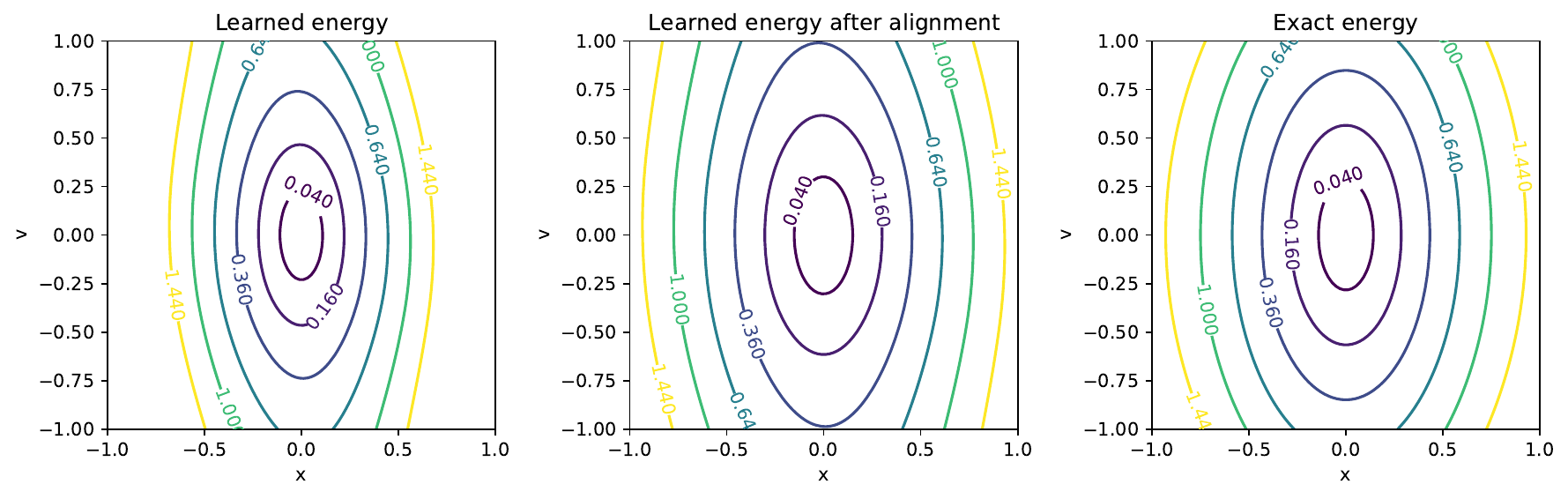}
  }
   \subfigure[ReQU OnsagerNet results for double-well potential
   $U(x)=(x^2-0.5)^2$ with $\gamma = 3$.]{
      \includegraphics[width=0.7\linewidth]{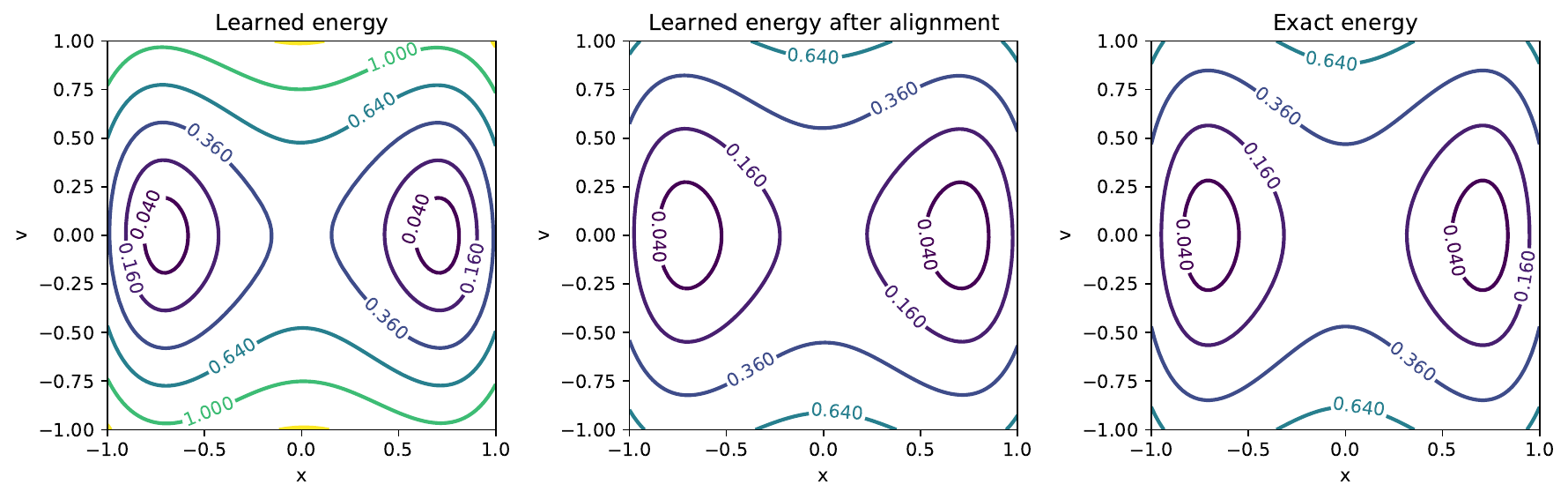}
  }
  \caption{\label{fig:LangevinPot} The learned energy
    functions and the exact energy functions in three problems.}
\end{figure}

In Fig. \ref{fig:LangevinPath}, we plot the trajectories for
exact dynamics and learned dynamics with initial values
starting from four boundaries of the sample region.  We see
that the learned dynamics are quantitatively accurate, and
the qualitative features of the phase space are also
preserved over long times.

\begin{figure}[ht]
  \centering
  \includegraphics[width=0.85\linewidth]{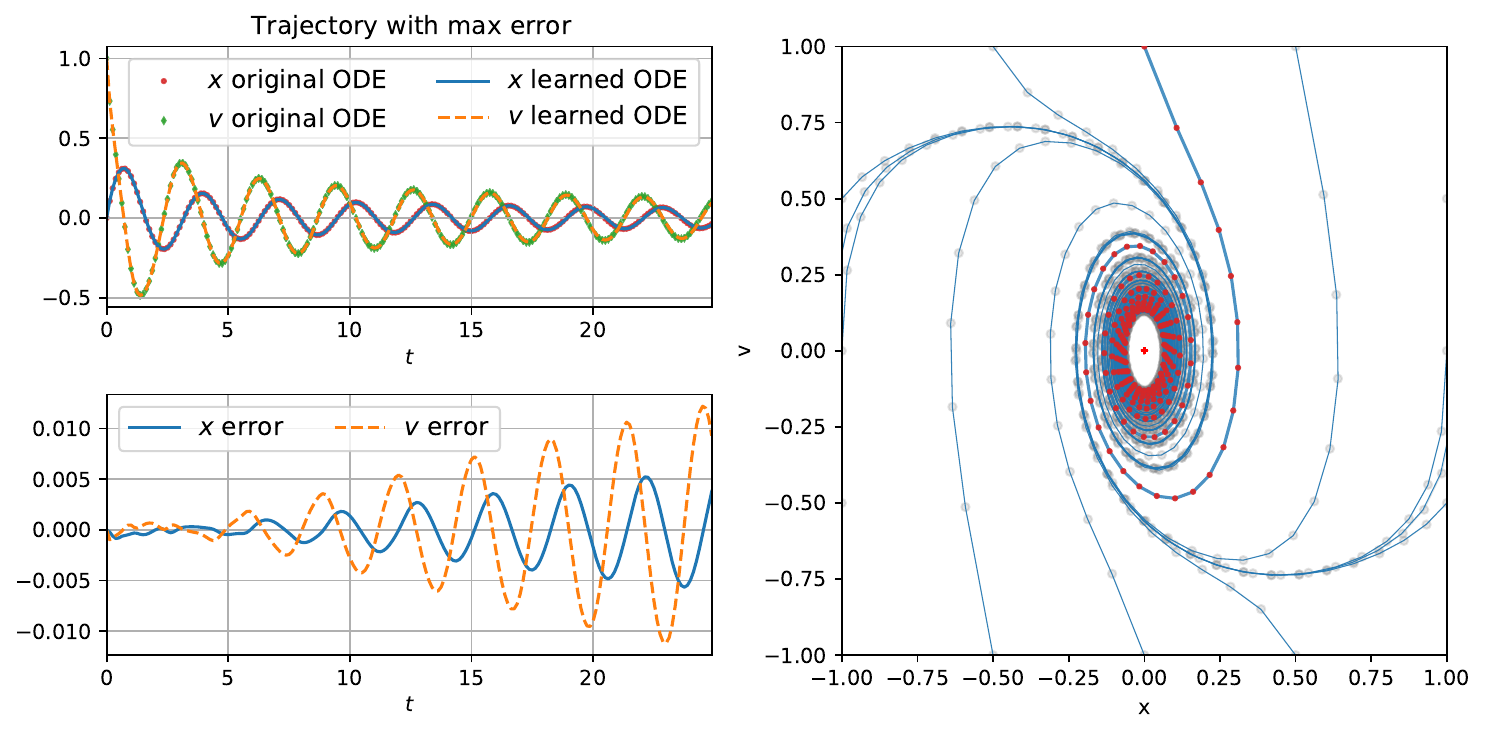}
  \caption{\label{fig:LangevinPath} Results of learned ODE
    model by ReQU OnsagerNet for the nonlinear damped
    pendulum problem.  In the right plot, the dots are
    trajectory snapshots generated from exact dynamics, the
    solid curves are trajectories generated from learned
    dynamics.  The red cross is the fixed point numerically
    calculated for the learned ODE system. Note that the
    time period of sampled trajectories used to train the
    OnsagerNet is $T=5$.}
\end{figure}

\subsection{Benchmark problem 2: the Lorenz '63 system}

The previous oscillator models have simple Hamiltonian
structures, so it is plausible that with some engineering
one can obtain a special structured model that works well
for such systems. In this section, we consider a target
nonlinear dynamical system that may not have such a simple
structure. Yet, we will show that owing to the flexibility of
the generalized Onsager principle, OnsagerNet still performs
well. Concretely, we consider the well-known Lorenz
system~\citep{lorenz_deterministic_1963}:
\begin{eqnarray}
  \frac{dX}{d\tau} & = & - \sigma X + \sigma Y, \label{Xeq} \\
  \frac{dY}{d\tau} & = & - XZ + rX - Y, \label{Yeq} \\
  \frac{dZ}{d\tau} & = & XY - bZ, \label{Zeq}
\end{eqnarray}
where $b>0$ is a geometric parameter, $\sigma$ is the
Prandtl number, $r$ is rescaled Rayleigh number.

The Lorenz system \eqref{Xeq}-\eqref{Zeq} is a simple
autonomous ODE system that produces chaotic solutions, and
its bifurcation diagram is well-studied
\citep{sparrow_lorenz_1982,barrio_threeparametric_2007}.  To
test the performance of OnsagerNet, we fix $b=8/3$,
$\sigma=10$, and vary $r$ as commonly studied in the
literature.  For the $b=8/3$, $\sigma=10$ case, the first
(pitchfork) bifurcation of the Lorenz system happens at
$r=1$, followed by a homoclinic explosion at
$r\approx 13.92$, and then a bifurcation to the Lorenz
attractor at $r\approx 24.06$. Soon after, the Lorenz
attractor becomes the only attractor at $r\approx 24.74$
(see e.g. \citep{zhou_study_2010}).  To show that OnsagerNet
is able to learn systems with different kinds of attractors
and chaotic behavior, we present results for $r=16$ and
$r=28$.

The procedure of generating sample data and training is
similar to the previously discussed case of learning
Langevin equations, except that here we set $\alpha=0.1$,
$\beta=0.1$ and a linear representation for ${f} (h)$ in
OnsagerNet (with $l=1$ and $n_l=20$). This is because the Lorenz system is a
forced
system.  The result for the case $r=16$ is presented in Fig
\ref{fig:Lorenz_r16}. We see that both the trajectories and
the stable fixed points and unstable limit cycles can be
learned accurately.  The results for the case $r=28$ is
presented in Fig \ref{fig:Lorenz_r28}. The largest Lyapunov
indices of numerical trajectories (run for sufficient long
times) are estimated using a method proposed by
\cite{rosenstein_practical_1993}. They are all positive,
which suggests that the learned ODE system indeed has
chaotic solutions.

\begin{figure}[!ht]
  \centering
  \includegraphics[width=0.65\linewidth]{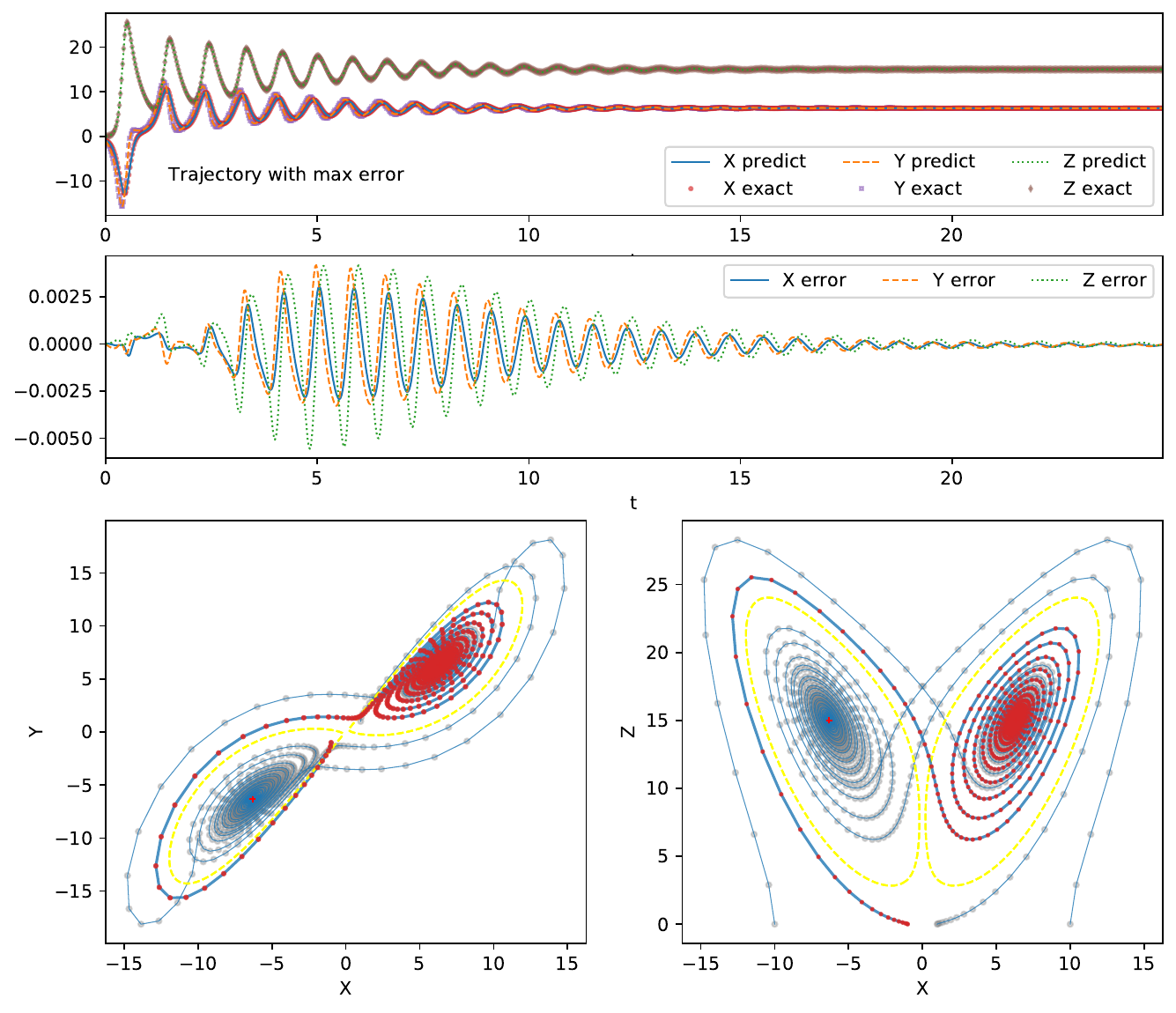}
  \caption{\label{fig:Lorenz_r16} Results of learned ODE
    model by ReQUr OnsagerNet for the Lorenz system
    \eqref{Xeq}-\eqref{Zeq} for $b=8/3, \sigma=10, r=16$.
    In the bottom plots, the dots are trajectory snapshots
    generated from exact dynamics, the solid curves are
    trajectories generated from learned dynamics, and the thick curve with red
    dots is the one with largest numerical error.  The small
    red crosses are the fixed points numerically calculated
    for the learned ODE system. The yellow closed curves
    are the unstable limit cycles calculated from the
    learned ODE system.}
\end{figure}

\begin{figure}[!ht]
  \centering
  \includegraphics[width=0.65\linewidth]{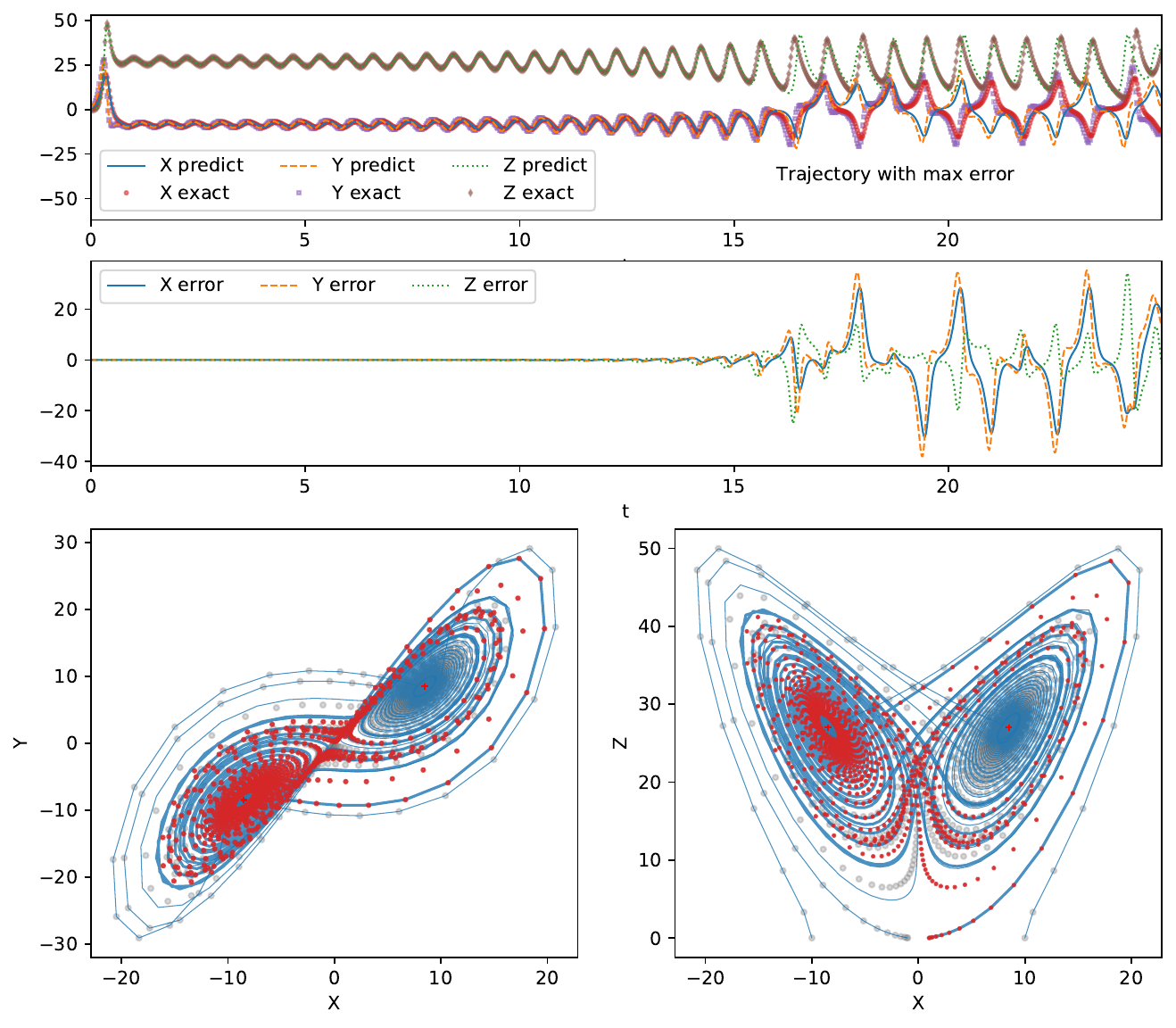}
  \caption{\label{fig:Lorenz_r28} Results of the learned
    ReQU OnsagerNet model for the Lorenz system
    \eqref{Xeq}-\eqref{Zeq} for $b=8/3, \sigma=10, r=28$.
    In the bottom plots, the dots are trajectory snapshots
    generated from exact dynamics, and the solid curves are
    trajectories generated from learned dynamics. The thick one with red dots
     is the one with largest numerical error.  The two
    red crosses are the (unstable) fixed points numerically
    calculated for the learned ODE system.}
\end{figure}

Finally, we compare OnsagerNet with MLP-ODEN for learning
the Lorenz system.  The SymODEN method
\citep{zhongSymplecticODENetLearning2020} cannot be applied
since the Lorenz system is non-Hamiltonian.
The OnsagerNets used here has one shared hidden layer with
20 hidden nodes, and the total number of trainable
parameters is 356. The MLP-ODEN nets have 2 hidden layers
with 16 hidden nodes in each layer, with a total 387 tunable
parameters.  In Fig. \ref{fig:Lorenz_cmp_ODENets}, we show
the accuracy on the test data set for OnsagerNet and
MLP-ODEN using $\tmop{ReQU}$ and $\tmop{ReQUr}$ as
activation functions, from which we see OnsagerNet performs
much better.
To ensure that these results hold across different model configurations,
we also tested learning the Lorenz system with larger OnsagerNets and MLP-ODENs with
3 hidden layers with 100 hidden nodes each.
Some typical training and testing loss curves are given
in Figure \ref{fig:Lorenz_nLw}, from which we see OnsagerNet perform better
than MLP-ODEN, and activation function ReQUr peforms better than {\tt tanh}
as activation function.
\begin{figure}[!ht]
  \centering
  \includegraphics[width=0.6\linewidth]{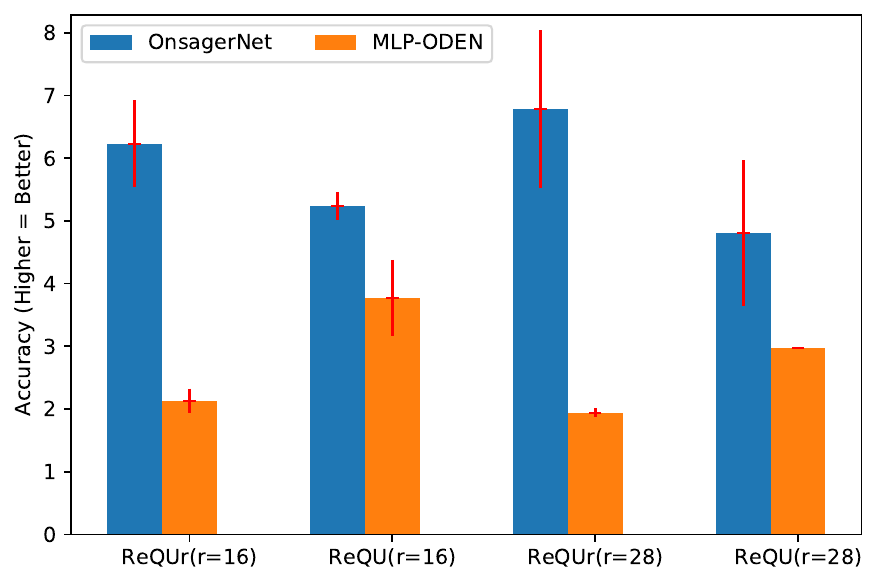}
  \caption{\label{fig:Lorenz_cmp_ODENets} The accuracy of
    learned dynamics by using OnsagerNets and MLP-ODEN with
    $\tmop{ReQU}$ and $\tmop{ReQUr}$ as activation functions
    for learning the Lorenz system with $r=16, 28$.  The
    height of the bars stands for $-\log_{10}$ of the
    testing MSE plus 3.5, the higher the better. The heights of the red
    crosses on top of the bars indicate standard deviations.
    The results are averages of trainings using three different random
    seeds.}
\end{figure}

\begin{figure}[!ht]
    \centering
    \subfigure[MLP-ODEN with ReQUr]{
    \includegraphics[width=0.38\linewidth]{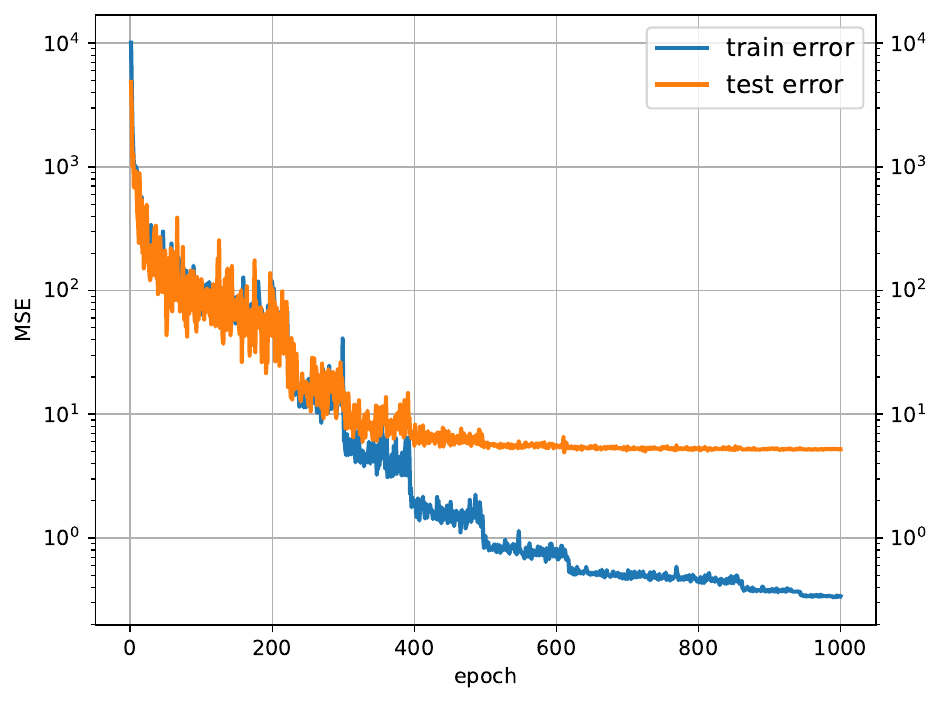} }
    \subfigure[MLP-ODEN with {\tt tanh}]{
    \includegraphics[width=0.38\linewidth]{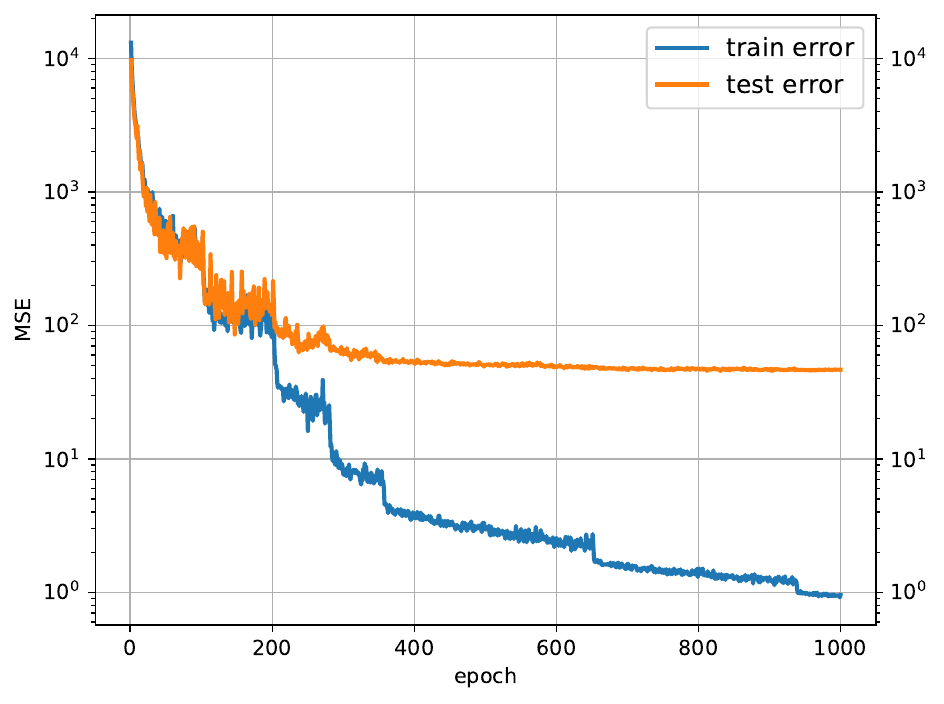} }
    \subfigure[OnsagerNet with ReQUr]{
    \includegraphics[width=0.38\linewidth]{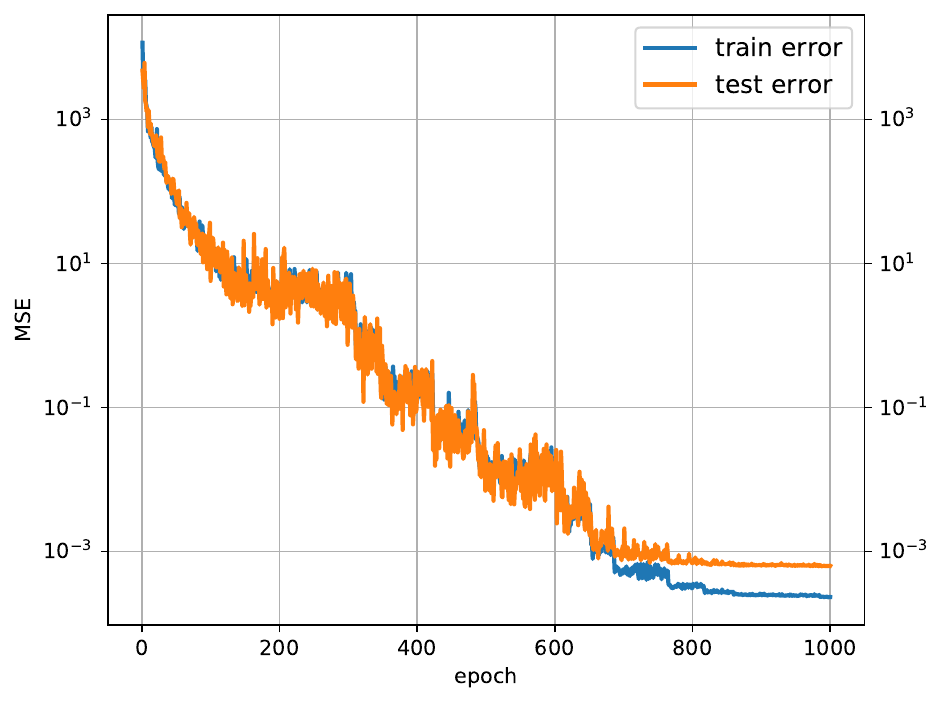} }
    \subfigure[OnsagerNet with {\tt tanh}]{
    \includegraphics[width=0.38\linewidth]{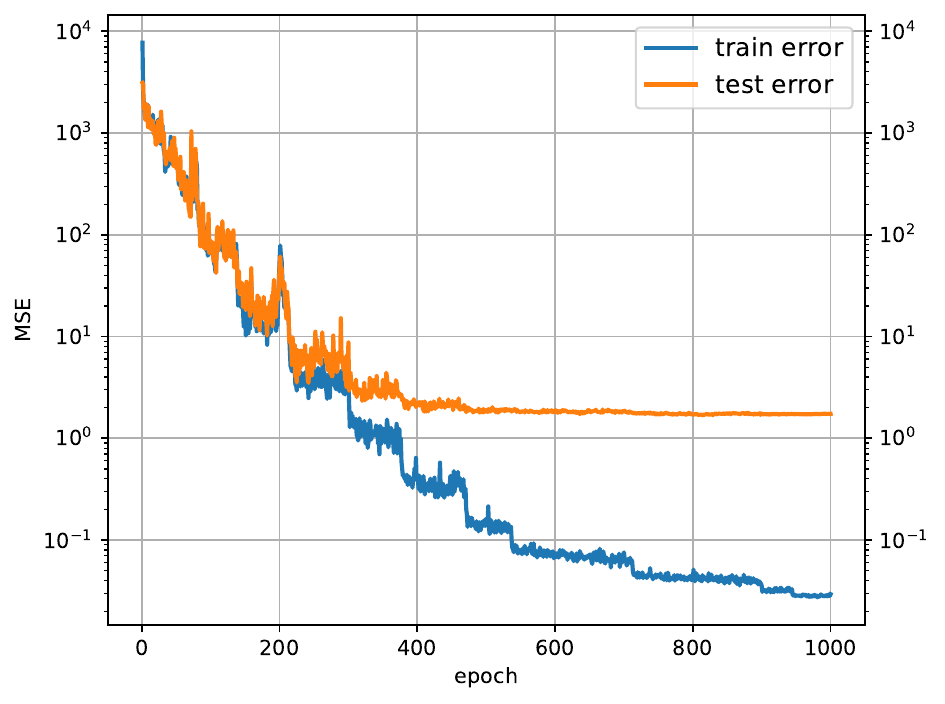} }
    \caption{\label{fig:Lorenz_nLw} Typical training MSE curves to learn the
    Lorenz system \eqref{Xeq}-\eqref{Zeq} ($r=28$) with
        over-parameterized neural ODE networks.}
\end{figure}

\subsection{Application to Rayleigh-B\'{e}nard convection with fixed parameters}

We now discuss the main application of OnsagerNet in this
paper, namely learning reduced order models for the
2-dimensional Rayleigh-B\'{e}nard convection (RBC) problem,
described by the coupled partial differential equations
\begin{align}
 & \partial_t{\bu} + (\bu \cdot \nabla) \bu = \nu \Delta \bu - \nabla p +
  \alpha_0  g_0 \hat{y}  \theta,
   \qquad \nabla \cdot \bu = 0, \label{nse1}\\
  & \partialt{\theta} + \bu \cdot \nabla \theta = \kappa \Delta \theta +
  \Gamma v, \label{nse2}
\end{align}
where $\bu = (u, u, 0)$ is the velocity field,
$g_0$ the gravitational acceleration, $\hat{y}$ is the
unit vector opposite to gravity, $\theta$
the departure of the temperature from the linear temperature
profile $\bar{\theta} (y) = \bar{\theta}_H - \Gamma y$,
$\Gamma = (\bar{\theta}_H - \bar{\theta}_L) / d$. $d$ is the
depth of the channel between two plates. The constants
$\alpha_0, \nu$, and $\kappa$ denote, respectively, the coefficient of thermal
expansion, the kinematic viscosity, and the thermal
conductivity.  The system is assumed to be homogeneous in
$z$ direction, and periodic in $x$ direction, with period
$L_x$. { The dynamics depends on three non-dimensional parameters:
the Prandtl number $Pr = \nu / \kappa$,
the Rayleigh number $Ra = d^4\frac{g_0 \alpha_0 \Gamma}{\nu \kappa}$,
and aspect ratio $a = 2d / L_x$.
The critical Rayleigh number to maintain the stability of zero
solution is $R_c = \pi^4 (1 + a^2)^3 / a^2$.
The terms $ \alpha_0  g_0 \hat{y}  \theta$ in \eqref{nse1} and $\Gamma
v$ in \eqref{nse2} are external forcing terms.
For given divergence-free velocity $\bu$ with
no-flux or periodic boundary conditions, the convection operator $\bu\cdot
\nabla$ is skew-symmetric. Thus, the RBC equations satisfy the generalized
Onsager principle with potential $\frac12\|u\|^2 + \frac{1}{2} \|\theta\|^2$.
Therefore, the OnsagerNet approach, which places the prior that
the reduced system also satisfy such a principle, is appropriate
in this case.}
The velocity $\bu$ can be represented by a
stream function $\phi (x, y)$:
\begin{equation}
  \bu = (- \partial_y \phi, \partial_x \phi, 0) . \label{uphirel}
\end{equation}
By eliminating pressure, one gets the following equations for $\phi$ and
$\theta$:
\begin{align}
  \partial_t \Delta \phi - \partial_y \phi \partial_x (\Delta \phi) +
  \partial_x \phi \partial_y (\Delta \phi)
  &= \nu \Delta^2 \phi + g_0 \alpha_0
    \partial_x \theta, \label{phieq}
  \\
  \partial_t \theta - \partial_y \phi \partial_x \theta + \partial_x \phi
  \partial_y \theta
  &= \kappa \Delta \theta + \Gamma \partial_x \phi .
\label{thetaeq}
\end{align}

  The solutions
$\phi, \theta$ to~\eqref{phieq} and~\eqref{thetaeq} have representations
in Fourier sine series as
\[ \theta (x, y) = \sum_{k_1 = -\infty}^{\infty} \sum_{k_2 = 1}^{\infty}
  \theta_{k_1 k_2} e^{2i \pi k_1 x / L_x} \sin (\pi k_2 y / d) \text{}, \]
 \[ \phi (x, y) = \sum_{k_1 = -\infty}^{\infty} \sum_{k_2 = 1}^{\infty}
   \phi_{k_1 k_2} e^{2i \pi k_1 x / L_x} \sin (\pi k_2 y / d), \] where
 $\theta_{k_1 k_2} = \bar{\theta}_{- k_1 k_2}$ and
 $\phi_{k_1 k_2} = \bar{\phi}_{- k_1 k_2}$ since $\theta$ and
 $\phi$ are real.  In the Lorenz system, only 3 most important modes
 $\phi_{11}$, $\theta_{11}$ and $\theta_{02}$ are retained, in this
 case, the solution have following form
\begin{equation}
  \phi (x, y, t) = \frac{(1 + a^2) \kappa}{a} \sqrt{2} X (t) \sin (2\pi x / L_x) \sin
  (\pi y / d), \label{phiX}
\end{equation}
\begin{equation}
  \theta (x, y, t) = \frac{R_c \Gamma d}{\pi \Ra}
  \left( \sqrt{2} Y (t) \cos(2\pi x / L_x) \sin (\pi y / d) - Z (t) \sin (2 \pi y / d) \right),
  \label{thetaYZ}
\end{equation}
The Lorenz '63 equations \eqref{Xeq}-\eqref{Zeq} for the
evolution of $X, Y, Z$ is obtained by a Galerkin approach
{\citep{lorenz_deterministic_1963}}
with time rescaling $\tau = \pi^2 \frac{(1 + a^2)}{d^2} \kappa t$, the
rescaled Rayleigh number $r = \Ra / R_c$. $b = 4 / (1 + a^2)$, and
$\sigma = \nu / \kappa$ is the Prandtl number.

Since Lorenz system is aggressively truncated from the original RBC
problem, it is not expected to give quantitatively
accurate prediction of the dynamics of the original system when $r\gg 1$.
Some extensions of Lorenz system to dimensions higher than 3 are proposed, e.g. Curry's 14-dimensional model \citep{curry_generalized_1978}. But numerical experiments have shown that much large numbers of spectral coefficients need be retained to get quantitatively good results in a Fourier-Galerkin approach \citep{curryOrderDisorderTwo1984}. In fact,
\cite{curryOrderDisorderTwo1984} used more than 100 modes to obtain convergent results
for a parameter region similar to $b=8/3$, $\sigma=10$, $r=28$ used by Lorenz.
In the following, we show that by using OnsagerNet, we are able to directly
learn reduced order models from RBC solution data that are quantitatively
accurate and require much fewer modes than the Galerkin projection method.

\subsubsection{Data generation and equation parameters}

We use a Legendre-Fourier spectral
method to solve the Rayleigh-B\'{e}nard convection problem
{\eqref{nse1}}-{\eqref{nse2}} based on stream function formulation
{\eqref{phieq}}-{\eqref{thetaeq}} to generate sample data. To be consistent with the case
considered by Lorenz, we use free boundary condition for velocity.  A
RBC problem is chosen with the following parameters: $\tmop{Re} = 10$,
$L_x = 4 \sqrt{2}$, $\alpha_0 = 0.1$, $g = 9.8$, $\kappa = 0.01$,
$\Gamma = 1.17413$.  The corresponding Prandtl number, Rayleigh number
are $\sigma = 10$, $\tmop{Ra} = 18410.3$. The relative Rayleigh number is
$r=\tmop{Ra} / \tmop{Rc} = 28$. In this section, we use OnsagerNet to learn one
dynamical model for each fixed $r$ value.
The results of learning one parametric dynamical model for a
wide range of $r$ value are given in
next subsection. 
Initial conditions of Lorenz form {\eqref{phiX}}-{\eqref{thetaYZ}} are
used, where $X, Y, Z$ are sampled from Gaussian distribution and
rescaled such that $X^2 + Y^2 + Z^2 \leqslant R_B$, with $R_B$ a constant.

The semi-discretized Legendre-Fourier system is solved using a second
order implicit-explicit time marching scheme with time step-size $\tau = 0.001$
for $100$ time units. $128$ real Fourier modes are used for $x$ direction and
$49$ Legendre modes for $y$ direction. To generate training and testing data,
we simulated 100 initial values. The solver outputs 2 snapshots of $(u, v,
\theta)$ at $(t, t + \tau)$ for each integer time $t$. The data from the first
80 trajectories are used as training data, while the last 20 trajectories are
used as testing data.

\begin{figure}[ht]
    \centering \subfigure[The relative variance of first 32 principal
    components (PC)]{\label{fig:pca1}
        \includegraphics[width=0.45\textwidth]{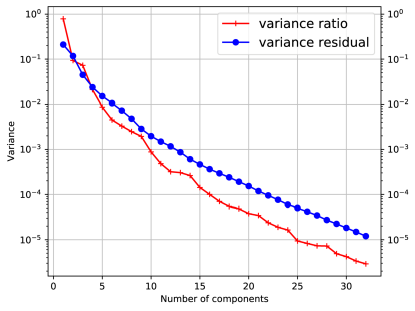}}
    \subfigure[The first 3 principal components of the sampled trajectories
    of Rayleigh-B\'{e}nard convection problem.]{\label{fig:traj3pc}
        \includegraphics[width=0.45\textwidth]{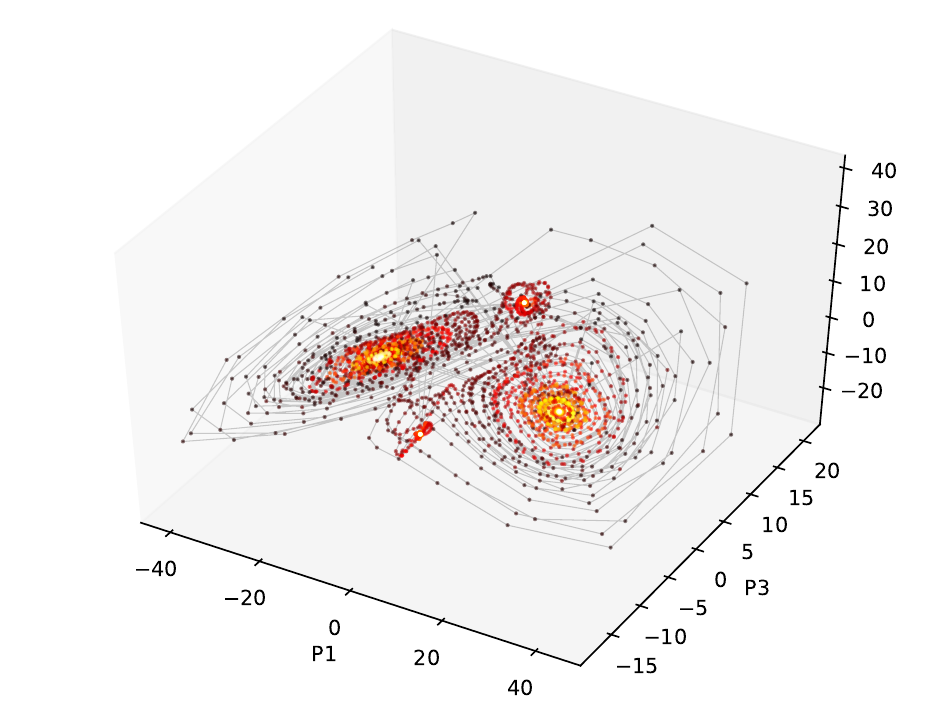}}
    \caption{Relative variance of principle components and the sample trajectories for the Rayleigh-B\'{e}nard convection
        problem with $r=28$ projected on the first three components.
        The trajectories are generated using Lorenz type
        initial values with random amplitude. }
\end{figure}

To have an estimate of the effective dimension, we first apply PCA to the
data. The result for $r=28$ is shown in Fig. \ref{fig:pca1}. We observe that
99.7\% variance are captured
by the first 9 principal components, so we seek a reduced model of comparable dimensions.

\subsubsection{The structure of OnsagerNet}

The network parameters are similar to that in learning the Lorenz ODE
system. Since the Rayleigh-B\'{e}nard convection system is not
a closed system, we ensure the stability established in Theorem
\ref{thm:noblowup} by taking $\alpha= 0.1$, $\beta = 0.1$. To make
${f}(h)$ Lipschitz, we simply let
${f}(h)$ be an affine function of $h$.
We use two common hidden layers (i.e. $l=2$) with each layer having
$n_1=2 C^m_{m + 2}$ hidden nodes for evaluating $L (h), R(h), U_i (h)$ in
OnsagerNet.
We use ReQUr as activation function in this application as it
gives slight better training results than ReQU.
The ReQUr OnsagerNet is trained using the standard Adam optimizer
\citep{kingmaAdamMethodStochastic2015}
to minimize the loss with one embedded RK2 step.
The embedding of multiple RK2 steps
improves the accuracy only slightly since the time interval of data set
$\tau=0.001$ is small, and so the error of RK2 integrator is not the major error
in the total loss.
The use of multiple hidden layers in OnsagerNet also improves accuracy.
Since our purpose is to find small models for the dynamics, thus we only
present numerical results of OnsagerNets with two common hidden layers.

\subsubsection{Numerical results for r=28}

We now present numerical results of the RBC problem with $b=8/3$, $\sigma=10$, and $r=28$,
a parameter set when used in the Lorenz approximation leads to chaotic solutions, but
for the original RBC problems, have only fixed points as attractors.
We perform PCA for the sampled data and train
OnsagerNets using $3, 5, 7$ principal components.
The OnsagerNets are trained with  batch size 200, initial learning rate $0.0064$
and reduced by half if the loss is not decrease in 25 epochs.
After the reduced ODE
models are learned, we simulate the learned equations using a third order Runge-Kutta method \citep{shu_efficient_1988}
for one time unit with initial conditions taken from the PCA data
at $t_j, j = 0, \ldots 99$, and compare the results to the original PCA
data at $t_{j + 1}$.
To show the learned ODE system is stable, we also
simulate the learned ODE models for $99$ time units.
We summarize the
results in Table \ref{tbl:RBCres}.
For comparison, we also present results of
MLP-ODEN.
We see that OnsagerNets have better long time stability than MLP-ODEN.
The huge $t=99$ prediction error in MLP-ODEN PCA $m=5$ model indicates that
the MLP-ODEN model learned is not stable.
From Table \ref{tbl:RBCres}, we also observe that the OnsagerNets
using only three variables ($m=3$) can give good prediction for the period of one
time unit ($E_{t=1}^{pred, rel}$), but has a large relative $L^2$ error ($E_{t=99}^{pred, rel}$) for long times $(t = 99)$.
By increasing $m$ gradually to $5, 7$, both the short time and long
time prediction accuracy increase.

\begin{table}[ht]
    \begin{center}
    \renewcommand{\tabcolsep}{0.3cm}
    \caption{\label{tbl:RBCres} Numerical results of learning a reduced
        hidden dynamics for the RBC problem ($r=28$). }
        \begin{tabular}{rcccccr}
            \hline\hline
          Method\&Dim & \# Parameters & $\tmop{MSE}_{\tmop{train}}$
          & $\tmop{MSE}_{\tmop{test}}$ & $E_{t = 1}^{pred,rel}$ & $E_{t =
          99}^{pred, rel}$
          & $N_{\tmop{fail}}$\\
          \hline
MLP-ODEN PCA 3 &   983 & 2.63$\times 10^{-1}$ & 3.37$\times 10^{-1}$ &
2.32$\times 10^{-2}$  & 3.51$\times 10^{-1}$ & 62/3
\\
MLP-ODEN PCA 5 &  4079 & 2.95$\times 10^{-2}$ & 7.84$\times 10^{-2}$ &
8.18$\times 10^{-3}$  & 4.12$\times 10^{+4}$ & 16/3
\\
MLP-ODEN PCA 7 & 11599 & 6.60$\times 10^{-3}$ & 2.68$\times 10^{-2}$ &
3.71$\times 10^{-3}$  & 4.79$\times 10^{-2}$ &  7/3
\\
          \hline
OnsagerNet PCA 3 &   776 & 3.17$\times 10^{-1}$ & 3.85$\times 10^{-1}$ &
2.54$\times 10^{-2}$  & 2.76$\times 10^{-1}$ &
53/3 \\
OnsagerNet PCA 5 &  3408 & 3.88$\times 10^{-2}$ & 7.47$\times 10^{-2}$ &
8.40$\times 10^{-3}$  & 6.87$\times 10^{-2}$ &
13/3 \\
OnsagerNet PCA 7 & 10032 & 6.71$\times 10^{-3}$ & 1.26$\times 10^{-2}$ &
2.68$\times 10^{-3}$  & 1.07$\times 10^{-2}$ &
2/3 \\          \hline
       OnsagerNet AE\  3 & 2.5M (AE) $+$ 776 & 4.95$\times 10^{-3}$ &
       1.97$\times 10^{-2}$ &
       5.34$\times 10^{-3}$ & 9.64$\times 10^{-2}$ & 16/3 \\
       OnsagerNet AE\  5 & 2.5M (AE) $+$ 3408 & 3.71$\times 10^{-3}$ & 1.19$\times 10^{-2}$ &
       3.35$\times 10^{-3}$ & 4.63$\times 10^{-2}$ & 9/3 \\
       OnsagerNet AE\  7 & 2.5M (AE) $+$ 10032 & 1.46$\times 10^{-3}$ & 4.88$\times 10^{-3}$ &
       1.48$\times 10^{-3}$ & 1.93$\times 10^{-2}$ & 3/3 \\
          \hline
        \end{tabular}\\
    \scriptsize $N_{\tmop{fail}}$ in last column presents the
        average (over 3 random seeds) number of trajectories (out of 100) in
        the learned ODE systems that do not converge to the correct final steady
        states.
    \end{center}
\end{table}

\begin{figure}
    \centering \subfigure[OnsagerNet + PCA $ m = 7$
    ]{\includegraphics[width=0.45\linewidth]{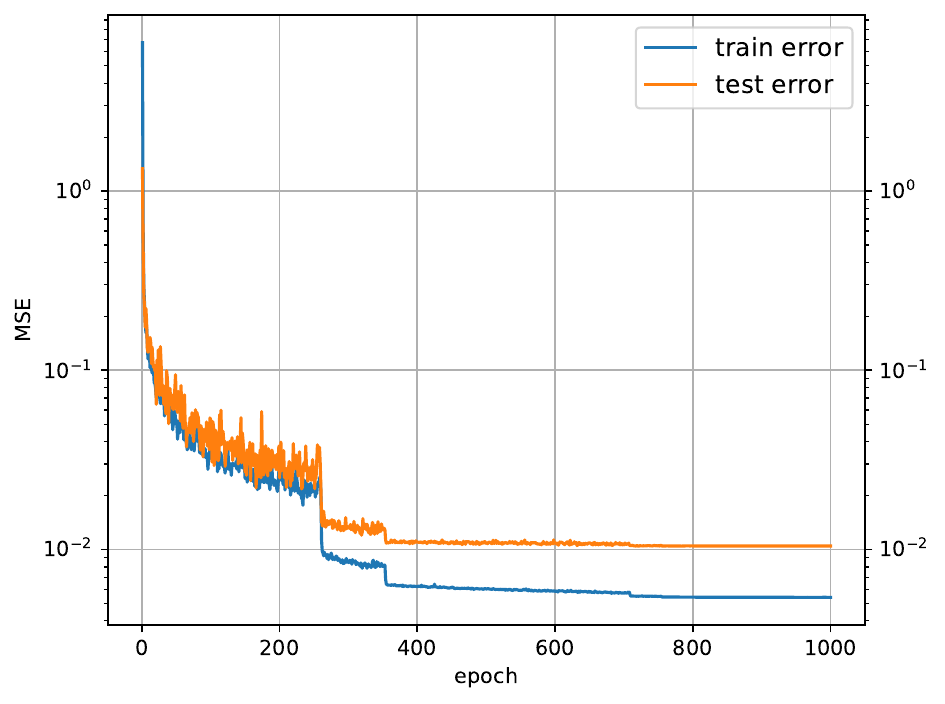}
    }
    \subfigure[MLP-ODEN + PCA $m =
    7$]{\includegraphics[width=0.45\linewidth]{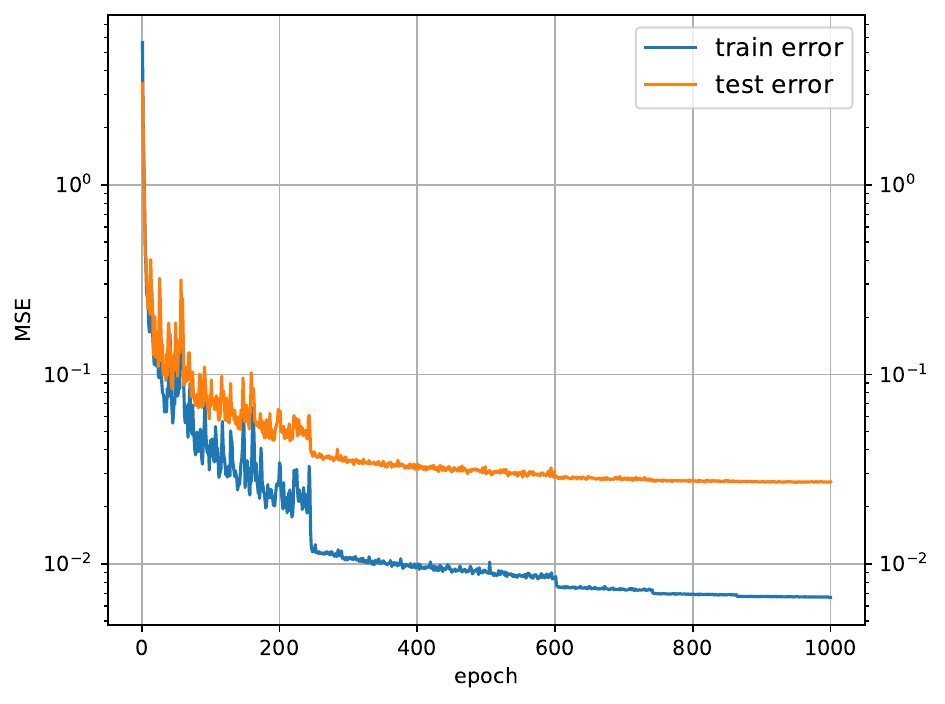}
    }
    \caption{The training and testing loss of OnsagerNet and MLP-ODEN for
        learning the RBC problem with $r=28$.}
    \label{fig:training}
\end{figure}

Some detailed training and testing losses are given in Figure \ref{fig:training}.
The gap between train error and testing error indicates that the train data
has non-negligible sampling error. But from the curve of testing error
shown in Fig \ref{fig:training}(a), we see the model is not overfitting.
We also observe MLP-ODEN has a much larger training-testing error gap than OnsagerNet.

To clearly visualize the numerical errors with respect to time, we show
numerical results of using learned OnsagerNets to simulate one selected trajectory in training set and one
in testing set with relatively large errors in Fig. \ref{fig:RBCtroj}. We see
that as more and more
hidden variables are used, the ODE models learned by OnsagerNets are increasingly
accurate, even for long times.

\begin{table}[ht]
    \begin{center}
        \renewcommand{\tabcolsep}{0.3cm}
        \caption{\label{tbl:RBCres_nn_ons}
            PCA trajectory prediction error on {\em testing set} for the RBC
            problem ($r=28$)
            using
            OnsagerNet and
            nearest neighbor(NN) method. }
        \begin{tabular}{crcccc}
            \hline\hline
             Method & Dim & $E_{t=1}^{pred,rel}$
                & $E_{t = 10}^{pred, rel}$
                &  $E_{t=20}^{pred,rel}$
                & $E_{t=99}^{pred, rel}$ \\
            \hline
          OnsagerNet &  $m$=3 & 2.65$\times 10^{-2}$ & 8.47$\times 10^{-2}$ &
          1.30$\times 10^{-1}$ & 2.63$\times 10^{-1}$ \\
          OnsagerNet &  $m$=5 & 9.29$\times 10^{-3}$ & 3.73$\times 10^{-2}$ &
          6.91$\times 10^{-2}$ & 8.46$\times 10^{-2}$ \\
          OnsagerNet &  $m$=7 & 3.24$\times 10^{-3}$ & 1.28$\times 10^{-2}$ &
          1.73$\times 10^{-2}$ & 3.23$\times 10^{-3}$ \\
            \hline
          NN &  $m$=3 & 9.76$\times 10^{-2}$ & 1.40$\times 10^{-1}$ &
          1.93$\times 10^{-1}$ & 1.02$\times 10^{-2}$  \\
          NN &  $m$=5 & 1.09$\times 10^{-1}$ & 1.676$\times 10^{-1}$ &
          2.10$\times 10^{-1}$ & 1.75$\times 10^{-1}$  \\
          NN &  $m$=7 & 1.13$\times 10^{-1}$ & 1.74$\times 10^{-1}$ &
          2.20$\times 10^{-1}$ & 2.69$\times 10^{-1}$ \\
          \hline
        \end{tabular}\\
    \end{center}
\end{table}

\begin{table}[ht]
    \begin{center}
        \renewcommand{\tabcolsep}{0.3cm}
        \caption{\label{tbl:RBCres_LSTM}
            PCA trajectory prediction MSE for the RBC problem with $r=28$ using
            LSTM.
            Since LSTM requires fixing the time sequence length a prior for prediction,
            different models have to be trained for different predictive horizons or resolutions.
            The architecture of the LSTM is chosen so that for each PCA embedding dimension (Dim),
            its number of trainable parameters is similar to that of the OnsagerNet.
        }
        \begin{tabular}{cccc}
            \hline\hline
            Dim & \# Parameters & $E_{t = 1}^{pred, rel}$ & $E_{t = 99}^{pred,
            rel}$ \\
            \hline
              $m=3$ & 1131 & 4.15$\times 10^{-1}$ & 2.09$\times 10^{-1}$ \\
              $m=5$ & 4245 & 1.23$\times 10^{-1}$ & 3.82$\times 10^{-1}$ \\
              $m=7$ & 15653 & 1.02$\times 10^{-1}$ & 3.73$\times 10^{-1}$ \\
            \hline
        \end{tabular}\\
    \end{center}
\end{table}

The results of training OnsagerNet together with a {PCA-ResNet} autoencoder
and regularized by the isometric loss defined in Eq. \eqref{eq:e2e_loss} are
also presented Table
\ref{tbl:RBCres}.
The autoencoder we used has 2 hidden encode layers and 2 hidden decode layers
with ReQUr as activation functions.
From Table
\ref{tbl:RBCres}, we see the results are improved in this ``Autoencoder + OnsagerNet'' approach, especially for models with few
hidden variables and short time predictions.

To demonstrate the advantange of PCAResNet for dimensionality reduction,
we also carried out additional numerical experiments on using different
autoencoders, including a standard stacked autoencoder (SAE), a
stacked autoencoder with contractive regularization
(CAE), which is a multilayer generalization of the original CAE
\cite{rifaiContractiveAutoEncodersExplicit2011}, and the
PCAResNet.
All three autoencoders use two hidden layers with 128 and 32 nodes respectively.
The activation function for PCAResNet is ReQUr, while the activation
functions for SAE and CAE are {\tt elu}. We tested other activation functions
including ReLU, softplus, elu, tanh and sigmoid for CAE and SAE,
and we found that {\tt elu} produce the best training results.
When training SAE and CAE with OnsagerNet, we first
pre-train SAE and CAE, than train OnsagerNet with low-dimensional
trajectories produced by pre-trained SAE and CAE, then train the autoencoder and
OnsagerNet together in the final step.
Note that PCAResNet does not require a pre-train step,
since it is initialized by a simple PCA. The performance of the three
different autoencoders are reported in Table \ref{tbl:ae_cmp}, where the
$\beta_{isom}$ and $\alpha_{isom}$ in equation \eqref{eq:e2e_loss} is
related to the parameter
$\beta_{isom}\star$ and $\alpha_{isom}^\star$ in Table \ref{tbl:ae_cmp} by:
\begin{equation}
\beta_{isom} = \beta_{isom}^\star \times
\frac{\mbox{pre-trained OnsagerNet MSE loss}}{\mbox{PCA isometric loss}},
\quad
\alpha_{isom} = \alpha_{isom}^\star \times [{\mbox{PCA isometric loss}}]
\end{equation}
From this table, we observe that PCAResNet produce better long time results
than SAE and CAE both with and without isometric regularization.
From the numerical results presented in
Table \ref{tbl:ae_cmp}, we clearly see the benefits of isometric
regularization - it reduces overfitting
(the ratio between testing MSE and training MSE is much smaller
when isometric regularization is used)
and improves the long time performance for all the three tested
autoencoders. Note that $\beta_{isom}\star=1$ and
$\alpha_{isom}^\star=0.8$ is used to produce the autoencoder results
presented in Table \ref{tbl:RBCres}.

\begin{table}[!htp]
    \renewcommand\arraystretch{1.2}
    \begin{center}
        \renewcommand{\tabcolsep}{0.3cm}
        \caption{\label{tbl:ae_cmp} Performance of three Autoencoders for RBC
            problem ($r=28$, nPC=3). }
        \begin{tabular}{rccrrrrrr}
            \hline\hline
            Autoencoders & $\beta_{isom}^\star$ & $\alpha_{isom}^\star$ &
            $\tmop{MSE}_{\tmop{train}}^{\mbox{\tiny ode}}$  &
            $\tmop{MSE}_{\tmop{test}}^{\mbox{\tiny ode}}$
            & $ \tfrac{\tmop{MSE}_{\tmop{test}}^{\mbox{\tiny
                        ode}}}{\tmop{MSE}_{\tmop{train}}^{\mbox{\tiny ode}}}$
            $\vphantom{ \frac{\frac{\mbox{O}}{D}}{\frac{\mbox{O}}{1} } }$
            & $E_{t=1}^{pred,rel}$ & $E_{t=99}^{pred, rel}$
            &$N_{\tmop{fail}}$\\
            \hline
            PCAResNet &1 & 0.0 & 1.99$\times 10^{-2}$ & 5.25$\times 10^{-2}$ &
            2.64 &
            7.53$\times 10^{-3}$ & 3.21$\times 10^{-1}$ & 85/5 \\
            PCAResNet &1 & 0.4 &1.27$\times 10^{-2}$ & 3.40$\times 10^{-2}$
            & 2.68 & 6.49$\times 10^{-3}$ & 2.67$\times 10^{-1}$ & 72/5 \\
            PCAResNet &1 & 0.8 & 4.91$\times 10^{-3}$ & 1.99$\times 10^{-2}$ &
            4.05 & 5.38$\times 10^{-3}$ & 1.01$\times 10^{-1}$ & {27/5} \\
            PCAResNet &1 & 1.2 & 4.00$\times 10^{-3}$ & 1.70$\times 10^{-2}$ &
            4.25 & 5.14$\times 10^{-3}$ & {\bf 8.64$\mathbf{\times 10^{-2}}$} & {\bf
                24/5} \\
            PCAResNet &1 & 1.6 & 3.78$\times 10^{-3}$ & 1.39$\times 10^{-2}$ &
            3.68 & 5.41$\times 10^{-3}$ & 1.31$\times 10^{-1}$ & 37/5 \\
            PCAResNet & 0 & - &1.07$\times 10^{-4}$ & 3.37$\times 10^{-3}$
            &31.50   & 2.24$\times 10^{-3}$ & 9.17$\times 10^{-2}$ & 49/5 \\
            \hline
            SAE &1 & 0.8 & 6.64$\times 10^{-3}$ & 3.62$\times 10^{-2}$ &
            5.45 & 4.13$\times 10^{-3}$ & 2.39$\times 10^{-1}$ & 92/5 \\
            SAE & 0 & - & 3.30$\times 10^{-4}$ & 2.26$\times 10^{-2}$
            & 68.49 & 6.60$\times 10^{-3}$ & 1.86$\times 10^{-1}$ & 96/5 \\
            \hline

            CAE &1 & 0.8 &5.58$\times 10^{-3}$ & 2.80$\times 10^{-2}$ &
            4.09 & 3.99$\times 10^{-3}$ & 1.62$\times 10^{-1}$ & 67/5 \\
            CAE & 0 & - & 3.45$\times 10^{-4}$ & 6.41$\times 10^{-2}$
            & 185.80 & 1.37$\times 10^{-2}$ & 3.26$\times 10^{-1}$ & 135/5 \\
            \hline
        \end{tabular}\\
        {\scriptsize
            $N_{\tmop{fail}}$ is the number of
            trajectories (out of 100, averaged over 5 seeds) in the learned
            ODE system that failed to correct final steady state.}
    \end{center}
\end{table}

We also carry out a comparison on the PCA trajectory prediction error
among the OnsagerNet, Long Short Term Memory (LSTM) recurrent neural network, and nearest neighbor
method (NN).
Since LSTM is a discrete-time process, it cannot learn an underlying ODE model.
Rather, we focus on predictions over fixed time intervals ($\Delta t = 0.01$).
On the other hand, given a test initial condition,
the NN method simply selects the closest initial condition from the training set
and use its associated trajectory as a prediction.
The numerical results are given in Table \ref{tbl:RBCres_nn_ons}
and Table \ref{tbl:RBCres_LSTM}. We see OnsagerNet out-performs NN method for
all cases except for $m=3$ and prediction time $t=99$. By using more PCs,
the OnsagerNet gives better results, but NN does not. From Table
\ref{tbl:RBCres_LSTM}, we also observe that LSTM does not give better results by
using more PCs.
We note that both LSTM and NN can only serve as predictive tools for trajectories
but does not give rise to a reduced dynamical model in the form of differential equations.

\begin{figure}[!ht]
\centering
  \subfigure[Trajectory 4, $m=3$]{
    \includegraphics[width=0.47\textwidth]{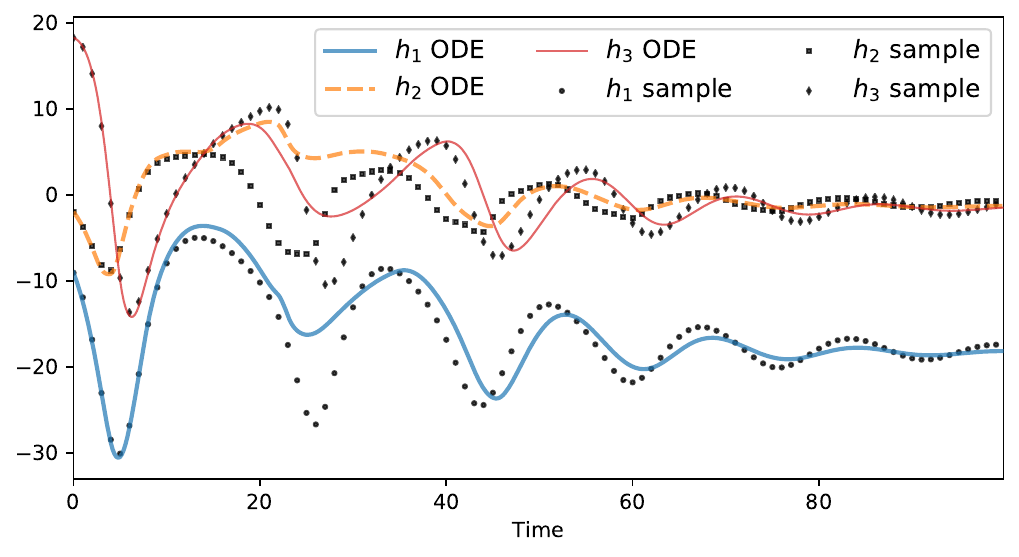}}
\subfigure[Trajectory 4, $m=7$]{
    \includegraphics[width=0.47\textwidth]{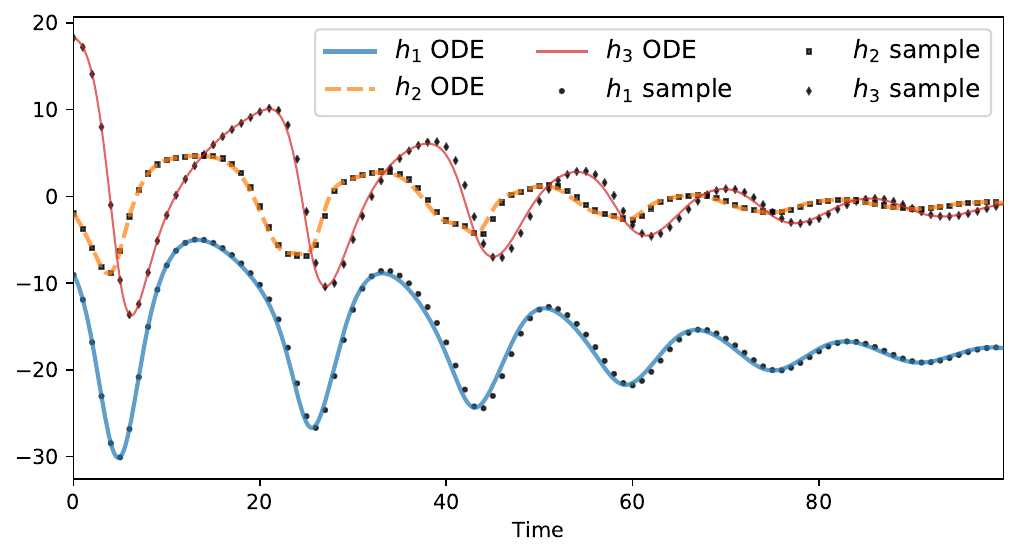}}
  \subfigure[Trajectory 89, $m=3$]{
  \includegraphics[width=0.47\textwidth]{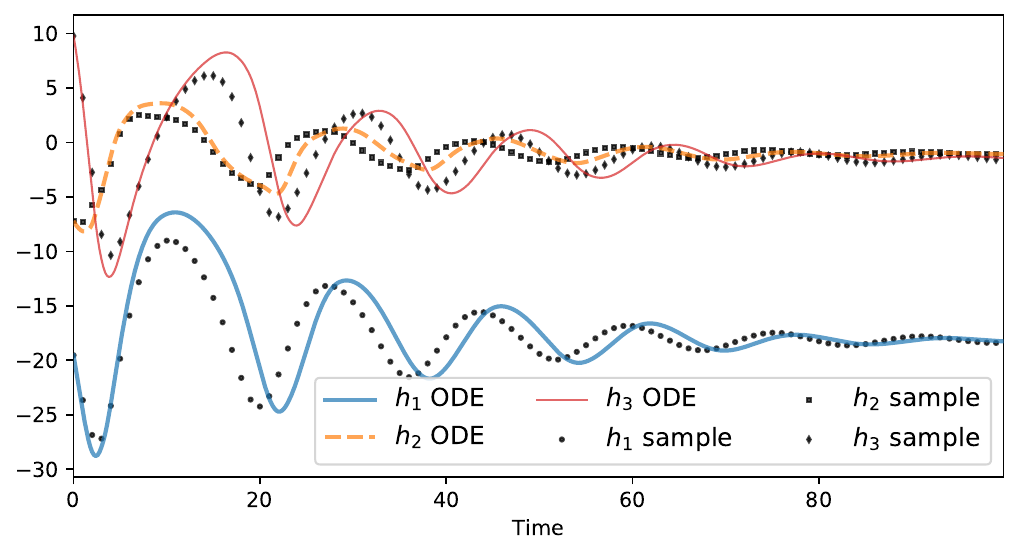}}
  \subfigure[Trajectory 89, $m=7$]{
  	\includegraphics[width=0.47\textwidth]{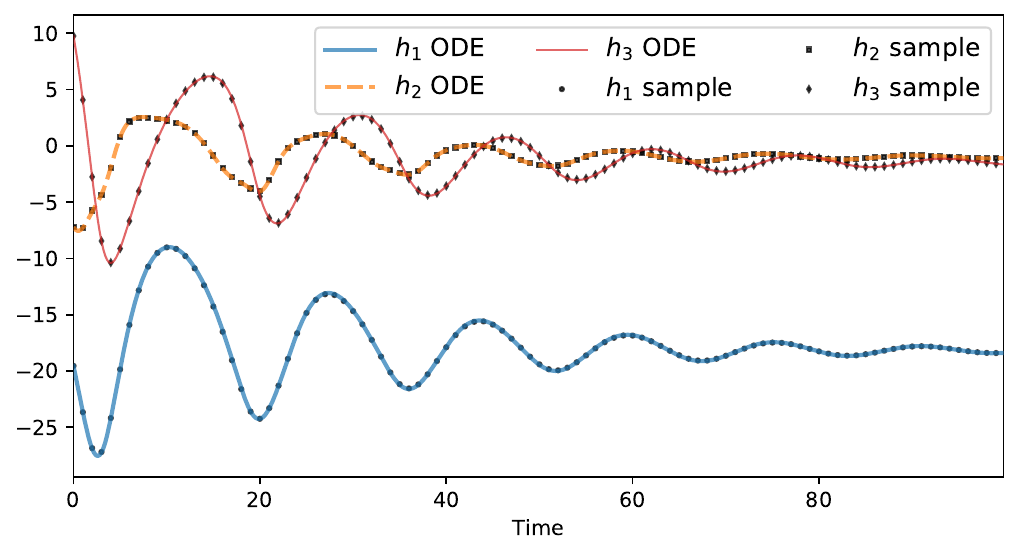}}
  \caption{\label{fig:RBCtroj}The first 3 principal components of one
    trajectory (Trajectory 4) in training set and one (Trajectory 89) in
    testing set for
    the RBC problem with $r=28$ and the corresponding simulation
    results of learned
    reduced models by OnsagerNets with $m=3$ and $7$, respectively.}
\end{figure}

To get an overview of the vector field that drives the learned ODE system
represented by OnsagerNet, we draw 2-dimensional projections of phase portraits
following several representative trajectories in Fig. \ref{fig:phd2dr28},
from which we see there are four
stable fixed points, two of which have relatively larger attraction basins,
the other two have very small attraction basins.
We numerically verified they are all linearly stable by checking the eigenvalues of Jacobian matrices of the learned system at those points.
The two fixed points
with larger attraction basin are very similar to those appearing in the Lorenz system
with $r<24.06$, which corresponds to the two fixed points resulting from the first pitchfork bifurcation, see, e.g. \citep[$q_+, q_-$ in Fig 2.1]{zhou_study_2010}, and
Fig. \ref{fig:Lorenz_r16}.

We also test the learned dynamics using
OnsagerNet by simulating 100 new numerical trajectories with random initial conditions
in the phase space.
We observe that they all have negative Lyapunov exponents, which suggests that the learned dynamics have no chaotic solutions.
This result is consistent with the original RBC system, which also does not exhibit chaos at these parameter values.
In this sense, OnsagerNet also preserves the qualitative properties of the full dynamics, despite being embedding in a much lower dimensional space.

\begin{figure}[!ht]
	\centering
	\includegraphics[width=0.8\textwidth]{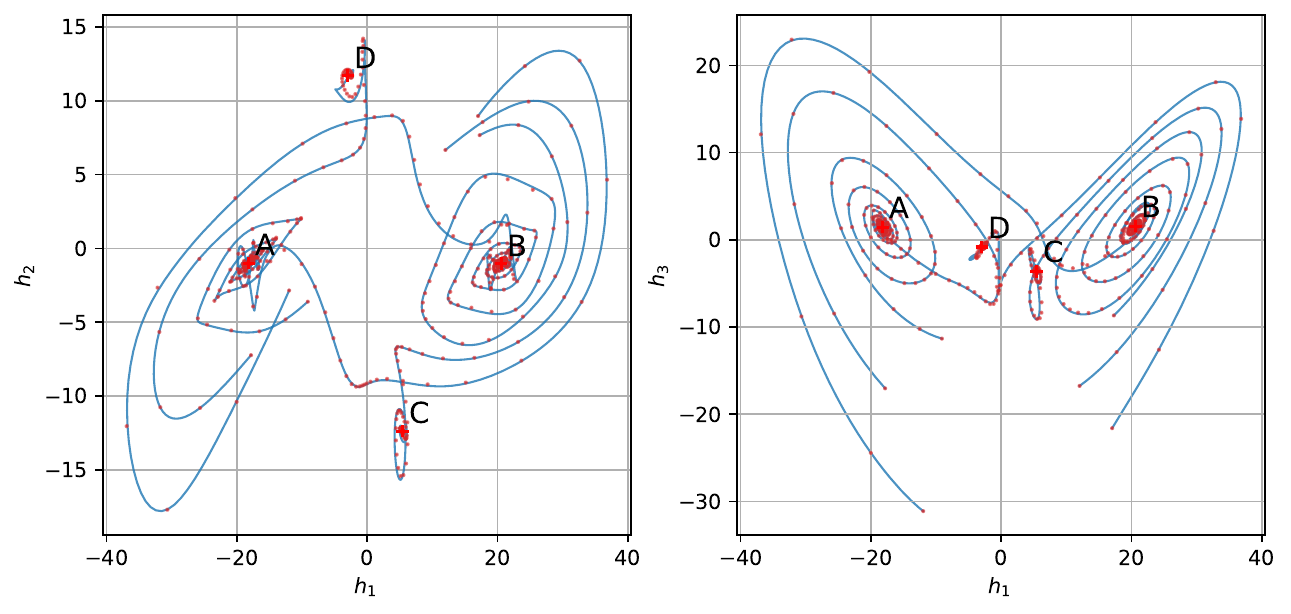}
	\caption{\label{fig:phd2dr28} Six representative trajectories from the
		sample data (red dots) and the corresponding ODE solutions (solid
		blue curves) evolved from the same initial values using learned ODE
		system for the RBC problem with $r=28$, $m=7$. The four stable critical
		points(red crosses) of the learned ODE system are numerically
		calculated and plotted.  Left: projection to $(h_1, h_2)$ plane;
		Right: projection to $(h_1, h_3)$ plane.}
\end{figure}

{\bf{The learned free energy function.} }
In Fig. \ref{fig:RBCpot}, we show the iso-surfaces of the learned
free energy function in the reduced OnsagerNet model for the RBC problem. When PCA
data is used, with 3 hidden variables, the free energy learned is irregular, which is very different to a physical free energy function one expect for a smooth dynamical system. As the
number of hidden variables is increased, the iso-surfaces of learned free energy function become more and more ellipsoidal, which means the free energy is more akin to a positive definite quadratic form. This is consistent with the exact PDE model. The use of AE in low dimensional case also helps to learn a better free energy function. This again highlights the preservation of physical principles in the OnsagerNet approach.

\begin{figure}[!ht]
   \centering
   \subfigure[PCA $m=3$]{
  \includegraphics[trim=15 20 50 60, clip,
  width=0.3\textwidth]{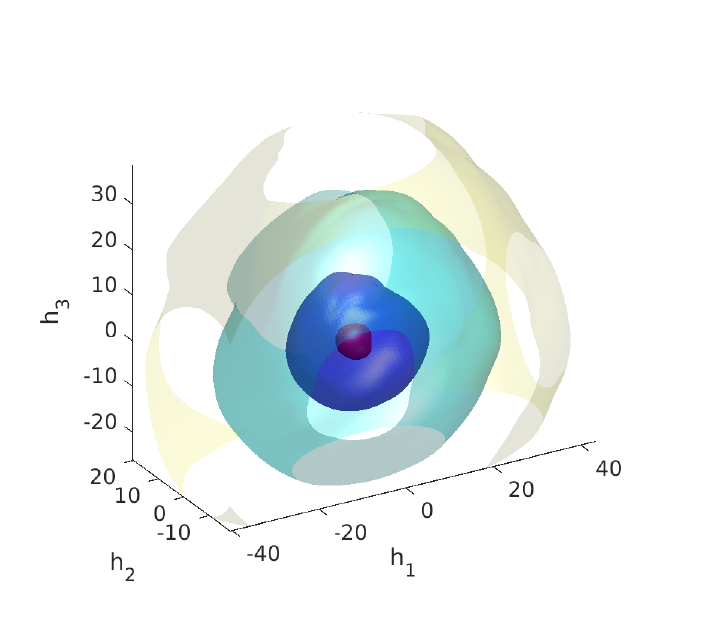}}
  \subfigure[PCA $m=7$]{\includegraphics[trim=15 20 50 60, clip,
  width=0.3\textwidth]{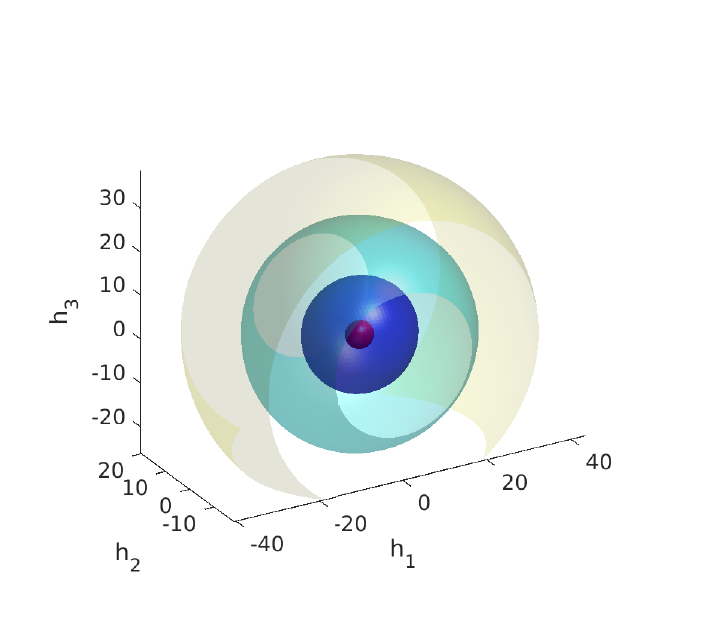}}
  \subfigure[Autoencoder + OnsagerNet
  $m=3$]{\includegraphics[trim=15 20 50 60, clip,
  width=0.3\textwidth]{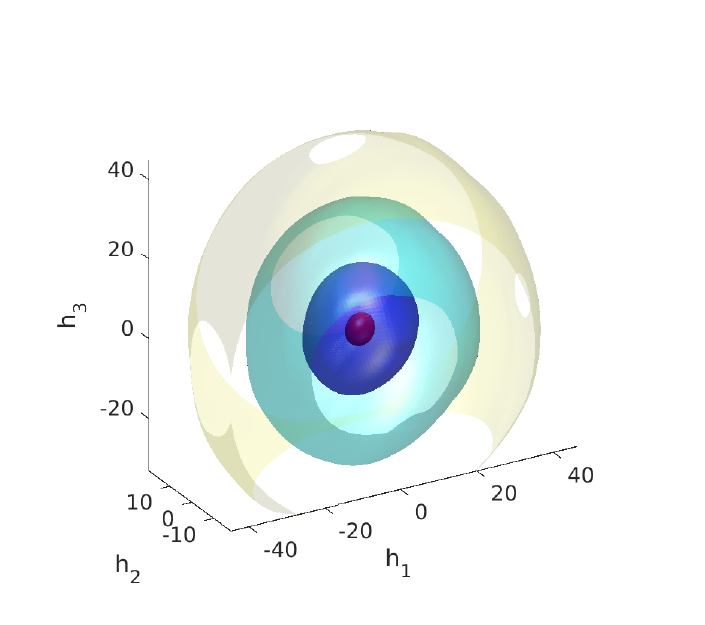}}
  \caption{\label{fig:RBCpot}The learned free energy by OnsagerNet with
    PCA and an autoencoder for the RBC problem with $r=28$. The dependence of the energy function on the first three
    principal components are shown.}
\end{figure}

\subsubsection{A higher Rayleigh number case r=84}

For $r=28$, the RBC problem has only fixed points as attractors.
To check the robustness of OnsagerNet to approximate dynamics with different properties, we present results of a higher Rayleigh number case with $r=84$.  The
corresponding Rayleigh number is $55,230.95$.
In this case, the RBC problem has four stable limit cycles.
However, starting from initial conditions considered by Lorenz's reduction, the solution needs to
evolve for a very long time to get close to the limit cycles.
Thus, whether or not the limit sets are limit cycles are not very clear
based on only observations of the sample data.
An overview of learned dynamics for $m=11$ is shown in Fig. \ref{fig:phd2dr84L},
where four representative trajectories close to the limit cycles
 together with critical points (saddles) and limit cycles calculated
from the learned model are plotted.
As before, from this figure we see limit cycles are accurately predicted by
reduced order model learned using OnsagerNet. Meanwhile, the saddles can be calculated from the
learned OnsagerNet.

\begin{figure}[htbp]
    \centering
    \includegraphics[width=0.8\textwidth]{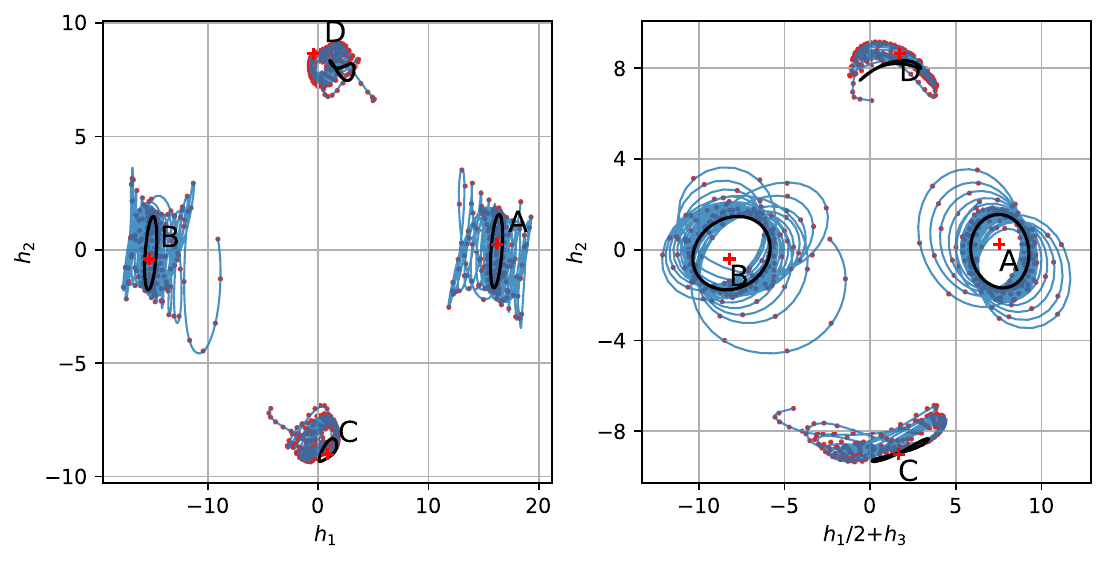}
    \caption{\label{fig:phd2dr84L} Four representative trajectories from
        the sample data (red dots) and the corresponding ODE solutions
        (solid blue curves) evolved from the initial values using learned ODE
        system ($m=11$) for the case $r=84$. The four unstable critical
        points (red
        crosses) and four stable limit cycles (black curves) of the learned
        ODE system are also plotted.  Left: projection to $(h_1, h_2)$
        plane; Right: projection to $(h_1/2+h_3, h_2)$ plane. }
\end{figure}

\begin{figure}[htbp]
    \centering
    \includegraphics[width=0.7\textwidth]{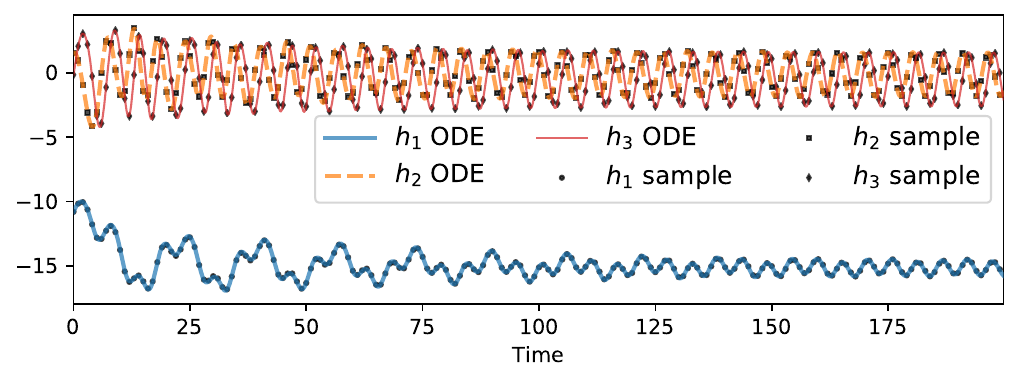}
    \caption{\label{fig:RBCtraj_r84}The first 3 principal components of an
        exact trajectory and the corresponding simulation results of learned
        reduced models by OnsagerNets trained using $11$ principal
        components for the RBC problem with $r=84$.}
\end{figure}

A typical trajectory in test set and
corresponding results by the learned dynamics are plotted in
Fig. \ref{fig:RBCtraj_r84}. These results supports the observation that OnsagerNet achieves
both quantitative trajectory accuracy and faithfully captures asymptotic behaviors of the
original high-dimensional system.

To check whether the learned reduced order model has
chaotic solutions, we carried out a numerical estimate of the largest Lyapunov indices of the trajectories in very long time simulations. Here, we did not find
any trajectory with a positive Lyapunov index, which suggest that for the
parameter setting used, the RBC problem has no chaotic solution, at least for
the type of initial conditions considered in this paper.

{
\subsection{Application to Rayleigh-B\'{e}nard convection in a wide range of
Rayleigh number\label{sec:nr}}

In this section, we build a reduced model that operates over a wide range of
Rayleigh
numbers, and hence we sample trajectories corresponding to ten different
$r$ values $\{8, 16, 22, 28, 42, 56, 70, 76, 80, 84\}$ in the range $[8, 84]$.
Parameters $\nu$, $\kappa$, $\Gamma$
are
fixed, and we vary $\alpha_0$ to obtain different $r$ values.
In all cases, the Prandtl number is fixed at $Pr =10/3$. The data generating
procedure is same to the fixed $r$ case, but since for multiple $r$ case,
the dimension of fast manifold is relatively high, we do not use the first 39
pairs of solution snapshots in the trajectories.

The first step is to construct the low-dimensional generalized coordinates
using PCA.
To account for varying Rayleigh numbers, we perform a common PCA on
trajectories with different $r$ values.
Consequently, we seek an OnsagerNet parameterization with quantities
$V,\tilde{M},\tilde{W}$ independent of the Rayleigh number $r$.
Examining~\eqref{nse1}-\eqref{nse2}, the $r$ dependence may
be regarded as strength of the external force, which is linear.
Thus, we seek external force $f$ of the OnsagerNet as an affine
function of $r$. But, different to the fixed $r$ case, here we assume
$f$ is a nonlinear function of $h$.

\subsubsection{Quantitative trajectory accuracy.}

We first show that despite being low-dimensional, the learned OnsagerNet
preserves trajectory-wise accuracy when compared with the full RBC equations.
We summarize the MSE between trajectories generated from the learned
OnsagerNet and the true RBC equations for different times in
Table~\ref{tbl:RBCres_nr}.
Observe that the model using only 7 coordinates ($m=7$)
gives good short time ($t=1$) prediction,
but has a relatively large error for the long time
$(t = 60)$ prediction. By increasing $m$ to $9$ and $11$, both the
short time and long time prediction accuracy increase.
We also tested generalization in terms of Rayleigh numbers by withholding data
for
$r=28,84$ and testing the learned models in these regimes
(Fig.~\ref{fig:phd2d}).
In all cases, the OnsagerNet dynamics remain a quantitatively accurate
reduction of the RBC system.

\begin{table}[tbhp]
    \begin{center}
        \caption{\label{tbl:RBCres_nr}Accuracy of learned models for RBC
        problem $r\in[8, 84]$.}
        \begin{tabular}{llllll}
            \hline\hline
            Dimension & $\tmop{MSE}_{\tmop{train}}$
            & $\tmop{MSE}_{\tmop{test}}$ & $E^{pred,rel}_{t = 1}$ &
            $E^{pred,rel}_{t
                = 60}$
            \\
            \hline
            $m$=7 $\ $ & 1.55$\times 10^{-2}\quad $ & 3.43$\times 10^{-2}\quad
            $ &
            7.07$\times
            10^{-4}\quad $ &
            6.04$\times 10^{-2}\quad $ \\
            $m$=9 & 2.63$\times 10^{-3}$ & 4.36$\times 10^{-3}$ & 6.84$\times
            10^{-4}$ &
            1.05$\times 10^{-2}$ \\
            $m$=11 & 2.14$\times 10^{-3}$ & 4.12$\times 10^{-3}$ & 5.12$\times
            10^{-4}$ &
            8.45$\times 10^{-3}$ \\
            \hline
        \end{tabular}
    \end{center}
\end{table}

\begin{figure}[tbhp]
    \centering
    \includegraphics[trim=0 10 0 10, clip, width=0.9\linewidth]
    {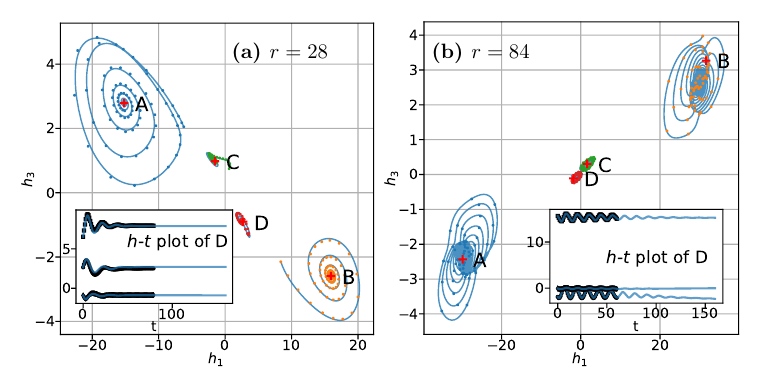}
    \caption{\label{fig:phd2d} Some representative trajectories from the
        sample data (colored dots) and the corresponding solutions of learned
        11-dimensional OnsagerNet ODEs (solid blue curves)
        from the same initial values for the RBC problem with $r=28,84$.
        The red crosses are fixed points calculated from the learned ODE
        systems.
    }
\end{figure}
\begin{figure}[tbhp]
    \centering
    \includegraphics[width=0.45\linewidth,]
    {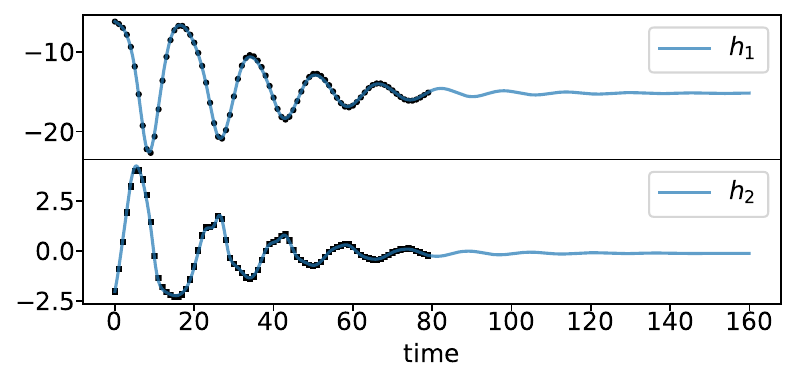}
    \includegraphics[width=0.45\linewidth,]
    {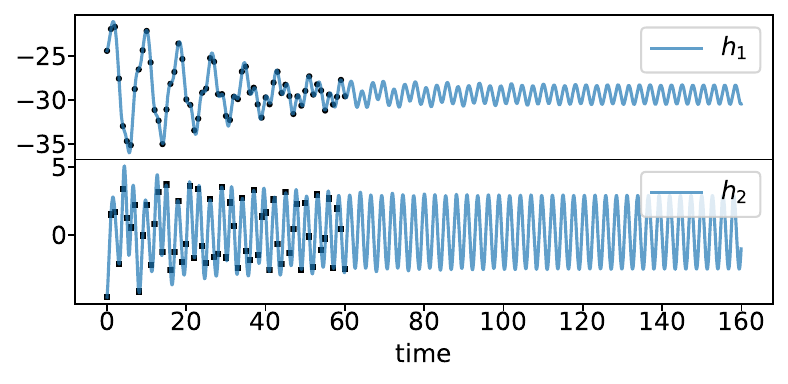}
    \caption{\label{fig:longtime}
        The long time comparison of trajectories
        originating from an initial condition converging to $A$
        for $r=28$(left panel), $84$(right panel).
        Black dots are RBC data and
        colored curves are predictions by the learned OnsagerNet.
    }
\end{figure}

\subsubsection{Qualitative reproduction of phase space structure.}
Next, we show that besides quantitative trajectory accuracy,
the qualitative aspects of the RBC, such as long time stability and
the nature of the attractor sets, are also adequately captured.
This highlights the fact that the approximation power
of OnsagerNet does not come at a cost of physical stability and relevance.

To get an overview of the vector field that drives the learned ODE system,
we draw 2-dimensional projections of phase portraits for several
representative trajectories at $r=28, 84$ in Fig. \ref{fig:phd2d}.
The data for these Rayleigh numbers are not included in the training set.
For the $r=28$ case, we observe four stable fixed points.
The two fixed points with larger attraction basins are similar to those
appearing in the Lorenz system with $r<24.06$, which corresponds to the two
fixed points
resulting from the first pitchfork bifurcation, see e.g., \cite[$q_+, q_-$
in Fig 2.1]{zhou_study_2010}.
Note that the Lorenz '63 model with $r=28$ is chaotic, but
the original RBC problem and our learned model have only fixed points as
attractors.
For $r=84$, the four attractors in the RBC problem become more complicated.
Starting from Lorenz-type initial conditions, the solutions need to
evolve for a very long time to get close to these attractors.
The results in Fig. \ref{fig:phd2d} show that the learned low-dimensional
models can accurately capture those complicated behavior.
Due to the fact that the $A, B$ attractors have larger attraction regions
than $C,D$, we have more $A,B$ type trajectories than $C,D$ type ones.
Thus, the vector field near fixed points $A,B$ are learned with higher
quantitative
accuracy (see Fig. \ref{fig:longtime}),
while $C,D$ type trajectory accuracy holds only for short times.
Nevertheless, in all cases, the asymptotic qualitative behavior of the
trajectories and
the attractor sets are faithfully captured.


As motivated earlier, each term in the OnsagerNet parameterization
has clear physical origin, and hence once we learn an accurate model,
we can study the behavior of each term to infer some physical characteristics
of the reduced dynamics.
In Fig.~\ref{fig:RBCpot_nr}, we plot the free energy $V$ in the learned
OnsagerNet
model
and compare it to the RBC energy function $\frac{1}{2}\|u\|^2 +
\frac{1}{2}\|\theta\|^2$ projected to the same reduced coordinates.
We observe that the shapes are similar with ellipsoidal iso-surfaces, but the
scales vary.
This highlights the preservation of physical structure in our approach.

\begin{figure}[tbhp]
    \centering
    \includegraphics[trim=20 40 20 40, clip,
    width=0.7\linewidth, 
    ]
    {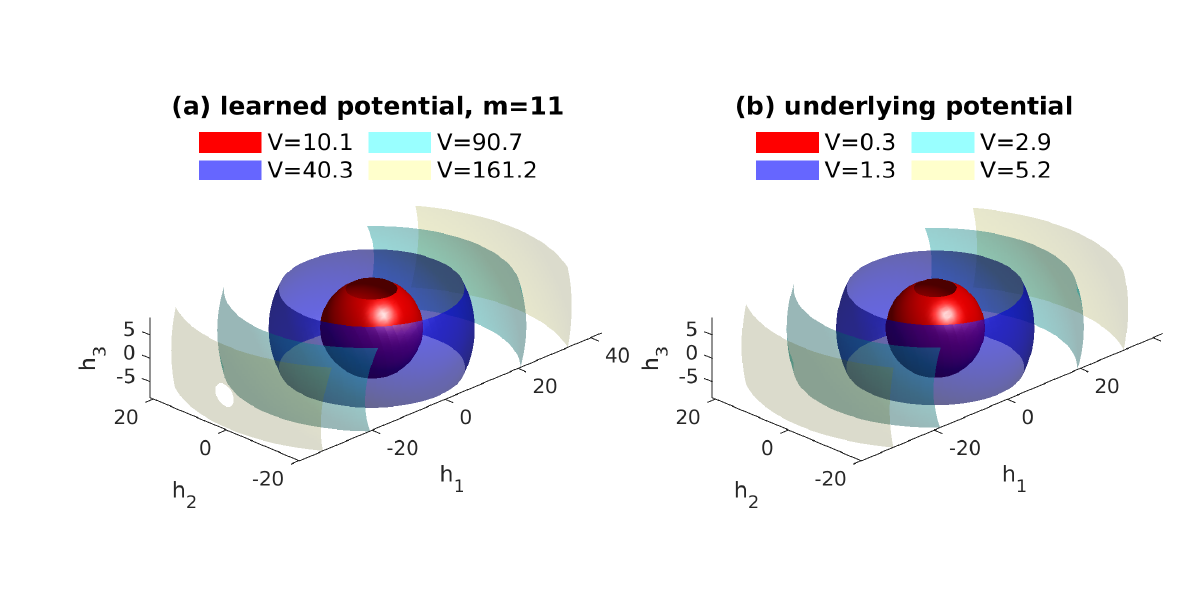}
    \caption{\label{fig:RBCpot_nr}The underlying and learned potential by
        OnsagerNet for the
        RBC problem. The iso-surfaces in first three
        principal component dimensions are shown.}
\end{figure}

Besides the learned potential, we can also study the diffusion matrix $M$, which
measures the rate of energy dissipation as a function of the learned
coordinates $h$.
In Fig.~\ref{fig:MMatrix}, we plot the eigenvalues of $M$
(which characterize dissipation rates along dynamical modes)
along the line from point $A$ to $B$ at $r=28$,
where we observe strong dependence on $h$.
This verifies that the OnsagerNet approach is different from
linear Galerkin type approximations
where the diffusive matrix $M$ is constant.
\begin{figure}[tb]
    \centering
    \includegraphics[width=0.6\linewidth,
    ]{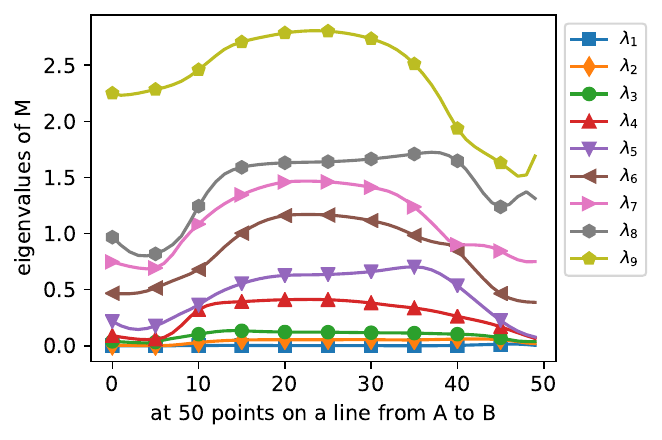}
    \caption{\label{fig:MMatrix}
        The eigenvalues of the diffusion matrix $M(h)$
        representing local dissipation rates
        along a line from $A$ to
        $B$ in a $r=28, m=9$ case.
        We can see clear deviation from a linear model where $M$ is constant.
    }
\end{figure}

}

In summary, using the OnsagerNet to learn a reduced model of the RBC equations,
we gained two main insights. First, it is possible, under the considered
conditions, to obtain an ODE model of relatively low dimension (3-11) that can
already capture the qualitative properties (energy functions, attractor sets)
while maintaining quantitative reconstruction accuracy {for a large
range
of Rayleigh number}. This is in line with Lorenz' intuition in building the '63
model. Second, unlike Lorenz highly truncated model, we show that under the
parameter settings investigated, the learned OnsagerNet, while having complex
dynamics, does not have chaotic behavior.

\section{Discussion and summary}

We have presented a systematic method to construct (reduced) stable and interpretable ODE systems by a machine learning approach based on a generalized Onsager principle.
Comparing to existing methods in the literature, our method has several distinct features.
\begin{itemize}
	\item Comparing to the non-structured
	machine learning approaches (e.g. \citep{bongardAutomatedReverseEngineering2007,brunton_discovering_2016,raissi_multistep_2018}, etc.), the dynamics learned by our approach have precise physical structure, which not only ensures
	the stability of the learned models automatically, but also gives physically interpretable quantities, such as free energy, diffusive and conservative terms, etc.
	\item Comparing to the existing structured approach, e.g. the Lyapunov function approach \citep{kolterLearningStableDeep2019}, symplectic structure approach \citep{jin_symplectic_2020,zhongDissipativeSymODENEncoding2020}, etc., our method, which relies on a principled generalization of the already general Onsager principle, has a more flexible structure incorporating both conservation and dissipation in one equation, which make it suitable for a large class of problems involving forced dissipative dynamics. In particular, the OnsagerNet structure we impose does not sacrifice short-term trajectory accuracy, but still achieves long term stability and qualitative agreement with the underlying dynamics.
	\item Different to the  linear multi-step method embedded training \citep{raissi_multistep_2018,xieNonintrusiveInferenceReduced2018} and recurrent neural networks \citep{panDataDrivenDiscoveryClosure2018,wangRecurrentNeuralNetwork2020}, we use multiple Runge-Kutta steps embedded in loss function for training. This can handle large and variable data sampling intervals in a straightforward manner.
	\item We proposed an isometry-regularized autoencoder to find the
	slow manifold for given trajectory data for dimension reduction, this is different to traditional feature extraction methods such as proper orthogonal decomposition (POD) \citep{lumley_stochastic_1970,holmes_turbulence_1996}, dynamic mode decomposition (DMD)
	and Koopman analysis \cite{schmid_dynamic_2010,rowley_spectral_2009,takeishi_learning_2017}, which
	only find important \emph{linear} structures.
	Our method is also different to the use of an autoencoder with sparse identification of nonlinear dynamics (SINDy) proposed by \citep{championDatadrivenDiscoveryCoordinates2019} where sparsity of the dynamics to be learned is used to regularize the autoencoder. The isometric regularization allows the autoencoder to be trained separately or jointly
	with ODE nets. Moreover, it is usually not easy to ensure long-time stability using these type of methods, especially when the dimension of the learned dynamics increases.
	\item As a model reduction method, different to the closure approximations approach \citep{wanDataassistedReducedorderModeling2018} that based on Mori-Zwanzig theory \citep{chorinOptimalPredictionMoriZwanzig2000}, which needs to work with an explicit mathematical form of underlying dynamical system and usually leads to non-autonomous differential equations, our approach uses only sampled trajectory data to learn autonomous dynamical systems, the learned dynamics can have very good quantitative accuracy and can be readily be further analyzed with traditional ODE analysis methods, or be used as surrogate models for fast online computation.
\end{itemize}

We have shown the versatility of OnsagerNet by applying it to
identify closed and non-closed ODE dynamics: the nonlinear pendulum and Lorenz system with chaotic attractor, and learn reduced order
models with high fidelity for the classical meteorological dynamical systems: the Rayleigh-B\'{e}nard convection problem.

While versatile, the proposed method can be further improved or extended
in several ways. For example, in the numerical results of the RBC problem, we see the accuracy is
limited by given sample data when we use enough hidden dimensions. One can use an active sampling strategy (see e.g. \cite{zhang2019active}, \cite{han_uniformly_2019}) together with OnsagerNet to further improve the accuracy, especially near saddle points.
Furthermore, for problems where transient solutions are of interest but lie on a very high-dimensional space, then directly learning PDEs might be a better approach than learning ODE systems. The learning of PDE using OnsagerNet can be accomplished by incorporating the differential operator filter constraints proposed in \citep{longPDENetLearningPDEs2018} into its components, e.g. the diffusive and conservative terms. Finally, for problems where sparsity (e.g. memory efficiency) is more important than accuracy, or a balance between these two are needed, one can either add a $l^1$ regularization on the weights of OnsagerNets into the training loss function or incorporate sparse identification methods in the OnsagerNet approach.

\section*{Acknowledgments}
We would like to thank Prof. Chun Liu, Weiguo Gao, Tiejun Li, Qi Wang,
and Dr. Han Wang for helpful discussions.
The work of HY and XT are supported by
NNSFC Grant 91852116, 11771439 and China Science Challenge Project
no. TZ2018001.
The work of WE is supported in part by a gift from iFlytek to Princeton University.
The work of QL is supported by the start-up grant under the NUS PYP programme.
The computations were partially performed on the LSSC4 PC cluster of State Key Laboratory of Scientific and Engineering Computing (LSEC).
\appendix
\section*{Appendix}

\section{Classical dynamics as generalized Onsager
principle}
\label{apd:gdyn}

Previously we showed how the $h$ dependence arise from projecting high-dimensional dynamics to one on reduced coordinates. There, it is assumed that the high-dimensional dynamics obey the generalized Onsager principle with constant diffusive and conservative terms. Here, we show how this assumption is sensible, by giving general classical dynamical systems that has such a representation.

\begin{theorem}
  Classical Hamilton system with Hamilton given by
  \[ H (x, q) = U (x) + \frac{1}{2 m} q^2 \] can be written in the form
  of Eq. {\eqref{eq:GOPmform}}.

  \begin{proof}
    The Hamilton equation for given $H$ is
    \[ \pt{x} = \frac{\partial H}{\partial q} = \frac{q}{m}, \]
    \[ \pt{q} = - \frac{\partial H}{\partial x} = - \nabla_x U. \]
    Taking $M = 0$, $W = \left(\begin{array}{cc}
      0 & 1\\
      - 1 & 0
    \end{array}\right)$, we have
    \[ W \left(\begin{array}{c}
         \pt{x}\\
         \pt{q}
       \end{array}\right) = - \left(\begin{array}{c}
         \nabla_x H\\
         \nabla_q H
       \end{array}\right) . \]
  \end{proof}
\end{theorem}

\begin{theorem}
  The deterministic damped Langevin dynamics
  \begin{equation}
    m \frac{d^2 x}{d t^2} = - \gamma \pt{x} - \nabla_x U (x) \label{eq:dLagevin}
  \end{equation}
  can be written in the form of Eq. {\eqref{eq:GOPmform}}.

  \begin{proof}
    Denote $v = \pt{x},$ then we have
    \[ \pt{x}  = v, \]
    \[ m \pt{v} + \gamma \pt{x}  = - \nabla_x U (x) . \]
    By letting $M = \left(\begin{array}{cc}
      \gamma & 0\\
      0 & 0
    \end{array}\right),$ $W = \left(\begin{array}{cc}
      0 & m\\
      - m & 0
    \end{array}\right)$, we have
    \[ (M + W) \left(\begin{array}{c}
         \pt{x}\\
         \pt{v}
       \end{array}\right) = - \left(\begin{array}{c}
         \nabla_x U (x)\\
         mv
       \end{array}\right) . \]
    This is of form {\eqref{eq:GOPmform}} with a new energy (total
    energy)
    \[ V (x, v) = U (x) + \frac{m}{2} v^2 . \]
  \end{proof}
\end{theorem}

\begin{theorem}
  The dynamics described by generalized Poisson brackets, defined below,
  can be written in form
  {\eqref{eq:GOPdform}}. In the Poisson bracket approach, the system is
  described by generalized coordinates $(q_1, \ldots, q_n)$ and
  generalized momentum $(p_1, \ldots, p_n)$. Denote the Hamiltonian of the system
  as $H (q_1, \ldots, q_n ; p_1, \ldots, p_n)$. Then, the dynamics of the
  system is described by the equation {\citep{beris_thermodynamics_1994}}
  \begin{equation}
    F_t = \{ F, H \} - [F, H_{}] \label{eq:PoissonBracketDyn}
  \end{equation}
  where F is an arbitrary functional depending on the system variables. The
  reversible and irreversible contributions to the system are represented by
  the Poisson bracket $\{ \cdot, \cdot \}$ and the dissipation bracket
  $[\cdot, \cdot]$, respectively, which are defined as
  \[ \{ F, H \} = \sum_{i = 1}^n \frac{\partial F}{\partial q_i}
    \frac{\partial H}{\partial p_i} - \frac{\partial H}{\partial q_i}
    \frac{\partial F}{\partial p_i}, \]
  \[ [F, H] = J_F MJ_H^T, \quad J_{\{F,H\}} = \left[ \frac{\partial }{\partial q_1},
     \ldots \frac{\partial }{\partial q_n}, \frac{\partial }{\partial p_1},
     \ldots, \frac{\partial }{\partial p_n} \right]\{F,H\}, \]
  where $M$ is symmetric positive semi-definite.

  \begin{proof}
    Denote $(h_1, \ldots, h_m) = (q_1, \ldots q_{n,} p_1, \ldots, p_n)$, with
    $m = 2 n$. Equation {\eqref{eq:PoissonBracketDyn}} can be written as
    \[ F_t = (\nabla_h F)^T \pt{h} = (\nabla_h F)^T \left(\begin{array}{cc}
         0 & I_n\\
         - I_n & 0
       \end{array}\right) \nabla_h H - (\nabla_h F)^T M \nabla_h H. \]
    By taking $F = (h_1, \ldots, h_m)$, such that $\nabla_h F = I_m$, we
    obtain immediately
    \[ \pt{h} = - W \cdot \nabla_h H - M \cdot \nabla_h H = (M + W) (- \nabla_h
       H), \]
    where $W = \left(\begin{array}{cc}
      0 & - I_n\\
      I_n & 0
    \end{array}\right)$ is an anti-symmetric matrix, $M$ is a symmetric
    positive semi-definite matrix, which is of form {\eqref{eq:GOPdform}}.
  \end{proof}
\end{theorem}

\section{Basic properties of non-symmetric positive definite matrices}
\label{apd:psdmatrix}

We present some basic properties of real non-symmetric positive-definite
matrices required to show our results.

For a $n\times n$ real matrix, we call it {\it{positive definite}}, if
there exist a constant $\sigma>0$, such that $x^T A x \ge \sigma x^T x$,
for any $x\in \bbR^n$.   If $\sigma=0$ in
the above inequality, we call it {\it{positive semi-definite}}.

For any real matrix $A$, it can be written as a sum of a symmetric part
$\frac{A+A^T}{2}$ and a skew-symmetric part $\frac{A-A^T}{2}$. For the
skew-symmetric part, we have $x^T \frac{A-A^T}{2} x = 0$, for any
$x\in \bbR^n$.  So a non-symmetric matrix is positive (semi-) positive
definite if and only if its symmetric part is positive (semi-)definite.

The following theorem is used in deriving alternative forms of the generalized Onsager dynamics.

\begin{theorem}
	\label{thm:psdmatrix}
	1) Suppose $A$ is a positive semi-definite real matrix, then the
    real parts of its eigenvalues are all non-negative.  2) The inverse
    of a non-singular positive semi-definite is also positive
    semi-definite.
\end{theorem}

\begin{proof}
  Suppose that $\lambda = \eta + i \mu$ is an eigenvalue of $A$, the
  corresponding eigenvector is $z = x + iy$, then
  \[ Ax + iAy = (\eta x - \mu y) + i (\eta y + \mu x). \] So
  \[ \left(\begin{array}{cc}
             A & 0\\
             0 & A
	\end{array}\right) \left(\begin{array}{c}
                               x\\
                               y
	\end{array}\right) = \left(\begin{array}{cc}
                                 \eta I & 0\\
                                 0 & \eta I
	\end{array}\right) \left(\begin{array}{c}
                               x\\
                               y
	\end{array}\right) + \left(\begin{array}{cc}
                                 0 & - \mu I\\
                                 \mu I & 0
	\end{array}\right) \left(\begin{array}{c}
                               x\\
                               y
	\end{array}\right) . \]
Left multiply the equation by $(x^T, y^T)$ to get (this is the real part
of $\bar{z}^T Az = \bar{z}^T \eta z$ )
\[ x^T Ax + y^T Ay = \eta x^T x + \eta y^T y = \eta (x^T x +
  y^T y) . \] If $A$ is positive semi-definite, then we obtain
$\eta \geqslant 0$, 1) is proved.

Now suppose that $A$ positive semi-definite and invertible, then for any
$x \in \mathbb{R}^n,$ we can define $y$ by $Ay = x$ to get
\[ x^T A^{- 1} x = y^T A^T A^{- 1} Ay = y^T A^T y = y^T Ay \geqslant
  0. \] The second part is proved.
\end{proof}

Note that the converse of the first part in Theorem \ref{thm:psdmatrix}
is not true. An simple counterexample is $A = \left(\begin{array}{cc}
                                                     3 & 2\\
                                                     - 2 & - 1
\end{array}\right)$, whose eigenvalues are $\lambda_{1, 2} = 1$. But the
eigenvalues of its symmetric part are $\lambda_{1, 2} = - 1, 3$.

\bibliography{OnsagerNet}

\begin{thebibliography}{69}%
\makeatletter
\providecommand \@ifxundefined [1]{%
 \@ifx{#1\undefined}
}%
\providecommand \@ifnum [1]{%
 \ifnum #1\expandafter \@firstoftwo
 \else \expandafter \@secondoftwo
 \fi
}%
\providecommand \@ifx [1]{%
 \ifx #1\expandafter \@firstoftwo
 \else \expandafter \@secondoftwo
 \fi
}%
\providecommand \natexlab [1]{#1}%
\providecommand \enquote  [1]{``#1''}%
\providecommand \bibnamefont  [1]{#1}%
\providecommand \bibfnamefont [1]{#1}%
\providecommand \citenamefont [1]{#1}%
\providecommand \href@noop [0]{\@secondoftwo}%
\providecommand \href [0]{\begingroup \@sanitize@url \@href}%
\providecommand \@href[1]{\@@startlink{#1}\@@href}%
\providecommand \@@href[1]{\endgroup#1\@@endlink}%
\providecommand \@sanitize@url [0]{\catcode `\\12\catcode `\$12\catcode
  `\&12\catcode `\#12\catcode `\^12\catcode `\_12\catcode `\%12\relax}%
\providecommand \@@startlink[1]{}%
\providecommand \@@endlink[0]{}%
\providecommand \url  [0]{\begingroup\@sanitize@url \@url }%
\providecommand \@url [1]{\endgroup\@href {#1}{\urlprefix }}%
\providecommand \urlprefix  [0]{URL }%
\providecommand \Eprint [0]{\href }%
\providecommand \doibase [0]{https://doi.org/}%
\providecommand \selectlanguage [0]{\@gobble}%
\providecommand \bibinfo  [0]{\@secondoftwo}%
\providecommand \bibfield  [0]{\@secondoftwo}%
\providecommand \translation [1]{[#1]}%
\providecommand \BibitemOpen [0]{}%
\providecommand \bibitemStop [0]{}%
\providecommand \bibitemNoStop [0]{.\EOS\space}%
\providecommand \EOS [0]{\spacefactor3000\relax}%
\providecommand \BibitemShut  [1]{\csname bibitem#1\endcsname}%
\let\auto@bib@innerbib\@empty
\bibitem [{\citenamefont {Bongard}\ and\ \citenamefont
  {Lipson}(2007)}]{bongardAutomatedReverseEngineering2007}%
  \BibitemOpen
  \bibfield  {author} {\bibinfo {author} {\bibfnamefont {J.}~\bibnamefont
  {Bongard}}\ and\ \bibinfo {author} {\bibfnamefont {H.}~\bibnamefont
  {Lipson}},\ }\bibfield  {title} {\bibinfo {title} {Automated reverse
  engineering of nonlinear dynamical systems},\ }\href
  {https://doi.org/10.1073/pnas.0609476104} {\bibfield  {journal} {\bibinfo
  {journal} {Proc. Natl. Acad. Sci.}\ }\textbf {\bibinfo {volume} {104}},\
  \bibinfo {pages} {9943} (\bibinfo {year} {2007})}\BibitemShut {NoStop}%
\bibitem [{\citenamefont {Schmidt}\ and\ \citenamefont
  {Lipson}(2009)}]{schmidt_distilling_2009}%
  \BibitemOpen
  \bibfield  {author} {\bibinfo {author} {\bibfnamefont {M.}~\bibnamefont
  {Schmidt}}\ and\ \bibinfo {author} {\bibfnamefont {H.}~\bibnamefont
  {Lipson}},\ }\bibfield  {title} {\bibinfo {title} {Distilling free-form
  natural laws from experimental data},\ }\href
  {https://doi.org/10.1126/science.1165893} {\bibfield  {journal} {\bibinfo
  {journal} {Science}\ }\textbf {\bibinfo {volume} {324}},\ \bibinfo {pages}
  {81} (\bibinfo {year} {2009})}\BibitemShut {NoStop}%
\bibitem [{\citenamefont {Brunton}\ \emph {et~al.}(2016)\citenamefont
  {Brunton}, \citenamefont {Proctor},\ and\ \citenamefont
  {Kutz}}]{brunton_discovering_2016}%
  \BibitemOpen
  \bibfield  {author} {\bibinfo {author} {\bibfnamefont {S.~L.}\ \bibnamefont
  {Brunton}}, \bibinfo {author} {\bibfnamefont {J.~L.}\ \bibnamefont
  {Proctor}},\ and\ \bibinfo {author} {\bibfnamefont {J.~N.}\ \bibnamefont
  {Kutz}},\ }\bibfield  {title} {\bibinfo {title} {Discovering governing
  equations from data by sparse identification of nonlinear dynamical
  systems},\ }\href {https://doi.org/10.1073/pnas.1517384113} {\bibfield
  {journal} {\bibinfo  {journal} {Proc. Natl. Acad. Sci.}\ }\textbf {\bibinfo
  {volume} {113}},\ \bibinfo {pages} {3932} (\bibinfo {year}
  {2016})}\BibitemShut {NoStop}%
\bibitem [{\citenamefont {Long}\ \emph {et~al.}(2018)\citenamefont {Long},
  \citenamefont {Lu}, \citenamefont {Ma},\ and\ \citenamefont
  {Dong}}]{longPDENetLearningPDEs2018}%
  \BibitemOpen
  \bibfield  {author} {\bibinfo {author} {\bibfnamefont {Z.}~\bibnamefont
  {Long}}, \bibinfo {author} {\bibfnamefont {Y.}~\bibnamefont {Lu}}, \bibinfo
  {author} {\bibfnamefont {X.}~\bibnamefont {Ma}},\ and\ \bibinfo {author}
  {\bibfnamefont {B.}~\bibnamefont {Dong}},\ }\bibfield  {title} {\bibinfo
  {title} {{PDE-Net: Learning PDEs from data}},\ }in\ \href@noop {} {\emph
  {\bibinfo {booktitle} {{35th International Conference on Machine Learning,
  ICML 2018}}}}\ (\bibinfo {year} {2018})\ pp.\ \bibinfo {pages}
  {5067--5078}\BibitemShut {NoStop}%
\bibitem [{\citenamefont {Raissi}\ \emph {et~al.}(2018)\citenamefont {Raissi},
  \citenamefont {Perdikaris},\ and\ \citenamefont
  {Karniadakis}}]{raissi_multistep_2018}%
  \BibitemOpen
  \bibfield  {author} {\bibinfo {author} {\bibfnamefont {M.}~\bibnamefont
  {Raissi}}, \bibinfo {author} {\bibfnamefont {P.}~\bibnamefont {Perdikaris}},\
  and\ \bibinfo {author} {\bibfnamefont {G.~E.}\ \bibnamefont {Karniadakis}},\
  }\bibfield  {title} {\bibinfo {title} {Multistep neural networks for
  data-driven discovery of nonlinear dynamical systems},\ }\href@noop {}
  {\bibfield  {journal} {\bibinfo  {journal} {ArXiv180101236}\ } (\bibinfo
  {year} {2018})},\ \Eprint {https://arxiv.org/abs/1801.01236}
  {arXiv:1801.01236} \BibitemShut {NoStop}%
\bibitem [{\citenamefont {Xie}\ \emph {et~al.}(2019{\natexlab{a}})\citenamefont
  {Xie}, \citenamefont {Zhang},\ and\ \citenamefont
  {Webster}}]{xieNonintrusiveInferenceReduced2018}%
  \BibitemOpen
  \bibfield  {author} {\bibinfo {author} {\bibfnamefont {X.}~\bibnamefont
  {Xie}}, \bibinfo {author} {\bibfnamefont {G.}~\bibnamefont {Zhang}},\ and\
  \bibinfo {author} {\bibfnamefont {C.~G.}\ \bibnamefont {Webster}},\
  }\bibfield  {title} {\bibinfo {title} {Non-intrusive inference reduced order
  model for fluids using linear multistep neural network},\ }\bibfield
  {journal} {\bibinfo  {journal} {Mathematics}\ }\textbf {\bibinfo {volume}
  {7}},\ \href {https://doi.org/10.3390/math7080757} {10.3390/math7080757}
  (\bibinfo {year} {2019}{\natexlab{a}}),\ \Eprint
  {https://arxiv.org/abs/1809.07820} {arXiv:1809.07820 [physics]} \BibitemShut
  {NoStop}%
\bibitem [{\citenamefont {Wan}\ \emph {et~al.}(2018)\citenamefont {Wan},
  \citenamefont {Vlachas}, \citenamefont {Koumoutsakos},\ and\ \citenamefont
  {Sapsis}}]{wanDataassistedReducedorderModeling2018}%
  \BibitemOpen
  \bibfield  {author} {\bibinfo {author} {\bibfnamefont {Z.~Y.}\ \bibnamefont
  {Wan}}, \bibinfo {author} {\bibfnamefont {P.~R.}\ \bibnamefont {Vlachas}},
  \bibinfo {author} {\bibfnamefont {P.}~\bibnamefont {Koumoutsakos}},\ and\
  \bibinfo {author} {\bibfnamefont {T.~P.}\ \bibnamefont {Sapsis}},\ }\bibfield
   {title} {\bibinfo {title} {Data-assisted reduced-order modeling of extreme
  events in complex dynamical systems},\ }\href
  {https://doi.org/10.1371/journal.pone.0197704} {\bibfield  {journal}
  {\bibinfo  {journal} {PLoS ONE}\ }\textbf {\bibinfo {volume} {13}},\ \bibinfo
  {pages} {e0197704} (\bibinfo {year} {2018})},\ \Eprint
  {https://arxiv.org/abs/1803.03365} {arXiv:1803.03365} \BibitemShut {NoStop}%
\bibitem [{\citenamefont {Pan}\ and\ \citenamefont
  {Duraisamy}(2018)}]{panDataDrivenDiscoveryClosure2018}%
  \BibitemOpen
  \bibfield  {author} {\bibinfo {author} {\bibfnamefont {S.}~\bibnamefont
  {Pan}}\ and\ \bibinfo {author} {\bibfnamefont {K.}~\bibnamefont
  {Duraisamy}},\ }\bibfield  {title} {\bibinfo {title} {Data-driven discovery
  of closure models},\ }\href {https://doi.org/10.1137/18M1177263} {\bibfield
  {journal} {\bibinfo  {journal} {SIAM J. Appl. Dyn. Syst.}\ }\textbf {\bibinfo
  {volume} {17}},\ \bibinfo {pages} {2381} (\bibinfo {year}
  {2018})}\BibitemShut {NoStop}%
\bibitem [{\citenamefont {Wang}\ \emph {et~al.}(2020)\citenamefont {Wang},
  \citenamefont {Ripamonti},\ and\ \citenamefont
  {Hesthaven}}]{wangRecurrentNeuralNetwork2020}%
  \BibitemOpen
  \bibfield  {author} {\bibinfo {author} {\bibfnamefont {Q.}~\bibnamefont
  {Wang}}, \bibinfo {author} {\bibfnamefont {N.}~\bibnamefont {Ripamonti}},\
  and\ \bibinfo {author} {\bibfnamefont {J.~S.}\ \bibnamefont {Hesthaven}},\
  }\bibfield  {title} {\bibinfo {title} {Recurrent neural network closure of
  parametric {{POD}}-{{Galerkin}} reduced-order models based on the
  {{Mori}}-{{Zwanzig}} formalism},\ }\href
  {https://doi.org/10.1016/j.jcp.2020.109402} {\bibfield  {journal} {\bibinfo
  {journal} {J. Comput. Phys.}\ }\textbf {\bibinfo {volume} {410}},\ \bibinfo
  {pages} {109402} (\bibinfo {year} {2020})}\BibitemShut {NoStop}%
\bibitem [{\citenamefont {Ma}\ \emph {et~al.}(2019)\citenamefont {Ma},
  \citenamefont {Wang},\ and\ \citenamefont {E}}]{maModelReductionMemory2019}%
  \BibitemOpen
  \bibfield  {author} {\bibinfo {author} {\bibfnamefont {C.}~\bibnamefont
  {Ma}}, \bibinfo {author} {\bibfnamefont {J.}~\bibnamefont {Wang}},\ and\
  \bibinfo {author} {\bibfnamefont {W.}~\bibnamefont {E}},\ }\bibfield  {title}
  {\bibinfo {title} {Model reduction with memory and the machine learning of
  dynamical systems},\ }\bibfield  {journal} {\bibinfo  {journal} {Commun.
  Comput. Phys.}\ }\textbf {\bibinfo {volume} {25}},\ \href
  {https://doi.org/10.4208/cicp.OA-2018-0269} {10.4208/cicp.OA-2018-0269}
  (\bibinfo {year} {2019})\BibitemShut {NoStop}%
\bibitem [{\citenamefont {Xie}\ \emph {et~al.}(2019{\natexlab{b}})\citenamefont
  {Xie}, \citenamefont {Li}, \citenamefont {Ma},\ and\ \citenamefont
  {Wang}}]{xie_modeling_2019}%
  \BibitemOpen
  \bibfield  {author} {\bibinfo {author} {\bibfnamefont {C.}~\bibnamefont
  {Xie}}, \bibinfo {author} {\bibfnamefont {K.}~\bibnamefont {Li}}, \bibinfo
  {author} {\bibfnamefont {C.}~\bibnamefont {Ma}},\ and\ \bibinfo {author}
  {\bibfnamefont {J.}~\bibnamefont {Wang}},\ }\bibfield  {title} {\bibinfo
  {title} {Modeling subgrid-scale force and divergence of heat flux of
  compressible isotropic turbulence by artificial neural network},\ }\bibfield
  {journal} {\bibinfo  {journal} {Phys. Rev. Fluids}\ }\textbf {\bibinfo
  {volume} {4}},\ \href {https://doi.org/10.1103/PhysRevFluids.4.104605}
  {10.1103/PhysRevFluids.4.104605} (\bibinfo {year}
  {2019}{\natexlab{b}})\BibitemShut {NoStop}%
\bibitem [{\citenamefont {Pathak}\ \emph {et~al.}(2018)\citenamefont {Pathak},
  \citenamefont {Hunt}, \citenamefont {Girvan}, \citenamefont {Lu},\ and\
  \citenamefont {Ott}}]{pathakModelFreePredictionLarge2018a}%
  \BibitemOpen
  \bibfield  {author} {\bibinfo {author} {\bibfnamefont {J.}~\bibnamefont
  {Pathak}}, \bibinfo {author} {\bibfnamefont {B.}~\bibnamefont {Hunt}},
  \bibinfo {author} {\bibfnamefont {M.}~\bibnamefont {Girvan}}, \bibinfo
  {author} {\bibfnamefont {Z.}~\bibnamefont {Lu}},\ and\ \bibinfo {author}
  {\bibfnamefont {E.}~\bibnamefont {Ott}},\ }\bibfield  {title} {\bibinfo
  {title} {Model-free prediction of large spatiotemporally chaotic systems from
  data: A reservoir computing approach},\ }\href
  {https://doi.org/10.1103/PhysRevLett.120.024102} {\bibfield  {journal}
  {\bibinfo  {journal} {Physical Review Letters}\ }\textbf {\bibinfo {volume}
  {120}},\ \bibinfo {pages} {024102} (\bibinfo {year} {2018})}\BibitemShut
  {NoStop}%
\bibitem [{\citenamefont {Vlachas}\ \emph {et~al.}(2020)\citenamefont
  {Vlachas}, \citenamefont {Pathak}, \citenamefont {Hunt}, \citenamefont
  {Sapsis}, \citenamefont {Girvan}, \citenamefont {Ott},\ and\ \citenamefont
  {Koumoutsakos}}]{vlachasBackpropagationAlgorithmsReservoir2020}%
  \BibitemOpen
  \bibfield  {author} {\bibinfo {author} {\bibfnamefont {P.~R.}\ \bibnamefont
  {Vlachas}}, \bibinfo {author} {\bibfnamefont {J.}~\bibnamefont {Pathak}},
  \bibinfo {author} {\bibfnamefont {B.~R.}\ \bibnamefont {Hunt}}, \bibinfo
  {author} {\bibfnamefont {T.~P.}\ \bibnamefont {Sapsis}}, \bibinfo {author}
  {\bibfnamefont {M.}~\bibnamefont {Girvan}}, \bibinfo {author} {\bibfnamefont
  {E.}~\bibnamefont {Ott}},\ and\ \bibinfo {author} {\bibfnamefont
  {P.}~\bibnamefont {Koumoutsakos}},\ }\bibfield  {title} {\bibinfo {title}
  {Backpropagation algorithms and reservoir computing in recurrent neural
  networks for the forecasting of complex spatiotemporal dynamics},\ }\href
  {https://doi.org/10.1016/j.neunet.2020.02.016} {\bibfield  {journal}
  {\bibinfo  {journal} {Neural Networks}\ }\textbf {\bibinfo {volume} {126}},\
  \bibinfo {pages} {191} (\bibinfo {year} {2020})}\BibitemShut {NoStop}%
\bibitem [{\citenamefont {Szunyogh}\ \emph {et~al.}(2020)\citenamefont
  {Szunyogh}, \citenamefont {Arcomano}, \citenamefont {Pathak}, \citenamefont
  {Wikner}, \citenamefont {Hunt},\ and\ \citenamefont
  {Ott}}]{szunyoghMachineLearningBasedGlobalAtmospheric2020}%
  \BibitemOpen
  \bibfield  {author} {\bibinfo {author} {\bibfnamefont {I.}~\bibnamefont
  {Szunyogh}}, \bibinfo {author} {\bibfnamefont {T.}~\bibnamefont {Arcomano}},
  \bibinfo {author} {\bibfnamefont {J.}~\bibnamefont {Pathak}}, \bibinfo
  {author} {\bibfnamefont {A.}~\bibnamefont {Wikner}}, \bibinfo {author}
  {\bibfnamefont {B.}~\bibnamefont {Hunt}},\ and\ \bibinfo {author}
  {\bibfnamefont {E.}~\bibnamefont {Ott}},\ }\href
  {https://doi.org/10.1002/essoar.10502527.1} {\emph {\bibinfo {title} {A
  Machine-Learning-Based Global Atmospheric Forecast Model}}},\ \bibinfo {type}
  {Preprint}\ (\bibinfo  {institution} {Atmospheric Sciences},\ \bibinfo {year}
  {2020})\BibitemShut {NoStop}%
\bibitem [{\citenamefont {Kolter}\ and\ \citenamefont
  {Manek}(2019)}]{kolterLearningStableDeep2019}%
  \BibitemOpen
  \bibfield  {author} {\bibinfo {author} {\bibfnamefont {J.~Z.}\ \bibnamefont
  {Kolter}}\ and\ \bibinfo {author} {\bibfnamefont {G.}~\bibnamefont {Manek}},\
  }\bibfield  {title} {\bibinfo {title} {Learning stable deep dynamics
  models},\ }\href@noop {} {\bibfield  {journal} {\bibinfo  {journal} {33rd
  Conf. Neural Inf. Process. Syst.}\ ,\ \bibinfo {pages} {9}} (\bibinfo {year}
  {2019})}\BibitemShut {NoStop}%
\bibitem [{\citenamefont {Giesl}\ \emph {et~al.}(2020)\citenamefont {Giesl},
  \citenamefont {Hamzi}, \citenamefont {Rasmussen},\ and\ \citenamefont
  {Webster}}]{gieslLyapunov2020}%
  \BibitemOpen
  \bibfield  {author} {\bibinfo {author} {\bibfnamefont {P.}~\bibnamefont
  {Giesl}}, \bibinfo {author} {\bibfnamefont {B.}~\bibnamefont {Hamzi}},
  \bibinfo {author} {\bibfnamefont {M.}~\bibnamefont {Rasmussen}},\ and\
  \bibinfo {author} {\bibfnamefont {K.}~\bibnamefont {Webster}},\ }\bibfield
  {title} {\bibinfo {title} {Approximation of {{Lyapunov}} functions from noisy
  data},\ }\href {https://doi.org/10.3934/jcd.2020003} {\bibfield  {journal}
  {\bibinfo  {journal} {J. Comput. Dyn.}\ }\textbf {\bibinfo {volume} {7}},\
  \bibinfo {pages} {57} (\bibinfo {year} {2020})}\BibitemShut {NoStop}%
\bibitem [{\citenamefont {Zhong}\ \emph
  {et~al.}(2020{\natexlab{a}})\citenamefont {Zhong}, \citenamefont {Dey},\ and\
  \citenamefont {Chakraborty}}]{zhongSymplecticODENetLearning2020}%
  \BibitemOpen
  \bibfield  {author} {\bibinfo {author} {\bibfnamefont {Y.~D.}\ \bibnamefont
  {Zhong}}, \bibinfo {author} {\bibfnamefont {B.}~\bibnamefont {Dey}},\ and\
  \bibinfo {author} {\bibfnamefont {A.}~\bibnamefont {Chakraborty}},\
  }\bibfield  {title} {\bibinfo {title} {Symplectic {{ODE}}-{{Net}}: Learning
  {Hamiltonian} dynamics with control},\ }\href@noop {} {\bibfield  {journal}
  {\bibinfo  {journal} {ICLR 2020, ArXiv190912077}\ } (\bibinfo {year}
  {2020}{\natexlab{a}})}\BibitemShut {NoStop}%
\bibitem [{\citenamefont {Jin}\ \emph {et~al.}(2020)\citenamefont {Jin},
  \citenamefont {Zhu}, \citenamefont {Karniadakis},\ and\ \citenamefont
  {Tang}}]{jin_symplectic_2020}%
  \BibitemOpen
  \bibfield  {author} {\bibinfo {author} {\bibfnamefont {P.}~\bibnamefont
  {Jin}}, \bibinfo {author} {\bibfnamefont {A.}~\bibnamefont {Zhu}}, \bibinfo
  {author} {\bibfnamefont {G.~E.}\ \bibnamefont {Karniadakis}},\ and\ \bibinfo
  {author} {\bibfnamefont {Y.}~\bibnamefont {Tang}},\ }\bibfield  {title}
  {\bibinfo {title} {Symplectic networks: {{Intrinsic}} structure-preserving
  networks for identifying {{Hamiltonian}} systems},\ }\href@noop {} {\bibfield
   {journal} {\bibinfo  {journal} {Neural Networks}\ }\textbf {\bibinfo
  {volume} {32}},\ \bibinfo {pages} {166} (\bibinfo {year} {2020})}\BibitemShut
  {NoStop}%
\bibitem [{\citenamefont {Zhong}\ \emph
  {et~al.}(2020{\natexlab{b}})\citenamefont {Zhong}, \citenamefont {Dey},\ and\
  \citenamefont {Chakraborty}}]{zhongDissipativeSymODENEncoding2020}%
  \BibitemOpen
  \bibfield  {author} {\bibinfo {author} {\bibfnamefont {Y.~D.}\ \bibnamefont
  {Zhong}}, \bibinfo {author} {\bibfnamefont {B.}~\bibnamefont {Dey}},\ and\
  \bibinfo {author} {\bibfnamefont {A.}~\bibnamefont {Chakraborty}},\
  }\bibfield  {title} {\bibinfo {title} {Dissipative {SymODEN}: Encoding
  {Hamiltonian} dynamics with dissipation and control into deep learning},\
  }\href@noop {} {\bibfield  {journal} {\bibinfo  {journal} {arXiv:2002.08860}\
  } (\bibinfo {year} {2020}{\natexlab{b}})}\BibitemShut {NoStop}%
\bibitem [{\citenamefont {Wikner}\ \emph {et~al.}(2020)\citenamefont {Wikner},
  \citenamefont {Pathak}, \citenamefont {Hunt}, \citenamefont {Girvan},
  \citenamefont {Arcomano}, \citenamefont {Szunyogh}, \citenamefont
  {Pomerance},\ and\ \citenamefont {Ott}}]{wiknerCombiningMachineLearning2020}%
  \BibitemOpen
  \bibfield  {author} {\bibinfo {author} {\bibfnamefont {A.}~\bibnamefont
  {Wikner}}, \bibinfo {author} {\bibfnamefont {J.}~\bibnamefont {Pathak}},
  \bibinfo {author} {\bibfnamefont {B.}~\bibnamefont {Hunt}}, \bibinfo {author}
  {\bibfnamefont {M.}~\bibnamefont {Girvan}}, \bibinfo {author} {\bibfnamefont
  {T.}~\bibnamefont {Arcomano}}, \bibinfo {author} {\bibfnamefont
  {I.}~\bibnamefont {Szunyogh}}, \bibinfo {author} {\bibfnamefont
  {A.}~\bibnamefont {Pomerance}},\ and\ \bibinfo {author} {\bibfnamefont
  {E.}~\bibnamefont {Ott}},\ }\bibfield  {title} {\bibinfo {title} {Combining
  machine learning with knowledge-based modeling for scalable forecasting and
  subgrid-scale closure of large, complex, spatiotemporal systems},\ }\href
  {https://doi.org/10.1063/5.0005541} {\bibfield  {journal} {\bibinfo
  {journal} {Chaos: An Interdisciplinary Journal of Nonlinear Science}\
  }\textbf {\bibinfo {volume} {30}},\ \bibinfo {pages} {053111} (\bibinfo
  {year} {2020})}\BibitemShut {NoStop}%
\bibitem [{\citenamefont
  {Onsager}(1931{\natexlab{a}})}]{onsagerReciprocalRelationsIrreversible1931}%
  \BibitemOpen
  \bibfield  {author} {\bibinfo {author} {\bibfnamefont {L.}~\bibnamefont
  {Onsager}},\ }\bibfield  {title} {\bibinfo {title} {Reciprocal relations in
  irreversible processes. {{I}}},\ }\href@noop {} {\bibfield  {journal}
  {\bibinfo  {journal} {Phys. Rev.}\ }\textbf {\bibinfo {volume} {37}},\
  \bibinfo {pages} {405} (\bibinfo {year} {1931}{\natexlab{a}})}\BibitemShut
  {NoStop}%
\bibitem [{\citenamefont
  {Onsager}(1931{\natexlab{b}})}]{onsagerReciprocalRelationsIrreversible1931a}%
  \BibitemOpen
  \bibfield  {author} {\bibinfo {author} {\bibfnamefont {L.}~\bibnamefont
  {Onsager}},\ }\bibfield  {title} {\bibinfo {title} {Reciprocal relations in
  irreversible processes. {{II}}},\ }\href@noop {} {\bibfield  {journal}
  {\bibinfo  {journal} {Phys. Rev.}\ }\textbf {\bibinfo {volume} {38}},\
  \bibinfo {pages} {2265} (\bibinfo {year} {1931}{\natexlab{b}})}\BibitemShut
  {NoStop}%
\bibitem [{\citenamefont {Qian}\ \emph {et~al.}(2006)\citenamefont {Qian},
  \citenamefont {Wang},\ and\ \citenamefont {Sheng}}]{qian_variational_2006}%
  \BibitemOpen
  \bibfield  {author} {\bibinfo {author} {\bibfnamefont {T.}~\bibnamefont
  {Qian}}, \bibinfo {author} {\bibfnamefont {X.-P.}\ \bibnamefont {Wang}},\
  and\ \bibinfo {author} {\bibfnamefont {P.}~\bibnamefont {Sheng}},\ }\bibfield
   {title} {\bibinfo {title} {A variational approach to moving contact line
  hydrodynamics},\ }\href@noop {} {\bibfield  {journal} {\bibinfo  {journal}
  {J. Fluid Mech.}\ }\textbf {\bibinfo {volume} {564}},\ \bibinfo {pages} {333}
  (\bibinfo {year} {2006})}\BibitemShut {NoStop}%
\bibitem [{\citenamefont {Doi}(2011)}]{doi_onsager_2011}%
  \BibitemOpen
  \bibfield  {author} {\bibinfo {author} {\bibfnamefont {M.}~\bibnamefont
  {Doi}},\ }\bibfield  {title} {\bibinfo {title} {Onsager's variational
  principle in soft matter},\ }\href
  {https://doi.org/10.1088/0953-8984/23/28/284118} {\bibfield  {journal}
  {\bibinfo  {journal} {J. Phys.: Condens. Matter}\ }\textbf {\bibinfo {volume}
  {23}},\ \bibinfo {pages} {284118} (\bibinfo {year} {2011})}\BibitemShut
  {NoStop}%
\bibitem [{\citenamefont {Yang}\ \emph {et~al.}(2016)\citenamefont {Yang},
  \citenamefont {Li}, \citenamefont {Forest},\ and\ \citenamefont
  {Wang}}]{yang_hydrodynamic_2016}%
  \BibitemOpen
  \bibfield  {author} {\bibinfo {author} {\bibfnamefont {X.}~\bibnamefont
  {Yang}}, \bibinfo {author} {\bibfnamefont {J.}~\bibnamefont {Li}}, \bibinfo
  {author} {\bibfnamefont {M.}~\bibnamefont {Forest}},\ and\ \bibinfo {author}
  {\bibfnamefont {Q.}~\bibnamefont {Wang}},\ }\bibfield  {title} {\bibinfo
  {title} {Hydrodynamic theories for flows of active liquid crystals and the
  generalized {{Onsager}} principle},\ }\href
  {https://doi.org/10.3390/e18060202} {\bibfield  {journal} {\bibinfo
  {journal} {Entropy}\ }\textbf {\bibinfo {volume} {18}},\ \bibinfo {pages}
  {202} (\bibinfo {year} {2016})}\BibitemShut {NoStop}%
\bibitem [{\citenamefont {Giga}\ \emph {et~al.}(2017)\citenamefont {Giga},
  \citenamefont {Kirshtein},\ and\ \citenamefont
  {Liu}}]{giga_variational_2017}%
  \BibitemOpen
  \bibfield  {author} {\bibinfo {author} {\bibfnamefont {M.-H.}\ \bibnamefont
  {Giga}}, \bibinfo {author} {\bibfnamefont {A.}~\bibnamefont {Kirshtein}},\
  and\ \bibinfo {author} {\bibfnamefont {C.}~\bibnamefont {Liu}},\ }\bibfield
  {title} {\bibinfo {title} {Variational modeling and complex fluids},\ }in\
  \href {https://doi.org/10.1007/978-3-319-10151-4_2-1} {\emph {\bibinfo
  {booktitle} {Handbook of {{Mathematical Analysis}} in {{Mechanics}} of
  {{Viscous Fluids}}}}},\ \bibinfo {editor} {edited by\ \bibinfo {editor}
  {\bibfnamefont {Y.}~\bibnamefont {Giga}}\ and\ \bibinfo {editor}
  {\bibfnamefont {A.}~\bibnamefont {Novotny}}}\ (\bibinfo  {publisher}
  {{Springer International Publishing}},\ \bibinfo {address} {{Cham}},\
  \bibinfo {year} {2017})\ pp.\ \bibinfo {pages} {1--41}\BibitemShut {NoStop}%
\bibitem [{\citenamefont {Jiang}\ \emph {et~al.}(2019)\citenamefont {Jiang},
  \citenamefont {Zhao}, \citenamefont {Qian}, \citenamefont {Srolovitz},\ and\
  \citenamefont {Bao}}]{jiang_application_2019}%
  \BibitemOpen
  \bibfield  {author} {\bibinfo {author} {\bibfnamefont {W.}~\bibnamefont
  {Jiang}}, \bibinfo {author} {\bibfnamefont {Q.}~\bibnamefont {Zhao}},
  \bibinfo {author} {\bibfnamefont {T.}~\bibnamefont {Qian}}, \bibinfo {author}
  {\bibfnamefont {D.~J.}\ \bibnamefont {Srolovitz}},\ and\ \bibinfo {author}
  {\bibfnamefont {W.}~\bibnamefont {Bao}},\ }\bibfield  {title} {\bibinfo
  {title} {Application of {{Onsager}}'s variational principle to the dynamics
  of a solid toroidal island on a substrate},\ }\href
  {https://doi.org/10.1016/j.actamat.2018.10.004} {\bibfield  {journal}
  {\bibinfo  {journal} {Acta Mater.}\ }\textbf {\bibinfo {volume} {163}},\
  \bibinfo {pages} {154} (\bibinfo {year} {2019})}\BibitemShut {NoStop}%
\bibitem [{\citenamefont {Doi}\ \emph {et~al.}(2019)\citenamefont {Doi},
  \citenamefont {Zhou}, \citenamefont {Di},\ and\ \citenamefont
  {Xu}}]{doi_application_2019a}%
  \BibitemOpen
  \bibfield  {author} {\bibinfo {author} {\bibfnamefont {M.}~\bibnamefont
  {Doi}}, \bibinfo {author} {\bibfnamefont {J.}~\bibnamefont {Zhou}}, \bibinfo
  {author} {\bibfnamefont {Y.}~\bibnamefont {Di}},\ and\ \bibinfo {author}
  {\bibfnamefont {X.}~\bibnamefont {Xu}},\ }\bibfield  {title} {\bibinfo
  {title} {Application of the {{Onsager}}-{{Machlup}} integral in solving
  dynamic equations in nonequilibrium systems},\ }\href
  {https://doi.org/10.1103/PhysRevE.99.063303} {\bibfield  {journal} {\bibinfo
  {journal} {Phys Rev E}\ }\textbf {\bibinfo {volume} {99}},\ \bibinfo {pages}
  {063303} (\bibinfo {year} {2019})}\BibitemShut {NoStop}%
\bibitem [{\citenamefont {Xu}\ and\ \citenamefont
  {Qian}(2019)}]{xu_generalized_2019}%
  \BibitemOpen
  \bibfield  {author} {\bibinfo {author} {\bibfnamefont {X.}~\bibnamefont
  {Xu}}\ and\ \bibinfo {author} {\bibfnamefont {T.}~\bibnamefont {Qian}},\
  }\bibfield  {title} {\bibinfo {title} {Generalized {{Lorentz}} reciprocal
  theorem in complex fluids and in non-isothermal systems},\ }\href
  {https://doi.org/10.1088/1361-648X/ab3898} {\bibfield  {journal} {\bibinfo
  {journal} {J. Phys.: Condens. Matter}\ }\textbf {\bibinfo {volume} {31}},\
  \bibinfo {pages} {475101} (\bibinfo {year} {2019})}\BibitemShut {NoStop}%
\bibitem [{\citenamefont {Beris}\ and\ \citenamefont
  {Edwards}(1994)}]{beris_thermodynamics_1994}%
  \BibitemOpen
  \bibfield  {author} {\bibinfo {author} {\bibfnamefont {A.~N.}\ \bibnamefont
  {Beris}}\ and\ \bibinfo {author} {\bibfnamefont {B.}~\bibnamefont
  {Edwards}},\ }\href@noop {} {\emph {\bibinfo {title} {Thermodynamics of
  {{Flowing Systems}}}}}\ (\bibinfo  {publisher} {{Oxford Science
  Publications}},\ \bibinfo {address} {{New York}},\ \bibinfo {year}
  {1994})\BibitemShut {NoStop}%
\bibitem [{\citenamefont {Doi}\ and\ \citenamefont
  {Edwards}(1986)}]{DoiEdwards1986}%
  \BibitemOpen
  \bibfield  {author} {\bibinfo {author} {\bibfnamefont {M.}~\bibnamefont
  {Doi}}\ and\ \bibinfo {author} {\bibfnamefont {S.~F.}\ \bibnamefont
  {Edwards}},\ }\href@noop {} {\emph {\bibinfo {title} {The Theory of Polymer
  Dynamics}}}\ (\bibinfo  {publisher} {{Oxford University Press, USA}},\
  \bibinfo {year} {1986})\BibitemShut {NoStop}%
\bibitem [{\citenamefont {De~Gennes}\ and\ \citenamefont
  {Prost}(1993)}]{deGennes1993}%
  \BibitemOpen
  \bibfield  {author} {\bibinfo {author} {\bibfnamefont {P.~G.}\ \bibnamefont
  {De~Gennes}}\ and\ \bibinfo {author} {\bibfnamefont {J.}~\bibnamefont
  {Prost}},\ }\href@noop {} {\emph {\bibinfo {title} {The Physicsof Liquid
  Crystals}}},\ \bibinfo {edition} {2nd}\ ed.\ (\bibinfo  {publisher}
  {{Clarendon Press}},\ \bibinfo {year} {1993})\BibitemShut {NoStop}%
\bibitem [{\citenamefont {Kr{\"o}ger}\ and\ \citenamefont
  {Ilg}(2007)}]{kroger_derivation_2007}%
  \BibitemOpen
  \bibfield  {author} {\bibinfo {author} {\bibfnamefont {M.}~\bibnamefont
  {Kr{\"o}ger}}\ and\ \bibinfo {author} {\bibfnamefont {P.}~\bibnamefont
  {Ilg}},\ }\bibfield  {title} {\bibinfo {title} {Derivation of
  {{Frank}}-{{Ericksen}} elastic coefficients for polydomain nematics from
  mean-field molecular theory for anisotropic particles},\ }\href
  {https://doi.org/10.1063/1.2743961} {\bibfield  {journal} {\bibinfo
  {journal} {J. Chem. Phy.}\ }\textbf {\bibinfo {volume} {127}},\ \bibinfo
  {pages} {034903} (\bibinfo {year} {2007})}\BibitemShut {NoStop}%
\bibitem [{\citenamefont {Yu}\ \emph {et~al.}(2010)\citenamefont {Yu},
  \citenamefont {Ji},\ and\ \citenamefont {Zhang}}]{yu_nonhomogeneous_2010}%
  \BibitemOpen
  \bibfield  {author} {\bibinfo {author} {\bibfnamefont {H.}~\bibnamefont
  {Yu}}, \bibinfo {author} {\bibfnamefont {G.}~\bibnamefont {Ji}},\ and\
  \bibinfo {author} {\bibfnamefont {P.}~\bibnamefont {Zhang}},\ }\bibfield
  {title} {\bibinfo {title} {A nonhomogeneous kinetic model of liquid crystal
  polymers and its thermodynamic closure approximation},\ }\href
  {https://doi.org/10.4208/cicp.2009.09.202} {\bibfield  {journal} {\bibinfo
  {journal} {Commun. Comput. Phy.}\ }\textbf {\bibinfo {volume} {7}},\ \bibinfo
  {pages} {383} (\bibinfo {year} {2010})}\BibitemShut {NoStop}%
\bibitem [{\citenamefont {E}(2011)}]{e_principles_2011}%
  \BibitemOpen
  \bibfield  {author} {\bibinfo {author} {\bibfnamefont {W.}~\bibnamefont
  {E}},\ }\href@noop {} {\emph {\bibinfo {title} {Principles of Multiscale
  Modeling}}}\ (\bibinfo  {publisher} {Cambridge University Press},\ \bibinfo
  {address} {New York},\ \bibinfo {year} {2011})\BibitemShut {NoStop}%
\bibitem [{\citenamefont {Han}\ \emph {et~al.}(2015)\citenamefont {Han},
  \citenamefont {Luo}, \citenamefont {Wang}, \citenamefont {Zhang},\ and\
  \citenamefont {Zhang}}]{han_microscopic_2015}%
  \BibitemOpen
  \bibfield  {author} {\bibinfo {author} {\bibfnamefont {J.}~\bibnamefont
  {Han}}, \bibinfo {author} {\bibfnamefont {Y.}~\bibnamefont {Luo}}, \bibinfo
  {author} {\bibfnamefont {W.}~\bibnamefont {Wang}}, \bibinfo {author}
  {\bibfnamefont {P.}~\bibnamefont {Zhang}},\ and\ \bibinfo {author}
  {\bibfnamefont {Z.}~\bibnamefont {Zhang}},\ }\bibfield  {title} {\bibinfo
  {title} {From microscopic theory to macroscopic theory: A systematic study on
  modeling for liquid crystals},\ }\href
  {https://doi.org/10.1007/s00205-014-0792-3} {\bibfield  {journal} {\bibinfo
  {journal} {Arch. Rational Mech. Anal.}\ }\textbf {\bibinfo {volume} {215}},\
  \bibinfo {pages} {741} (\bibinfo {year} {2015})}\BibitemShut {NoStop}%
\bibitem [{\citenamefont {Hinton}\ and\ \citenamefont
  {Salakhutdinov}(2006)}]{hintonReducingDimensionalityData2006}%
  \BibitemOpen
  \bibfield  {author} {\bibinfo {author} {\bibfnamefont {G.~E.}\ \bibnamefont
  {Hinton}}\ and\ \bibinfo {author} {\bibfnamefont {R.~R.}\ \bibnamefont
  {Salakhutdinov}},\ }\bibfield  {title} {\bibinfo {title} {Reducing the
  dimensionality of data with neural networks},\ }\href
  {https://doi.org/10.1126/science.1127647} {\bibfield  {journal} {\bibinfo
  {journal} {Science}\ }\textbf {\bibinfo {volume} {313}},\ \bibinfo {pages}
  {504} (\bibinfo {year} {2006})}\BibitemShut {NoStop}%
\bibitem [{\citenamefont {Rifai}\ \emph {et~al.}(2011)\citenamefont {Rifai},
  \citenamefont {Vincent}, \citenamefont {Muller}, \citenamefont {Glorot},\
  and\ \citenamefont {Bengio}}]{rifaiContractiveAutoEncodersExplicit2011}%
  \BibitemOpen
  \bibfield  {author} {\bibinfo {author} {\bibfnamefont {S.}~\bibnamefont
  {Rifai}}, \bibinfo {author} {\bibfnamefont {P.}~\bibnamefont {Vincent}},
  \bibinfo {author} {\bibfnamefont {X.}~\bibnamefont {Muller}}, \bibinfo
  {author} {\bibfnamefont {X.}~\bibnamefont {Glorot}},\ and\ \bibinfo {author}
  {\bibfnamefont {Y.}~\bibnamefont {Bengio}},\ }\bibfield  {title} {\bibinfo
  {title} {Contractive auto-encoders: Explicit invariance during feature
  extraction},\ }\href@noop {} {\bibfield  {journal} {\bibinfo  {journal}
  {Proceedings of the 28th International Conference on Machine Learning}\ ,\
  \bibinfo {pages} {8}} (\bibinfo {year} {2011})}\BibitemShut {NoStop}%
\bibitem [{\citenamefont {Lorenz}(1963)}]{lorenz_deterministic_1963}%
  \BibitemOpen
  \bibfield  {author} {\bibinfo {author} {\bibfnamefont {E.~N.}\ \bibnamefont
  {Lorenz}},\ }\bibfield  {title} {\bibinfo {title} {Deterministic nonperiodic
  flow},\ }\href@noop {} {\bibfield  {journal} {\bibinfo  {journal} {J. Atmos.
  Sci.}\ }\textbf {\bibinfo {volume} {20}},\ \bibinfo {pages} {130} (\bibinfo
  {year} {1963})}\BibitemShut {NoStop}%
\bibitem [{\citenamefont {Curry}\ \emph {et~al.}(1984)\citenamefont {Curry},
  \citenamefont {Herring}, \citenamefont {Loncaric},\ and\ \citenamefont
  {Orszag}}]{curryOrderDisorderTwo1984}%
  \BibitemOpen
  \bibfield  {author} {\bibinfo {author} {\bibfnamefont {J.~H.}\ \bibnamefont
  {Curry}}, \bibinfo {author} {\bibfnamefont {J.~R.}\ \bibnamefont {Herring}},
  \bibinfo {author} {\bibfnamefont {J.}~\bibnamefont {Loncaric}},\ and\
  \bibinfo {author} {\bibfnamefont {S.~A.}\ \bibnamefont {Orszag}},\ }\bibfield
   {title} {\bibinfo {title} {Order and disorder in two- and three-dimensional
  {{B{\'e}nard}} convection},\ }\href
  {https://doi.org/10.1017/S0022112084001968} {\bibfield  {journal} {\bibinfo
  {journal} {J. Fluid Mech.}\ }\textbf {\bibinfo {volume} {147}},\ \bibinfo
  {pages} {1} (\bibinfo {year} {1984})}\BibitemShut {NoStop}%
\bibitem [{\citenamefont {De~Groot}\ and\ \citenamefont
  {Mazur}(1962)}]{degroot_nonequilibrium_1962}%
  \BibitemOpen
  \bibfield  {author} {\bibinfo {author} {\bibfnamefont {S.~R.}\ \bibnamefont
  {De~Groot}}\ and\ \bibinfo {author} {\bibfnamefont {P.}~\bibnamefont
  {Mazur}},\ }\href@noop {} {\emph {\bibinfo {title} {Non-{{Equilibrium
  Thermodynamics}}}}}\ (\bibinfo  {publisher} {{Dover Publications}},\ \bibinfo
  {year} {1962})\BibitemShut {NoStop}%
\bibitem [{\citenamefont {Han}\ \emph {et~al.}(2019)\citenamefont {Han},
  \citenamefont {Ma}, \citenamefont {Ma},\ and\ \citenamefont
  {E}}]{han_uniformly_2019}%
  \BibitemOpen
  \bibfield  {author} {\bibinfo {author} {\bibfnamefont {J.}~\bibnamefont
  {Han}}, \bibinfo {author} {\bibfnamefont {C.}~\bibnamefont {Ma}}, \bibinfo
  {author} {\bibfnamefont {Z.}~\bibnamefont {Ma}},\ and\ \bibinfo {author}
  {\bibfnamefont {W.}~\bibnamefont {E}},\ }\bibfield  {title} {\bibinfo {title}
  {Uniformly accurate machine learning-based hydrodynamic models for kinetic
  equations},\ }\href {https://doi.org/10.1073/pnas.1909854116} {\bibfield
  {journal} {\bibinfo  {journal} {Proc. Natl. Acad. Sci.}\ }\textbf {\bibinfo
  {volume} {116}},\ \bibinfo {pages} {21983} (\bibinfo {year}
  {2019})}\BibitemShut {NoStop}%
\bibitem [{\citenamefont {Rayleigh}(1878)}]{rayleigh_instability_1878}%
  \BibitemOpen
  \bibfield  {author} {\bibinfo {author} {\bibfnamefont {L.}~\bibnamefont
  {Rayleigh}},\ }\bibfield  {title} {\bibinfo {title} {On the instability of
  jets},\ }\href {https://doi.org/10.1112/plms/s1-10.1.4} {\bibfield  {journal}
  {\bibinfo  {journal} {P. Lond. Math. Soc.}\ }\textbf {\bibinfo {volume}
  {s1-10}},\ \bibinfo {pages} {4} (\bibinfo {year} {1878})}\BibitemShut
  {NoStop}%
\bibitem [{\citenamefont {Green}(1954)}]{green_markoff_1954}%
  \BibitemOpen
  \bibfield  {author} {\bibinfo {author} {\bibfnamefont {M.~S.}\ \bibnamefont
  {Green}},\ }\bibfield  {title} {\bibinfo {title} {Markoff random processes
  and the statistical mechanics of time-dependent phenomena. {{II}}.
  {{Irreversible}} processes in fluids},\ }\href@noop {} {\bibfield  {journal}
  {\bibinfo  {journal} {J. Chem. Phys.}\ }\textbf {\bibinfo {volume} {22}},\
  \bibinfo {pages} {398} (\bibinfo {year} {1954})}\BibitemShut {NoStop}%
\bibitem [{\citenamefont {Kubo}(1957)}]{kubo_statistical-mechanical_1957-1}%
  \BibitemOpen
  \bibfield  {author} {\bibinfo {author} {\bibfnamefont {R.}~\bibnamefont
  {Kubo}},\ }\bibfield  {title} {\bibinfo {title} {Statistical-mechanical
  theory of irreversible processes. {{I}}. {{General}} theory and simple
  applications to magnetic and conduction problems},\ }\href@noop {} {\bibfield
   {journal} {\bibinfo  {journal} {J. Phys. Soc. Jpn.}\ }\textbf {\bibinfo
  {volume} {12}},\ \bibinfo {pages} {570} (\bibinfo {year} {1957})}\BibitemShut
  {NoStop}%
\bibitem [{\citenamefont {Evans}\ and\ \citenamefont
  {Searles}(2002)}]{evans_fluctuation_2002}%
  \BibitemOpen
  \bibfield  {author} {\bibinfo {author} {\bibfnamefont {D.~J.}\ \bibnamefont
  {Evans}}\ and\ \bibinfo {author} {\bibfnamefont {D.~J.}\ \bibnamefont
  {Searles}},\ }\bibfield  {title} {\bibinfo {title} {The fluctuation
  theorem},\ }\href@noop {} {\bibfield  {journal} {\bibinfo  {journal} {Adv.
  Phys.}\ }\textbf {\bibinfo {volume} {51}},\ \bibinfo {pages} {1529} (\bibinfo
  {year} {2002})}\BibitemShut {NoStop}%
\bibitem [{\citenamefont {Zhao}\ \emph {et~al.}(2018)\citenamefont {Zhao},
  \citenamefont {Yang}, \citenamefont {Gong}, \citenamefont {Zhao},
  \citenamefont {Yang}, \citenamefont {Li},\ and\ \citenamefont
  {Wang}}]{zhao_general_2018}%
  \BibitemOpen
  \bibfield  {author} {\bibinfo {author} {\bibfnamefont {J.}~\bibnamefont
  {Zhao}}, \bibinfo {author} {\bibfnamefont {X.}~\bibnamefont {Yang}}, \bibinfo
  {author} {\bibfnamefont {Y.}~\bibnamefont {Gong}}, \bibinfo {author}
  {\bibfnamefont {X.}~\bibnamefont {Zhao}}, \bibinfo {author} {\bibfnamefont
  {X.}~\bibnamefont {Yang}}, \bibinfo {author} {\bibfnamefont {J.}~\bibnamefont
  {Li}},\ and\ \bibinfo {author} {\bibfnamefont {Q.}~\bibnamefont {Wang}},\
  }\bibfield  {title} {\bibinfo {title} {A general strategy for numerical
  approximations of non-equilibrium models--part {{I}}: Thermodynamical
  systems},\ }\href@noop {} {\bibfield  {journal} {\bibinfo  {journal} {Int. J.
  Numer. Anal. Model.}\ }\textbf {\bibinfo {volume} {15}},\ \bibinfo {pages}
  {884} (\bibinfo {year} {2018})}\BibitemShut {NoStop}%
\bibitem [{\citenamefont {Ottinger}(2005)}]{ottinger_equilibrium_2005}%
  \BibitemOpen
  \bibfield  {author} {\bibinfo {author} {\bibfnamefont {H.~C.}\ \bibnamefont
  {Ottinger}},\ }\href@noop {} {\emph {\bibinfo {title} {Beyond {{Equilibrium
  Thermodynamics}}}}}\ (\bibinfo  {publisher} {{John Wiley \& Sons}},\ \bibinfo
  {year} {2005})\BibitemShut {NoStop}%
\bibitem [{\citenamefont {Hern{\'a}ndez}\ \emph {et~al.}(2021)\citenamefont
  {Hern{\'a}ndez}, \citenamefont {Badias}, \citenamefont {Gonzalez},
  \citenamefont {Chinesta},\ and\ \citenamefont
  {Cueto}}]{hernandez_structurepreserving_2021}%
  \BibitemOpen
  \bibfield  {author} {\bibinfo {author} {\bibfnamefont {Q.}~\bibnamefont
  {Hern{\'a}ndez}}, \bibinfo {author} {\bibfnamefont {A.}~\bibnamefont
  {Badias}}, \bibinfo {author} {\bibfnamefont {D.}~\bibnamefont {Gonzalez}},
  \bibinfo {author} {\bibfnamefont {F.}~\bibnamefont {Chinesta}},\ and\
  \bibinfo {author} {\bibfnamefont {E.}~\bibnamefont {Cueto}},\ }\bibfield
  {title} {\bibinfo {title} {Structure-preserving neural networks},\ }\href
  {https://doi.org/10.1016/j.jcp.2020.109950} {\bibfield  {journal} {\bibinfo
  {journal} {J. Comput. Phy.}\ }\textbf {\bibinfo {volume} {426}},\ \bibinfo
  {pages} {109950} (\bibinfo {year} {2021})}\BibitemShut {NoStop}%
\bibitem [{\citenamefont {Schmid}(2010)}]{schmid_dynamic_2010}%
  \BibitemOpen
  \bibfield  {author} {\bibinfo {author} {\bibfnamefont {P.~J.}\ \bibnamefont
  {Schmid}},\ }\bibfield  {title} {{\selectlanguage {english}\bibinfo {title}
  {Dynamic mode decomposition of numerical and experimental data}},\ }\href
  {https://doi.org/10.1017/S0022112010001217} {\bibfield  {journal} {\bibinfo
  {journal} {J. Fluid Mech.}\ }\textbf {\bibinfo {volume} {656}},\ \bibinfo
  {pages} {5} (\bibinfo {year} {2010})}\BibitemShut {NoStop}%
\bibitem [{\citenamefont {Takeishi}\ \emph {et~al.}(2017)\citenamefont
  {Takeishi}, \citenamefont {Kawahara},\ and\ \citenamefont
  {Yairi}}]{takeishi_learning_2017}%
  \BibitemOpen
  \bibfield  {author} {\bibinfo {author} {\bibfnamefont {N.}~\bibnamefont
  {Takeishi}}, \bibinfo {author} {\bibfnamefont {Y.}~\bibnamefont {Kawahara}},\
  and\ \bibinfo {author} {\bibfnamefont {T.}~\bibnamefont {Yairi}},\ }\bibfield
   {title} {\bibinfo {title} {Learning {{Koopman}} invariant subspaces for
  dynamic mode decomposition},\ }in\ \href@noop {} {\emph {\bibinfo {booktitle}
  {Advances in {{Neural Information Processing Systems}} 30}}}\ (\bibinfo
  {publisher} {{Curran Associates, Inc.}},\ \bibinfo {year} {2017})\ pp.\
  \bibinfo {pages} {1130--1140}\BibitemShut {NoStop}%
\bibitem [{\citenamefont {Paszke}\ \emph {et~al.}(2019)\citenamefont {Paszke},
  \citenamefont {Gross}, \citenamefont {Massa}, \citenamefont {Lerer},
  \citenamefont {Bradbury}, \citenamefont {Chanan}, \citenamefont {Killeen},
  \citenamefont {Lin}, \citenamefont {Gimelshein}, \citenamefont {Antiga},
  \citenamefont {Desmaison}, \citenamefont {Kopf}, \citenamefont {Yang},
  \citenamefont {DeVito}, \citenamefont {Raison}, \citenamefont {Tejani},
  \citenamefont {Chilamkurthy}, \citenamefont {Steiner}, \citenamefont {Fang},
  \citenamefont {Bai},\ and\ \citenamefont {Chintala}}]{paszke_pytorch_2019}%
  \BibitemOpen
  \bibfield  {author} {\bibinfo {author} {\bibfnamefont {A.}~\bibnamefont
  {Paszke}}, \bibinfo {author} {\bibfnamefont {S.}~\bibnamefont {Gross}},
  \bibinfo {author} {\bibfnamefont {F.}~\bibnamefont {Massa}}, \bibinfo
  {author} {\bibfnamefont {A.}~\bibnamefont {Lerer}}, \bibinfo {author}
  {\bibfnamefont {J.}~\bibnamefont {Bradbury}}, \bibinfo {author}
  {\bibfnamefont {G.}~\bibnamefont {Chanan}}, \bibinfo {author} {\bibfnamefont
  {T.}~\bibnamefont {Killeen}}, \bibinfo {author} {\bibfnamefont
  {Z.}~\bibnamefont {Lin}}, \bibinfo {author} {\bibfnamefont {N.}~\bibnamefont
  {Gimelshein}}, \bibinfo {author} {\bibfnamefont {L.}~\bibnamefont {Antiga}},
  \bibinfo {author} {\bibfnamefont {A.}~\bibnamefont {Desmaison}}, \bibinfo
  {author} {\bibfnamefont {A.}~\bibnamefont {Kopf}}, \bibinfo {author}
  {\bibfnamefont {E.}~\bibnamefont {Yang}}, \bibinfo {author} {\bibfnamefont
  {Z.}~\bibnamefont {DeVito}}, \bibinfo {author} {\bibfnamefont
  {M.}~\bibnamefont {Raison}}, \bibinfo {author} {\bibfnamefont
  {A.}~\bibnamefont {Tejani}}, \bibinfo {author} {\bibfnamefont
  {S.}~\bibnamefont {Chilamkurthy}}, \bibinfo {author} {\bibfnamefont
  {B.}~\bibnamefont {Steiner}}, \bibinfo {author} {\bibfnamefont
  {L.}~\bibnamefont {Fang}}, \bibinfo {author} {\bibfnamefont {J.}~\bibnamefont
  {Bai}},\ and\ \bibinfo {author} {\bibfnamefont {S.}~\bibnamefont
  {Chintala}},\ }\bibfield  {title} {\bibinfo {title} {{{PyTorch}}: An
  imperative style, high-performance deep learning library},\ }in\ \href@noop
  {} {\emph {\bibinfo {booktitle} {Advances in {{Neural Information Processing
  Systems}} 32}}}\ (\bibinfo  {publisher} {{Curran Associates, Inc.}},\
  \bibinfo {year} {2019})\ pp.\ \bibinfo {pages} {8026--8037}\BibitemShut
  {NoStop}%
\bibitem [{\citenamefont {He}\ \emph {et~al.}(2016)\citenamefont {He},
  \citenamefont {Zhang}, \citenamefont {Ren},\ and\ \citenamefont
  {Sun}}]{he2016deep}%
  \BibitemOpen
  \bibfield  {author} {\bibinfo {author} {\bibfnamefont {K.}~\bibnamefont
  {He}}, \bibinfo {author} {\bibfnamefont {X.}~\bibnamefont {Zhang}}, \bibinfo
  {author} {\bibfnamefont {S.}~\bibnamefont {Ren}},\ and\ \bibinfo {author}
  {\bibfnamefont {J.}~\bibnamefont {Sun}},\ }\bibfield  {title} {\bibinfo
  {title} {Deep residual learning for image recognition},\ }in\ \href@noop {}
  {\emph {\bibinfo {booktitle} {Proceedings of the IEEE conference on computer
  vision and pattern recognition}}}\ (\bibinfo {year} {2016})\ pp.\ \bibinfo
  {pages} {770--778}\BibitemShut {NoStop}%
\bibitem [{\citenamefont {Kingma}\ and\ \citenamefont
  {Ba}(2015)}]{kingmaAdamMethodStochastic2015}%
  \BibitemOpen
  \bibfield  {author} {\bibinfo {author} {\bibfnamefont {D.~P.}\ \bibnamefont
  {Kingma}}\ and\ \bibinfo {author} {\bibfnamefont {J.}~\bibnamefont {Ba}},\
  }\bibfield  {title} {\bibinfo {title} {Adam: A method for stochastic
  optimization},\ }in\ \href@noop {} {\emph {\bibinfo {booktitle} {3rd
  International Conference for Learning Representations, {arXiv}:1412.6980}}}\
  (\bibinfo {address} {San Diego},\ \bibinfo {year} {2015})\BibitemShut
  {NoStop}%
\bibitem [{\citenamefont {Reddi}\ \emph {et~al.}(2018)\citenamefont {Reddi},
  \citenamefont {Kale},\ and\ \citenamefont {Kumar}}]{reddi_convergence_2018}%
  \BibitemOpen
  \bibfield  {author} {\bibinfo {author} {\bibfnamefont {S.~J.}\ \bibnamefont
  {Reddi}}, \bibinfo {author} {\bibfnamefont {S.}~\bibnamefont {Kale}},\ and\
  \bibinfo {author} {\bibfnamefont {S.}~\bibnamefont {Kumar}},\ }\bibfield
  {title} {\bibinfo {title} {On the convergence of {Adam} and beyond},\ }in\
  \href@noop {} {\emph {\bibinfo {booktitle} {International {{Conference}} on
  {{Learning Representations}}}}}\ (\bibinfo {year} {2018})\BibitemShut
  {NoStop}%
\bibitem [{\citenamefont {Linot}\ and\ \citenamefont
  {Graham}(2020)}]{linot_deep_2020}%
  \BibitemOpen
  \bibfield  {author} {\bibinfo {author} {\bibfnamefont {A.~J.}\ \bibnamefont
  {Linot}}\ and\ \bibinfo {author} {\bibfnamefont {M.~D.}\ \bibnamefont
  {Graham}},\ }\bibfield  {title} {\bibinfo {title} {Deep learning to discover
  and predict dynamics on an inertial manifold},\ }\href@noop {} {\bibfield
  {journal} {\bibinfo  {journal} {Phys. Rev. E}\ }\textbf {\bibinfo {volume}
  {101}},\ \bibinfo {pages} {062209} (\bibinfo {year} {2020})}\BibitemShut
  {NoStop}%
\bibitem [{\citenamefont {Shu}\ and\ \citenamefont
  {Osher}(1988)}]{shu_efficient_1988}%
  \BibitemOpen
  \bibfield  {author} {\bibinfo {author} {\bibfnamefont {C.-W.}\ \bibnamefont
  {Shu}}\ and\ \bibinfo {author} {\bibfnamefont {S.}~\bibnamefont {Osher}},\
  }\bibfield  {title} {\bibinfo {title} {Efficient implementation of
  essentially non-oscillatory shock-capturing schemes},\ }\href
  {https://doi.org/10.1016/0021-9991(88)90177-5} {\bibfield  {journal}
  {\bibinfo  {journal} {J. Comput. Phys.}\ }\textbf {\bibinfo {volume} {77}},\
  \bibinfo {pages} {439} (\bibinfo {year} {1988})}\BibitemShut {NoStop}%
\bibitem [{\citenamefont {Li}\ \emph {et~al.}(2020)\citenamefont {Li},
  \citenamefont {Tang},\ and\ \citenamefont
  {Yu}}]{liBetterApproximationsHigh2020}%
  \BibitemOpen
  \bibfield  {author} {\bibinfo {author} {\bibfnamefont {B.}~\bibnamefont
  {Li}}, \bibinfo {author} {\bibfnamefont {S.}~\bibnamefont {Tang}},\ and\
  \bibinfo {author} {\bibfnamefont {H.}~\bibnamefont {Yu}},\ }\bibfield
  {title} {\bibinfo {title} {Better approximations of high dimensional smooth
  functions by deep neural networks with rectified power units},\ }\href
  {https://doi.org/10.4208/cicp.OA-2019-0168} {\bibfield  {journal} {\bibinfo
  {journal} {Commun. Comput. Phs.}\ }\textbf {\bibinfo {volume} {27}},\
  \bibinfo {pages} {379} (\bibinfo {year} {2020})}\BibitemShut {NoStop}%
\bibitem [{\citenamefont {Sparrow}(1982)}]{sparrow_lorenz_1982}%
  \BibitemOpen
  \bibfield  {author} {\bibinfo {author} {\bibfnamefont {C.}~\bibnamefont
  {Sparrow}},\ }\href {https://doi.org/10.1007/978-1-4612-5767-7} {\emph
  {\bibinfo {title} {The {{Lorenz Equations}}: {{Bifurcations}}, {{Chaos}}, and
  {{Strange Attractors}}}}}\ (\bibinfo  {publisher} {{Springer-Verlag}},\
  \bibinfo {address} {{New York}},\ \bibinfo {year} {1982})\BibitemShut
  {NoStop}%
\bibitem [{\citenamefont {Barrio}\ and\ \citenamefont
  {Serrano}(2007)}]{barrio_threeparametric_2007}%
  \BibitemOpen
  \bibfield  {author} {\bibinfo {author} {\bibfnamefont {R.}~\bibnamefont
  {Barrio}}\ and\ \bibinfo {author} {\bibfnamefont {S.}~\bibnamefont
  {Serrano}},\ }\bibfield  {title} {\bibinfo {title} {A three-parametric study
  of the {{Lorenz}} model},\ }\href
  {https://doi.org/10.1016/j.physd.2007.03.013} {\bibfield  {journal} {\bibinfo
   {journal} {Physica D}\ }\textbf {\bibinfo {volume} {229}},\ \bibinfo {pages}
  {43} (\bibinfo {year} {2007})}\BibitemShut {NoStop}%
\bibitem [{\citenamefont {Zhou}\ and\ \citenamefont
  {E}(2010)}]{zhou_study_2010}%
  \BibitemOpen
  \bibfield  {author} {\bibinfo {author} {\bibfnamefont {X.}~\bibnamefont
  {Zhou}}\ and\ \bibinfo {author} {\bibfnamefont {W.}~\bibnamefont {E}},\
  }\bibfield  {title} {\bibinfo {title} {Study of noise-induced transitions in
  the {{Lorenz}} system using the minimum action method},\ }\href@noop {}
  {\bibfield  {journal} {\bibinfo  {journal} {Commun. Math. Sci.}\ }\textbf
  {\bibinfo {volume} {8}},\ \bibinfo {pages} {341} (\bibinfo {year}
  {2010})}\BibitemShut {NoStop}%
\bibitem [{\citenamefont {Rosenstein}\ \emph {et~al.}(1993)\citenamefont
  {Rosenstein}, \citenamefont {Collins},\ and\ \citenamefont
  {De~Luca}}]{rosenstein_practical_1993}%
  \BibitemOpen
  \bibfield  {author} {\bibinfo {author} {\bibfnamefont {M.~T.}\ \bibnamefont
  {Rosenstein}}, \bibinfo {author} {\bibfnamefont {J.~J.}\ \bibnamefont
  {Collins}},\ and\ \bibinfo {author} {\bibfnamefont {C.~J.}\ \bibnamefont
  {De~Luca}},\ }\bibfield  {title} {\bibinfo {title} {A practical method for
  calculating largest {{Lyapunov}} exponents from small data sets},\ }\href
  {https://doi.org/10.1016/0167-2789(93)90009-P} {\bibfield  {journal}
  {\bibinfo  {journal} {Physica D}\ }\textbf {\bibinfo {volume} {65}},\
  \bibinfo {pages} {117} (\bibinfo {year} {1993})}\BibitemShut {NoStop}%
\bibitem [{\citenamefont {Curry}(1978)}]{curry_generalized_1978}%
  \BibitemOpen
  \bibfield  {author} {\bibinfo {author} {\bibfnamefont {J.~H.}\ \bibnamefont
  {Curry}},\ }\bibfield  {title} {\bibinfo {title} {A generalized {{Lorenz}}
  system},\ }\href {https://doi.org/10.1007/BF01612888} {\bibfield  {journal}
  {\bibinfo  {journal} {Commun. Math. Phys.}\ }\textbf {\bibinfo {volume}
  {60}},\ \bibinfo {pages} {193} (\bibinfo {year} {1978})}\BibitemShut
  {NoStop}%
\bibitem [{\citenamefont {Lumley}(1970)}]{lumley_stochastic_1970}%
  \BibitemOpen
  \bibfield  {author} {\bibinfo {author} {\bibfnamefont {J.}~\bibnamefont
  {Lumley}},\ }\href@noop {} {\emph {\bibinfo {title} {Stochastic {{Tools}} in
  {{Turbulence}}}}}\ (\bibinfo  {publisher} {{Academic Press}},\ \bibinfo
  {year} {1970})\BibitemShut {NoStop}%
\bibitem [{\citenamefont {Holmes}\ \emph {et~al.}(1996)\citenamefont {Holmes},
  \citenamefont {Lumley},\ and\ \citenamefont
  {Berkooz}}]{holmes_turbulence_1996}%
  \BibitemOpen
  \bibfield  {author} {\bibinfo {author} {\bibfnamefont {P.}~\bibnamefont
  {Holmes}}, \bibinfo {author} {\bibfnamefont {J.}~\bibnamefont {Lumley}},\
  and\ \bibinfo {author} {\bibfnamefont {G.}~\bibnamefont {Berkooz}},\
  }\href@noop {} {\emph {\bibinfo {title} {Turbulence, {{Coherent Structures}},
  {{Dynamical Systems}} and {{Symmetry}}}}}\ (\bibinfo  {publisher} {{Cambridge
  Univ Press}},\ \bibinfo {year} {1996})\BibitemShut {NoStop}%
\bibitem [{\citenamefont {Rowley}\ \emph {et~al.}(2009)\citenamefont {Rowley},
  \citenamefont {Mezi{\'c}}, \citenamefont {Bagheri}, \citenamefont
  {Schlatter},\ and\ \citenamefont {Henningson}}]{rowley_spectral_2009}%
  \BibitemOpen
  \bibfield  {author} {\bibinfo {author} {\bibfnamefont {C.~W.}\ \bibnamefont
  {Rowley}}, \bibinfo {author} {\bibfnamefont {I.}~\bibnamefont {Mezi{\'c}}},
  \bibinfo {author} {\bibfnamefont {S.}~\bibnamefont {Bagheri}}, \bibinfo
  {author} {\bibfnamefont {P.}~\bibnamefont {Schlatter}},\ and\ \bibinfo
  {author} {\bibfnamefont {D.~S.}\ \bibnamefont {Henningson}},\ }\bibfield
  {title} {{\selectlanguage {english}\bibinfo {title} {Spectral analysis of
  nonlinear flows}},\ }\href {https://doi.org/10.1017/S0022112009992059}
  {\bibfield  {journal} {\bibinfo  {journal} {J. Fluid Mech.}\ }\textbf
  {\bibinfo {volume} {641}},\ \bibinfo {pages} {115} (\bibinfo {year}
  {2009})}\BibitemShut {NoStop}%
\bibitem [{\citenamefont {Champion}\ \emph {et~al.}(2019)\citenamefont
  {Champion}, \citenamefont {Lusch}, \citenamefont {Kutz},\ and\ \citenamefont
  {Brunton}}]{championDatadrivenDiscoveryCoordinates2019}%
  \BibitemOpen
  \bibfield  {author} {\bibinfo {author} {\bibfnamefont {K.}~\bibnamefont
  {Champion}}, \bibinfo {author} {\bibfnamefont {B.}~\bibnamefont {Lusch}},
  \bibinfo {author} {\bibfnamefont {J.~N.}\ \bibnamefont {Kutz}},\ and\
  \bibinfo {author} {\bibfnamefont {S.~L.}\ \bibnamefont {Brunton}},\
  }\bibfield  {title} {\bibinfo {title} {Data-driven discovery of coordinates
  and governing equations},\ }\href {https://doi.org/10.1073/pnas.1906995116}
  {\bibfield  {journal} {\bibinfo  {journal} {Proc. Natl. Acad. Sci.}\ }\textbf
  {\bibinfo {volume} {116}},\ \bibinfo {pages} {22445} (\bibinfo {year}
  {2019})}\BibitemShut {NoStop}%
\bibitem [{\citenamefont {Chorin}\ \emph {et~al.}(2000)\citenamefont {Chorin},
  \citenamefont {Hald},\ and\ \citenamefont
  {Kupferman}}]{chorinOptimalPredictionMoriZwanzig2000}%
  \BibitemOpen
  \bibfield  {author} {\bibinfo {author} {\bibfnamefont {A.~J.}\ \bibnamefont
  {Chorin}}, \bibinfo {author} {\bibfnamefont {O.~H.}\ \bibnamefont {Hald}},\
  and\ \bibinfo {author} {\bibfnamefont {R.}~\bibnamefont {Kupferman}},\
  }\bibfield  {title} {\bibinfo {title} {Optimal prediction and the
  {{Mori}}-{{Zwanzig}} representation of irreversible processes},\ }\href
  {https://doi.org/10.1073/pnas.97.7.2968} {\bibfield  {journal} {\bibinfo
  {journal} {Proc. Natl. Acad. Sci.}\ }\textbf {\bibinfo {volume} {97}},\
  \bibinfo {pages} {2968} (\bibinfo {year} {2000})}\BibitemShut {NoStop}%
\bibitem [{\citenamefont {Zhang}\ \emph {et~al.}(2019)\citenamefont {Zhang},
  \citenamefont {Lin}, \citenamefont {Wang}, \citenamefont {Car},\ and\
  \citenamefont {E}}]{zhang2019active}%
  \BibitemOpen
  \bibfield  {author} {\bibinfo {author} {\bibfnamefont {L.}~\bibnamefont
  {Zhang}}, \bibinfo {author} {\bibfnamefont {D.~Y.}\ \bibnamefont {Lin}},
  \bibinfo {author} {\bibfnamefont {H.}~\bibnamefont {Wang}}, \bibinfo {author}
  {\bibfnamefont {R.}~\bibnamefont {Car}},\ and\ \bibinfo {author}
  {\bibfnamefont {W.}~\bibnamefont {E}},\ }\bibfield  {title} {\bibinfo {title}
  {Active learning of uniformly accurate interatomic potentials for materials
  simulation},\ }\href@noop {} {\bibfield  {journal} {\bibinfo  {journal}
  {Physical Review Materials}\ }\textbf {\bibinfo {volume} {3}},\ \bibinfo
  {pages} {023804} (\bibinfo {year} {2019})}\BibitemShut {NoStop}%
\end{thebibliography}%

\end{document}